\documentclass[a4paper]{amsart}
\usepackage{etex}
\usepackage{stackrel}
\usepackage[utf8]{inputenc}
\usepackage[english]{babel}
\usepackage{amsmath}
\usepackage{amssymb}
\usepackage{amsopn}                 
\usepackage{amsbsy}
\usepackage{amsfonts}
\usepackage{MnSymbol}
\usepackage{amsthm}
\usepackage{newlfont}
\usepackage{amscd}
\usepackage{mathtools}              
\usepackage{mathrsfs}
\usepackage{relsize}
\usepackage{tensor}
\usepackage{enumerate}
\usepackage{enumitem}
\usepackage{comment}
\usepackage{float}
\usepackage[abbreviations]{foreign} 
\usepackage{eucal}
\usepackage{xy}
\xyoption{arrow}
\xyoption{matrix}
\SilentMatrices
\xyoption{tips}
\SelectTips{cm}{10}
\xyoption{2cell}
\UseAllTwocells
\usepackage{tikz}
\usetikzlibrary{cd}
\usetikzlibrary{calc}
\usetikzlibrary{math}
\usetikzlibrary{arrows}
\usetikzlibrary{fit}
\usetikzlibrary{positioning}

\usetikzlibrary{decorations.pathmorphing, arrows.meta}

\tikzset{Rightarrow/.style={double equal sign distance,>={Implies},->},
triple/.style={-,preaction={draw,Rightarrow}},
quadruple/.style={preaction={draw,Rightarrow,shorten >=0pt},shorten >=1pt,-,double,double
distance=0.2pt}}
\usepackage{xcolor}
\definecolor{darkblue}{rgb}{0,0,0.3}
\usepackage[colorlinks=true, citecolor=darkblue, filecolor=darkblue, linkcolor=darkblue, urlcolor=darkblue]{hyperref}
\newtheorem{thm}{Theorem}[subsection]

\newtheorem{cor}[thm]{Corollary}

\newtheorem{lemma}[thm]{Lemma}
\newtheorem{prop}[thm]{Proposition}

\theoremstyle{definition}
\newtheorem{define}[thm]{Definition}
\newtheorem{notate}[thm]{Notation}

\theoremstyle{remark}
\newtheorem{rem}[thm]{Remark}
\newtheorem{example}[thm]{Example}

\newtheorem{warning}[thm]{Warning}

\newenvironment{nscenter}%
{\parskip=0pt\par\nopagebreak\centering}
{\par\noindent\ignorespacesafterend}        
\newcommand\nbd\nobreakdash
\newcommand{\ndef}{\emph}

\def\lrar{\longrightarrow}
\newcommand{\ZZ}{\mathbb{Z}}

\newcommand{\bS}{\mathbf{S}}        
\newcommand{\op}{\mathrm{op}}

\newcommand{\Cat}{{\mathcal{C}\mspace{-2.mu}\mathit{at}}}
\newcommand{\nCat}[1]{{#1}\hbox{\protect\nbd-}\kern1pt\Cat}	    
\newcommand{\inn}{\mathrm{inn}}
\newcommand{\out}{\mathrm{out}}
\newcommand{\s}{\mathcal{S}\mspace{-2.mu}\text{et}_{\Delta}}
\newcommand{\sca}{\mathrm{sc}}
\newcommand{\Ss}{\mathcal{S}\mspace{-2.mu}\text{et}_{\Delta}^{\,\mathrm{sc}}}
\newcommand{\Sms}{\mathcal{S}\mspace{-2.mu}\text{et}_{\Delta}^{+, \mathrm{sc}}}

\newcommand{\Nsc}{\mathrm{N}^{\mathrm{sc}}}			
\newcommand{\rN}{\mathrm{N}}                        

\newcommand{\Catoo}{\Cat_{\infty}}
\newcommand{\BiCat}{\mathrm{BiCat}_{\infty}}

\newcommand{\Fun}{\mathrm{Fun}}

\newcommand\pdfoo{\texorpdfstring{$\infty$}{oo}}

\DeclareMathOperator{\Hom}{Hom}

\DeclareMathOperator{\Map}{Map}

\DeclareMathOperator{\Id}{id}
\DeclareMathOperator{\co}{co}
\DeclareMathOperator{\coop}{coop}
\newcommand{\gr}{\mathrm{gr}}               

\newcommand{\opgr}{\mathrm{opgr}}
\newcommand{\LMap}{\Fun^{\opgr}}
\newcommand{\RMap}{\Fun^{\gr}}

\DeclareMathOperator{\Un}{Un}

\DeclareMathOperator{\thi}{th}
\def\alp{{\alpha}}
\def\bet{{\beta}}

\def\eps{{\varepsilon}}

\def\sig{{\sigma}}

\def\Del{{\Delta}}

\def\Lam{{\Lambda}}

\def\hrar{\hookrightarrow}

\def\ovl{\overline}

\def\wtl{\widetilde}

\newcommand{\tr}[2]{\mathchoice
	{#1\raise -1.8pt\vbox{\hbox{$\kern -.8pt/\mathsmaller{#2} $}}}
	{#1\raise -1.8pt\vbox{\hbox{$\kern -.8pt/#2$}}\kern .8pt}
	{#1\raise -1.8pt\vbox{\hbox{$\scriptstyle\kern -.8pt /#2$}}}
	{#1\raise -1.8pt\vbox{\hbox{$\scriptscriptstyle\kern -.8pt /#2$}}}}

\newcommand{\trbis}[2]{\mathchoice
	{#1\raise -1.8pt\vbox{\hbox{$\kern -.8pt\mathsmaller{/#2} $}}}
	{#1\raise -1.8pt\vbox{\hbox{$\kern -.8pt\mathsmaller{/#2}$}}\kern .8pt}
	{#1\raise -1.8pt\vbox{\hbox{$\scriptstyle\kern -.8pt /#2$}}}
	{#1\raise -1.8pt\vbox{\hbox{$\scriptscriptstyle\kern -.8pt /#2$}}}}
\newcommand{\overslice}[2]{\mathchoice
	{#1\raise -1.8pt\vbox{\hbox{$\kern -.8pt\mathsmaller{#2/} $}}}
	{#1\raise -1.8pt\vbox{\hbox{$\kern -.8pt\mathsmaller{#2/}$}}\kern .8pt}
	{#1\raise -1.8pt\vbox{\hbox{$\scriptstyle\kern -.8pt #2/$}}}
	{#1\raise -1.8pt\vbox{\hbox{$\scriptscriptstyle\kern -.8pt #2/$}}}}

	\makeatletter

\def\labelstylecode@triangle#1{%
	\pgfkeys@split@path%
	\edef\label@key{/triangle/label/\pgfkeyscurrentname}%
	\edef\style@key{\pgfkeyscurrentkey/.@val}%
	\def\temp@a{#1}%
	\def\temp@b{\pgfkeysnovalue}%
	\ifx\temp@a\temp@b
	\pgfkeysgetvalue{\label@key}\temp@a
	\ifx\temp@a\temp@b\else
	\pgfkeysalso{commutative diagrams/.cd, \style@key}%
	\fi
	\else
	\pgfkeys{\style@key/.code = \pgfkeysalso{#1}}%
	\fi}
\def\arrowstylecode@triangle#1{%
	\edef\style@key{\pgfkeyscurrentkey/.@val}%
	\def\temp@a{#1}%
	\def\temp@b{\pgfkeysnovalue}%
	\ifx\temp@a\temp@b
	\pgfkeysalso{commutative diagrams/.cd, \style@key}%
	\else
	\pgfkeys{\style@key/.code = \pgfkeysalso{#1}}%
	\fi}

\pgfkeys{
	/triangle/label/.cd,
	0/.initial = {$\bullet$}, 1/.initial = {$\bullet$}, 2/.initial = {$\bullet$},
	01/.initial, 12/.initial, 02/.initial,
	012/.initial,
	/triangle/labelstyle/.cd,
	01/.@val/.initial=\pgfkeysalso{swap}, 12/.@val/.initial=\pgfkeysalso{swap},
	02/.@val/.initial,
	012/.@val/.initial=\pgfkeysalso{near start},
	01/.code=\labelstylecode@triangle{#1}, 02/.code=\labelstylecode@triangle{#1},
	12/.code=\labelstylecode@triangle{#1}, 012/.code=\labelstylecode@triangle{#1},
	/triangle/arrowstyle/.cd,
	01/.@val/.initial, 12/.@val/.initial,
	02/.@val/.initial,
	012/.@val/.initial=\pgfkeysalso{Rightarrow},
	01/.code=\arrowstylecode@triangle{#1}, 02/.code=\arrowstylecode@triangle{#1},
	12/.code=\arrowstylecode@triangle{#1}, 012/.code=\arrowstylecode@triangle{#1},
}

\def\tr@abc{%
	\draw [/triangle/arrowstyle/012] (90:0.20) --
	node [/triangle/labelstyle/012] {
		\pgfkeysvalueof{/triangle/label/012}} (270:0.10);
}

\def\tr@#1#2{
	\begin{scope}[shift=#2, commutative diagrams/every diagram]
		
		\node (n{#1}0) at (150:1) {
			\pgfkeysvalueof{/triangle/label/0}};
		\node (n{#1}1) at (270:0.6) {
			\pgfkeysvalueof{/triangle/label/1}};
		\node (n{#1}2) at (30:1) {
			\pgfkeysvalueof{/triangle/label/2}};			
		
		\node (s#1) at (0,0) [circle, inner sep = 0pt,
		fit = (n{#1}0.center)(n{#1}1.center)(n{#1}2.center)] {};
		
		\begin{scope}[commutative diagrams/.cd, every arrow, every label]
			\ifcase #1
			\def\list{0/1, 1/2, 0/2}\or
			\def\list{0/1, 1/2, 0/2}\else
			\def\list{}\fi
			
			\foreach \s / \e in \list {
				\draw [/triangle/arrowstyle/\s\e] (n{#1}\s) --
				node [/triangle/labelstyle/\s\e] {
					\pgfkeysvalueof{/triangle/label/\s\e}} (n{#1}\e);
			}
			
			\ifcase #1
			\tr@abc\or
			\tr@abc
			\else\fi
			
		\end{scope}
	\end{scope}
}

\def\triangle#1{
	\pgfkeys{#1}
	\tr@{0}{(0:0)}
}

\makeatother

\makeatletter

\def\labelstylecode@square#1{%
	\pgfkeys@split@path%
	\edef\label@key{/square/label/\pgfkeyscurrentname}%
	\edef\style@key{\pgfkeyscurrentkey/.@val}%
	\def\temp@a{#1}%
	\def\temp@b{\pgfkeysnovalue}%
	\ifx\temp@a\temp@b
	\pgfkeysgetvalue{\label@key}\temp@a
	\ifx\temp@a\temp@b\else
	\pgfkeysalso{commutative diagrams/.cd, \style@key}%
	\fi
	\else
	\pgfkeys{\style@key/.code = \pgfkeysalso{#1}}%
	\fi}
\def\arrowstylecode@square#1{%
	\edef\style@key{\pgfkeyscurrentkey/.@val}%
	\def\temp@a{#1}%
	\def\temp@b{\pgfkeysnovalue}%
	\ifx\temp@a\temp@b
	\pgfkeysalso{commutative diagrams/.cd, \style@key}%
	\else
	\pgfkeys{\style@key/.code = \pgfkeysalso{#1}}%
	\fi}

\pgfkeys{
	/square/label/.cd,
	0/.initial = {$\bullet$}, 1/.initial = {$\bullet$}, 2/.initial = {$\bullet$},
	3/.initial = {$\bullet$},
	01/.initial, 12/.initial, 23/.initial,
	02/.initial, 03/.initial, 13/.initial,
	012/.initial, 013/.initial, 023/.initial, 123/.initial,
	0123/.initial,
	/square/labelstyle/.cd,
	01/.@val/.initial=\pgfkeysalso{swap},
	12/.@val/.initial=\pgfkeysalso{swap},
	23/.@val/.initial=\pgfkeysalso{swap},
	03/.@val/.initial,
	02/.@val/.initial=\pgfkeysalso{description},
	13/.@val/.initial=\pgfkeysalso{description},
	012/.@val/.initial=\pgfkeysalso{above left = -1pt and -1pt},
	013/.@val/.initial=\pgfkeysalso{swap, right},
	023/.@val/.initial=\pgfkeysalso{left},
	123/.@val/.initial=\pgfkeysalso{swap, above right = -1pt and -1pt},
	0123/.@val/.initial,
	01/.code=\labelstylecode@square{#1}, 02/.code=\labelstylecode@square{#1},
	03/.code=\labelstylecode@square{#1}, 12/.code=\labelstylecode@square{#1},
	13/.code=\labelstylecode@square{#1}, 23/.code=\labelstylecode@square{#1},
	012/.code=\labelstylecode@square{#1}, 013/.code=\labelstylecode@square{#1},
	023/.code=\labelstylecode@square{#1}, 123/.code=\labelstylecode@square{#1},
	0123/.code=\labelstylecode@square{#1},
	/square/arrowstyle/.cd,
	01/.@val/.initial, 12/.@val/.initial, 23/.@val/.initial,
	02/.@val/.initial, 03/.@val/.initial, 13/.@val/.initial,
	012/.@val/.initial=\pgfkeysalso{Rightarrow},
	013/.@val/.initial=\pgfkeysalso{Rightarrow},
	023/.@val/.initial=\pgfkeysalso{Rightarrow},
	123/.@val/.initial=\pgfkeysalso{Rightarrow},
	0123/.@val/.initial=\pgfkeysalso{Rightarrow},
	01/.code=\arrowstylecode@square{#1}, 02/.code=\arrowstylecode@square{#1},
	03/.code=\arrowstylecode@square{#1}, 12/.code=\arrowstylecode@square{#1},
	13/.code=\arrowstylecode@square{#1}, 23/.code=\arrowstylecode@square{#1},
	012/.code=\arrowstylecode@square{#1}, 013/.code=\arrowstylecode@square{#1},
	023/.code=\arrowstylecode@square{#1}, 123/.code=\arrowstylecode@square{#1},
	0123/.code=\arrowstylecode@square{#1}
}

\def\sq@abc{%
	\draw [/square/arrowstyle/012] (235:0.25) --
	node [/square/labelstyle/012] {
		\pgfkeysvalueof{/square/label/012}} (235:0.6);
}
\def\sq@bcd{%
	\draw [/square/arrowstyle/123] (-54:0.25) --
	node [/square/labelstyle/123] {
		\pgfkeysvalueof{/square/label/123}} (-54:0.6);
}
\def\sq@acd{%
	\draw [/square/arrowstyle/023] (55:0.55) --
	node [/square/labelstyle/023] {
		\pgfkeysvalueof{/square/label/023}} (15:0.45);
}
\def\sq@abd{%
	\draw [/square/arrowstyle/013] (125:0.55) --
	node [/square/labelstyle/013] {
		\pgfkeysvalueof{/square/label/013}} (165:0.45);
}

\def\sq@#1#2{
	\begin{scope}[shift=#2, commutative diagrams/every diagram]
		
		\foreach \i in {0,1,2,3} {
			\tikzmath{\a = 135 + (90 * \i);}
			\node (n{#1}\i) at (\a:1) {
				\pgfkeysvalueof{/square/label/\i}};
		}
		
		\node (s#1) at (0,0) [circle, inner sep = 0pt,
		fit = (n{#1}0.center)(n{#1}1.center)(n{#1}2.center)
		(n{#1}3.center)] {};
		
		\begin{scope}[commutative diagrams/.cd, every arrow, every label]
			\ifcase #1
			\def\list{0/1, 1/2, 2/3, 0/2, 0/3}\or
			\def\list{0/1, 1/2, 2/3, 1/3, 0/3}\else
			\def\list{}\fi
			
			\foreach \s / \e in \list {
				\draw [/square/arrowstyle/\s\e] (n{#1}\s) --
				node [/square/labelstyle/\s\e] {
					\pgfkeysvalueof{/square/label/\s\e}} (n{#1}\e);
			}
			
			\ifcase #1
			\sq@abc\sq@acd\or
			\sq@abd\sq@bcd
			\else\fi
			
		\end{scope}
	\end{scope}
}

\def\square#1{
	\pgfkeys{#1}
	\sq@{0}{(180:2)}\sq@{1}{(0:2)}
	
	\begin{scope}[commutative diagrams/.cd, every arrow, every label]
		\draw[->] [shorten >=10pt, shorten <=10pt, /square/arrowstyle/0123] (s0) --
		node [/square/labelstyle/0123] {%
			\pgfkeysvalueof{/square/label/0123}} (s1);
		
	\end{scope}
}

\makeatother

\makeatletter

\def\labelstylecode@pent#1{%
	\pgfkeys@split@path%
	\edef\label@key{/pentagon/label/\pgfkeyscurrentname}%
	\edef\style@key{\pgfkeyscurrentkey/.@val}%
	\def\temp@a{#1}%
	\def\temp@b{\pgfkeysnovalue}%
	\ifx\temp@a\temp@b
	\pgfkeysgetvalue{\label@key}\temp@a
	\ifx\temp@a\temp@b\else
	\pgfkeysalso{commutative diagrams/.cd, \style@key}%
	\fi
	\else
	\pgfkeys{\style@key/.code = \pgfkeysalso{#1}}%
	\fi}
\def\arrowstylecode@pent#1{%
	\edef\style@key{\pgfkeyscurrentkey/.@val}%
	\def\temp@a{#1}%
	\def\temp@b{\pgfkeysnovalue}%
	\ifx\temp@a\temp@b
	\pgfkeysalso{commutative diagrams/.cd, \style@key}%
	\else
	\pgfkeys{\style@key/.code = \pgfkeysalso{#1}}%
	\fi}

\pgfkeys{
	/pentagon/label/.cd,
	0/.initial = {$\bullet$}, 1/.initial = {$\bullet$}, 2/.initial = {$\bullet$},
	3/.initial = {$\bullet$}, 4/.initial = {$\bullet$},
	01/.initial, 12/.initial, 23/.initial, 34/.initial, 04/.initial,
	02/.initial, 03/.initial, 13/.initial, 14/.initial, 24/.initial,
	012/.initial, 013/.initial, 014/.initial, 023/.initial, 024/.initial,
	034/.initial, 123/.initial, 124/.initial, 134/.initial, 234/.initial,
	0123/.initial, 0124/.initial, 0134/.initial, 0234/.initial, 1234/.initial,
	01234/.initial,
	/pentagon/labelstyle/.cd,
	01/.@val/.initial, 12/.@val/.initial, 23/.@val/.initial, 34/.@val/.initial,
	04/.@val/.initial=\pgfkeysalso{swap},
	02/.@val/.initial=\pgfkeysalso{description},
	03/.@val/.initial=\pgfkeysalso{description},
	13/.@val/.initial=\pgfkeysalso{description},
	14/.@val/.initial=\pgfkeysalso{description},
	24/.@val/.initial=\pgfkeysalso{description},
	012/.@val/.initial=\pgfkeysalso{swap, above},
	013/.@val/.initial=\pgfkeysalso{below = 1pt},
	014/.@val/.initial=\pgfkeysalso{below},
	023/.@val/.initial=\pgfkeysalso{swap, above = 1pt},
	024/.@val/.initial=\pgfkeysalso{below left = -1pt and -1pt},
	034/.@val/.initial=\pgfkeysalso{swap, right},
	123/.@val/.initial=\pgfkeysalso{below left = -1pt and -1pt},
	124/.@val/.initial=\pgfkeysalso{swap, right},
	134/.@val/.initial=\pgfkeysalso{left},
	234/.@val/.initial=\pgfkeysalso{swap, below right = -1pt and -1pt},
	0123/.@val/.initial,
	0124/.@val/.initial=\pgfkeysalso{swap},
	0134/.@val/.initial,
	0234/.@val/.initial=\pgfkeysalso{swap},
	1234/.@val/.initial,
	01234/.@val/.initial,
	01/.code=\labelstylecode@pent{#1}, 02/.code=\labelstylecode@pent{#1},
	03/.code=\labelstylecode@pent{#1}, 04/.code=\labelstylecode@pent{#1},
	12/.code=\labelstylecode@pent{#1}, 13/.code=\labelstylecode@pent{#1},
	14/.code=\labelstylecode@pent{#1}, 23/.code=\labelstylecode@pent{#1},
	24/.code=\labelstylecode@pent{#1}, 34/.code=\labelstylecode@pent{#1},
	012/.code=\labelstylecode@pent{#1}, 013/.code=\labelstylecode@pent{#1},
	014/.code=\labelstylecode@pent{#1}, 023/.code=\labelstylecode@pent{#1},
	024/.code=\labelstylecode@pent{#1}, 034/.code=\labelstylecode@pent{#1},
	123/.code=\labelstylecode@pent{#1}, 124/.code=\labelstylecode@pent{#1},
	134/.code=\labelstylecode@pent{#1}, 234/.code=\labelstylecode@pent{#1},
	0123/.code=\labelstylecode@pent{#1}, 0124/.code=\labelstylecode@pent{#1},
	0134/.code=\labelstylecode@pent{#1}, 0234/.code=\labelstylecode@pent{#1},
	1234/.code=\labelstylecode@pent{#1}, 01234/.code=\labelstylecode@pent{#1},
	/pentagon/arrowstyle/.cd,
	01/.@val/.initial, 12/.@val/.initial, 23/.@val/.initial, 34/.@val/.initial,
	04/.@val/.initial, 02/.@val/.initial, 03/.@val/.initial, 13/.@val/.initial,
	14/.@val/.initial, 24/.@val/.initial,
	012/.@val/.initial=\pgfkeysalso{Rightarrow},
	013/.@val/.initial=\pgfkeysalso{Rightarrow},
	014/.@val/.initial=\pgfkeysalso{Rightarrow},
	023/.@val/.initial=\pgfkeysalso{Rightarrow},
	024/.@val/.initial=\pgfkeysalso{Rightarrow},
	034/.@val/.initial=\pgfkeysalso{Rightarrow},
	123/.@val/.initial=\pgfkeysalso{Rightarrow},
	124/.@val/.initial=\pgfkeysalso{Rightarrow},
	134/.@val/.initial=\pgfkeysalso{Rightarrow},
	234/.@val/.initial=\pgfkeysalso{Rightarrow},
	0123/.@val/.initial=\pgfkeysalso{triple},
	0124/.@val/.initial=\pgfkeysalso{triple},
	0134/.@val/.initial=\pgfkeysalso{triple},
	0234/.@val/.initial=\pgfkeysalso{triple},
	1234/.@val/.initial=\pgfkeysalso{triple},
	01234/.@val/.initial=\pgfkeysalso{quadruple},
	01/.code=\arrowstylecode@pent{#1}, 02/.code=\arrowstylecode@pent{#1},
	03/.code=\arrowstylecode@pent{#1}, 04/.code=\arrowstylecode@pent{#1},
	12/.code=\arrowstylecode@pent{#1}, 13/.code=\arrowstylecode@pent{#1},
	14/.code=\arrowstylecode@pent{#1}, 23/.code=\arrowstylecode@pent{#1},
	24/.code=\arrowstylecode@pent{#1}, 34/.code=\arrowstylecode@pent{#1},
	012/.code=\arrowstylecode@pent{#1}, 013/.code=\arrowstylecode@pent{#1},
	014/.code=\arrowstylecode@pent{#1}, 023/.code=\arrowstylecode@pent{#1},
	024/.code=\arrowstylecode@pent{#1}, 034/.code=\arrowstylecode@pent{#1},
	123/.code=\arrowstylecode@pent{#1}, 124/.code=\arrowstylecode@pent{#1},
	134/.code=\arrowstylecode@pent{#1}, 234/.code=\arrowstylecode@pent{#1},
	0123/.code=\arrowstylecode@pent{#1}, 0124/.code=\arrowstylecode@pent{#1},
	0134/.code=\arrowstylecode@pent{#1}, 0234/.code=\arrowstylecode@pent{#1},
	1234/.code=\arrowstylecode@pent{#1}, 01234/.code=\arrowstylecode@pent{#1}
}

\def\pent@abc{%
	\draw [/pentagon/arrowstyle/012] (198:0.45) --
	node [/pentagon/labelstyle/012] {
		\pgfkeysvalueof{/pentagon/label/012}} (198:0.8);
}
\def\pent@bcd{%
	\draw [/pentagon/arrowstyle/123] (126:0.45) --
	node [/pentagon/labelstyle/123] {
		\pgfkeysvalueof{/pentagon/label/123}} (126:0.8);
}
\def\pent@cde{%
	\draw [/pentagon/arrowstyle/234] (54:0.45) --
	node [/pentagon/labelstyle/234] {
		\pgfkeysvalueof{/pentagon/label/234}} (54:0.8);
}
\def\pent@ade{%
	\draw [/pentagon/arrowstyle/034] (-40:0.6) --
	node [/pentagon/labelstyle/034] {
		\pgfkeysvalueof{/pentagon/label/034}} (-5:0.5);
}
\def\pent@abe{
	\draw [/pentagon/arrowstyle/014] (-70:0.55) --
	node [/pentagon/labelstyle/014] {
		\pgfkeysvalueof{/pentagon/label/014}} (-110:0.55);
}
\def\pent@acd{%
	\draw [/pentagon/arrowstyle/023] (55:0.3) --
	node [/pentagon/labelstyle/023] {
		\pgfkeysvalueof{/pentagon/label/023}} (125:0.3);
}
\def\pent@bde{%
	\draw [/pentagon/arrowstyle/134] (-5:0.4) --
	node [/pentagon/labelstyle/134] {
		\pgfkeysvalueof{/pentagon/label/134}} (35:0.5);
}
\def\pent@ace{%
	\draw [/pentagon/arrowstyle/024] (-45:0.45) --
	node [/pentagon/labelstyle/024] {
		\pgfkeysvalueof{/pentagon/label/024}} (-45:0.1);
}
\def\pent@abd{%
	\draw [/pentagon/arrowstyle/013] (-90:0.22) --
	node [/pentagon/labelstyle/013] {
		\pgfkeysvalueof{/pentagon/label/013}} (-150:0.46);
}
\def\pent@bce{%
	\draw [/pentagon/arrowstyle/124] (188:0.4) --
	node [/pentagon/labelstyle/124] {
		\pgfkeysvalueof{/pentagon/label/124}} (150:0.55);
}

\def\pent@#1#2{
	\begin{scope}[shift=#2, commutative diagrams/every diagram]
		
		\foreach \i in {0,1,2,3,4} {
			\tikzmath{\a = 270 - (72 * \i);}
			\node (n{#1}\i) at (\a:1) {
				\pgfkeysvalueof{/pentagon/label/\i}};
		}
		
		\node (p#1) at (0,0) [circle, inner sep = 0pt,
		fit = (n{#1}0.center)(n{#1}1.center)(n{#1}2.center)
		(n{#1}3.center)(n{#1}4.center)] {};
		
		\begin{scope}[commutative diagrams/.cd, every arrow, every label]
			\ifcase #1
			\def\list{0/1, 1/2, 2/3, 3/4, 0/4, 0/2, 0/3}\or
			\def\list{0/1, 1/2, 2/3, 3/4, 0/4, 1/3, 1/4}\or
			\def\list{0/1, 1/2, 2/3, 3/4, 0/4, 0/2, 2/4}\or
			\def\list{0/1, 1/2, 2/3, 3/4, 0/4, 0/3, 1/3}\or
			\def\list{0/1, 1/2, 2/3, 3/4, 0/4, 1/4, 2/4}\else
			\def\list{}\fi
			
			\foreach \s / \e in \list {
				\draw [/pentagon/arrowstyle/\s\e] (n{#1}\s) --
				node [/pentagon/labelstyle/\s\e] {
					\pgfkeysvalueof{/pentagon/label/\s\e}} (n{#1}\e);
			}
			
			\ifcase #1
			\pent@abc\pent@acd\pent@ade\or
			\pent@bcd\pent@bde\pent@abe\or
			\pent@cde\pent@ace\pent@abc\or
			\pent@ade\pent@abd\pent@bcd\or
			\pent@abe\pent@bce\pent@cde
			\else\fi
			
		\end{scope}
	\end{scope}
}

\makeatother

\newcommand{\A}{\mathcal{A}}
\newcommand{\B}{\mathcal{B}}
\newcommand{\C}{\mathcal{C}}
\newcommand{\D}{\mathcal{D}}
\newcommand{\E}{\mathcal{E}}

\newcommand{\X}{\mathcal{X}}

\newcommand{\Car}{\mathrm{Car}}
\newcommand{\coCar}{\mathrm{coCar}}

\renewcommand{\plus}[1]{\mathop{\amalg}\limits_{#1}}

\newcommand{\fCs}{\mathfrak{C}^{\mathrm{sc}}}
\newcommand{\Maptr}{\Hom^{\triangleright}}
\newcommand{\uMaptr}{\ovl{\Hom}^{\triangleright}}

\def\lrar{\rightarrow}

\usepackage[font=small,format=hang,labelfont={sf,bf}]{caption}

\newcommand{\minisimeq}{\scalebox{0.8}{\ensuremath\simeq}}

\newcommand{\id}{\Id}

\def\lrar{\longrightarrow}

\newcommand{\dom}{\mathrm{d}}

\makeatletter

\renewcommand{\tocsection}[3]{%
\indentlabel{\@ifnotempty{#2}{\parbox[b]{3ex}{\bfseries\ignorespaces#1 #2}}}\bfseries#3} 

\renewcommand{\tocsubsection}[3]{%
\indentlabel{\@ifnotempty{#2}{\hspace{1.6em}\parbox[b]{5ex}{\ignorespaces#1 #2}}}#3}

\makeatother

\setlist[itemize]{leftmargin=*}
\setlist[enumerate]{leftmargin=*}

\title{Cartesian fibrations of \((\infty,2)\)-categories}
\author{Andrea Gagna}
\address{Institute of Mathematics, Czech Academy of Sciences\\ \v{Z}itn\'a 25 \\115 67   Praha 1\\ Czech Republic}
\email{gagna@math.cas.cz}
\urladdr{https://sites.google.com/view/andreagagna/home}
\author{Yonatan Harpaz}
\address{Institut Galilée\\ Université Paris 13\\ 99 avenue Jean-Baptiste Clément\\ 93430 Villeta-neuse\\ France}
\email{harpaz@math.univ-paris13.fr}
\urladdr{https://www.math.univ-paris13.fr/~harpaz}
\author{Edoardo Lanari}
\address{Institute of Mathematics, Czech Academy of Sciences\\ \v{Z}itn\'a 25 \\115 67   Praha 1\\ Czech Republic}
\email{edoardo.lanari.el@gmail.com}
\urladdr{https://edolana.github.io/}

\begin{document}
%
%
\begin{abstract}
In this article we introduce four variance flavours of (co)cartesian fibrations of \(\infty\)-bicategories
with \(\infty\)-bicategorical fibres, in the framework of scaled simplicial sets. 
Given a map \(p\colon \E \rightarrow\B\) of \(\infty\)-bicategories, we define \(p\)\nbd-(co)\-car\-tesian arrows and inner/outer triangles by means of lifting properties against~\(p\), leading to a notion of 2-inner/outer (co)cartesian fibrations as those maps with enough (co)cartesian lifts for arrows and enough inner/outer lifts for triangles, together with a compatibility property
with respect to whiskerings in the outer case. By doing so, we also recover in particular the case of \(\infty\)-bicategories fibred in \(\infty\)-categories studied in previous work. We also prove that equivalences of such fibrations can be tested fibrewise.
As a motivating example, we show that the domain projection \(\mathrm{d}\colon\RMap(\Delta^1,\C)\rightarrow \C\) is a prototypical example of a 2-outer cartesian fibration, where \(\RMap(X,Y)\) denotes the \(\infty\)-bicategory of functors, lax natural transformations and modifications.
We then define 2-inner and 2-outer flavours of (co)cartesian fibrations 
of categories enriched in \(\infty\)-categories, 
and we show that a fibration \(p\colon \E \rightarrow\B\) of such enriched categories is a (co)cartesian 2-inner/outer fibration if and only if the corresponding map \(\rN^{\sca}(p)\colon \rN^{\sca}\E \rightarrow \rN^{\sca}\B\) is a fibration of this type between \(\infty\)-bicategories.
\end{abstract}

\maketitle

\tableofcontents
\section*{Introduction}
The present paper is part of an ongoing series of works on the theory of \((\infty,2)\)-categories. We will generally refer to these as \(\infty\)-bicategories, and identify them with scaled simplicial sets satisfying suitable extension properties. Our goal here is to set up the fundamentals of a theory of (co)cartesian fibrations \(\E \to \B\) of \(\infty\)-bicategories, whose fibres encode a family of \(\infty\)-bicategories \(\E_b\) depending functorially on \(b \in \B\). When \(\E\) and \(\B\) are \(\infty\)-categories the corresponding notion of (co)cartesian fibration was set up by Lurie in~\cite{HTT}. Generalizing classical work of Grothendieck, Lurie showed that such fibrations over a base \(\infty\)-category \(\B\) are in complete correspondence with functors \(\B \to \Catoo\) in the cocartesian case, and contravariant such functors in the cartesian case. This Grothendieck--Lurie correspondence plays a key role in higher category theory, as it permits to handle highly coherent pieces of structure, such as functors, in a relatively accessible manner. 

When coming to consider the \(\infty\)-bicategorical counterpart, an immediate difference presents itself:  
here we have not just two different variances, but four, depending on whether or not the functorial dependence on 2-morphisms is covariant or contravariant. This additional axis of symmetry is already visible when the base is an \(\infty\)-bicategory and the fibres \(\E_b\) are \(\infty\)-categories, since \(\Catoo\) itself has a non-trivial \(\infty\)-bicategorical structure. A thorough treatment of this case was taken up in the previous work~\cite{GagnaHarpazLanariLaxLimits}, where we have used the term inner (co)cartesian fibrations to indicate those cases where the dependence on 2-morphisms matches the one on 1-morphisms, and outer (co)cartesian fibration for those where these two dependencies have opposite variance. The division into these two types is also visible on the level of mapping \(\infty\)-categories: when both variances match (the inner case) the induced functor on mapping \(\infty\)-bicategories is a right fibration, and when they don't (the outer case) the induced functor on mapping \(\infty\)-categories is a left fibration. One of the main results of~\cite{GagnaHarpazLanariLaxLimits} is an \(\infty\)-bicategorical Grothendieck--Lurie correspondence for all the four variances. The proof crucially relies on the work of Lurie~\cite{LurieGoodwillie} in the inner cocartesian case, and the straightening-unstraightening Quillen equivalence he proved in that setting.

Our goal in the present paper is to extend these notions of fibrations to the setting where both the base \(\B\) and the fibres \(\E_b\) are \(\infty\)-bicategories. This requires defining, in addition to a notion of (co)cartesian 1-morphisms, a suitable notion of (co)cartesian 2-morphism. Working in the setting of scaled simplicial sets, such 2-morphisms are determined by triangles, though this correspondence is not perfect (the same 2-morphism can be encoded by different triangles corresponding to different factorizations of the target 1-morphism), and one needs to be a bit more careful about how to define this. To match with the notation for fibrations, we use the terms \emph{inner} and \emph{outer} to describe triangles which roughly correspond to cartesian and cocartesian 2-morphisms. For somewhat technical reasons one needs to separate those further into two types, which we call left and right inner/outer triangles. This distinction turns out to behave slightly differently in the inner and outer cases, an issue due to the built-in a-symmetry of triangles, which encode a 2-morphism together with a factorization of its target 1-morphism, but not of its domain.

In the eventual continuation of the present work we plan to establish a complete Grothendieck--Lurie correspondence for such fibrations over a fixed base \(\B\), showing that they encode the four possible variances of functors \(\B \to \BiCat\). Though we will not arrive to this goal yet here, we do lay the foundations and prove the basic properties we expect to need, based on our current work in progress in that direction. In particular, we prove that equivalences between such fibrations are detected on fibres, and construct a certain universal example in the form of the domain projection \(\RMap(\Del^1,\B) \to \B\), where \(\RMap(\Del^1,\B)\) is the \(\infty\)-bicategory of arrows in \(\B\) and lax squares between them. We then establish a comparison between the notions of inner/outer (co)cartesian fibrations and an analogous one in the setting of categories enriched in marked simplicial sets. Such a comparison played an important role in~\cite{GagnaHarpazLanariLaxLimits} by implying that the \((\ZZ/2 \times \ZZ/2)\)-symmetry of the theory of \(\infty\)-bicategories switched between the four types of fibrations. This allowed one to reduce the Grothendieck--Lurie correspondence in~\cite{GagnaHarpazLanariLaxLimits} to the inner cocartesian case, and we expect a similar role to be played in a future generalization of that correspondence to fibrations. 

After finishing the work on the present paper, we became aware of indepedent work of Abell\'{a}n Garc\'{i}a and Stern~\cite{GarciaStern2Fib}, 
which investigates the outer variant of these fibrations 
using a model category structure on marked biscaled 
simplicial sets. 
By contrast, our approach here consistently covers all four variances, and also interacts well with the framework we set up in~\cite{GagnaHarpazLanariLaxLimits}, a property we expect would be useful in a future proof of a complete \(\infty\)-bicategorical Grothendieck--Lurie correspondence for inner/outer (co)cartesian fibrations.  
Finally, let us also note that a Grothendieck--Lurie type correspondence for outer cartesian fibrations is sketched in the Appendix of~\cite{GaitsgoryRozenblyumStudy}, though the argument relies in various parts on unproven statements. The extensive treament of derived algebraic geomtery developed in the body of~\cite{GaitsgoryRozenblyumStudy} based on that appendix makes for a powerful motivation for obtaining the \(\infty\)-bicategorical Grothendieck--Lurie correspondence rigorously. We also plan to pursue the study of inner/outer \emph{locally} (co)cartesian fibrations, which encode
lax/oplax 2-functors from an \(\infty\)-bicategory \(\B\) to \(\BiCat\),
as defined in~\cite[\S 3]{GagnaHarpazLanariGrayLaxFunctors}. This will enable one to compare that last definition with the one of~\cite{GaitsgoryRozenblyumStudy}, which should also allow for the comparison of the two notions of Gray products, thus establishing many key unproven statements involving Gray products made in the appendix of loc.~cit. The results of the present work are fundamental preliminaries for all these applications.

The present paper is organized as follows. We begin with a preliminary section where we give pointers to the necessary definitions and results concerning marked and scaled simplicial sets, and recall the framework of inner/outer (co)cartesian fibrations set up in~\cite{GagnaHarpazLanariLaxLimits}, which we rename 1-inner/outer (co)cartesian fibration in order to better distinguish between them and the type of fibrations introduced in the present paper.
Next, we introduce the main concepts of this paper, namely those of left/right inner/outer triangles and the corresponding notions of fibrations. We establish fundamental results concerning closure properties of inner/outer triangles, as well as stability properties and a fibre-wise criterion to test equivalences of fibrations. Finally, we prove Theorem \ref{thm:domain proj}, which concerns the domain projection on an \(\infty\)-bicategory and provides a prototypical example of a 2-outer cartesian fibration. 

In Section~\ref{sec:enriched-fibrations} we briefly recall the theory of (co)cartesian fibrations for ordinary 2-categories, as developed by Buckley in \cite{BuckleyFibred}, as a motivation for its enhancement to the setting of categories enriched in marked simplicial sets. We then provide an extension of the results of \cite[\S 3]{GagnaHarpazLanariLaxLimits}, by showing that enriched 2-inner/outer (co)cartesian fibrations identify via the coherent nerve functor with 
the respective notion of fibrations between \(\infty\)-bicategories (see theorem~\ref{N detects fib cart} for the precise statement).

\section*{Vistas and applications}
A theory of fibrations provides the backbone to define symmetric monoidal \((\infty,2)\)-categories, as a suitable class of inner cocartesian fibrations over \(\mathrm{Fin}_\ast\). We expect symmetric monoidal \(\infty\)-bicategories
to play a fundamental role as their 1-dimensional counterpart, providing extra expressive power thanks to the 2-dimensional structure.
For instance, they can be used to encode the relevant dualities of Ind-coherent sheaves of (derived) schemes
of finite type, as in~\cite[Ch.~9]{GaitsgoryRozenblyumStudy}.

Furthermore, inner/outer (locally) (co)cartesian fibrations are used as a tool in~\cite[Ch.~11 \& 12]{GaitsgoryRozenblyumStudy} to establish an adjoint theorem that plays a crucial role in extending the quasi-coherent sheaves functor from derived affine schemes to derived prestacks,
see~\cite[\S\S 3.1.1 \& 3.2.1]{GaitsgoryRozenblyumStudy}.

Finally, the theory of fibrations can be used to define a simplicial version of the notion
of (relative)~\((\infty, 2)\)-operad, see~\cite{BataninMonoidal,BataninEckmannHilton}.
Possible applications in this direction come from symplectic geometry, and in particular
from the functoriality of the Fukaya category, see for instance~\cite{BottmanShacharAinfty2categories}.

\section*{Acknowledgements}
The first and third author gratefully acknowledge the support of Praemium Academiae of M.~Markl and RVO:67985840. The third author is also grateful to Rune Haugseng and Nick Rozenblyum for the fruitful conversations during his stay at MSRI. 

\section{Preliminaries}

In this section we establish notation and recall some preliminary definitions and results that will be used in the present paper. 

\subsection{Marked-simplicial sets and enriched categories}

	Recall that a marked simplicial set is a pair \((X, E_X)\) where \(X\) is a simplicial set and \(E_X\) is a collection of edges in \(X\), called the \ndef{marked} edges, containing all degenerate edges. A map of marked simplicial sets \(f\colon (X,E_X)\rightarrow (Y,E_Y)\) is a map of simplicial sets \(f\colon X \rightarrow Y\) satisfying \(f(E_X)\subseteq E_Y\). When denoting an explicit marked simplicial set we will often omit the reference to the degenerate edges. For example, we will write \((\Del^n,\Del^{\{0,1\}})\) for the marked simplicial set whose underlying simplicial set is the \(n\)-simplex and whose marked edges are all the degenerate edges together with the edge \(\Del^{\{0,1\}}\). The category of marked simplicial sets will be denoted by \(\s^+\). It is locally presentable and cartesian closed. For more background on marked simplicial sets we refer the reader to the comprehensive treatment in \cite{HTT}.

We will denote by \(\nCat{\s^+}\) the category of categories in enriched in \(\s^+\) with respect to the cartesian product on \(\s^+\), to which we will refer to as \emph{marked-simplicial categories}. For a marked-simplicial category \(\C\) and two objects \(x,y \in \C\) we will denote by \(\C(x,y) \in \s^+\) the associated mapping marked simplicial set. By an arrow \(e\colon x \to y\) in a marked-simplicial category \(\C\) we will simply mean a vertex \(e \in \C(x,y)_0\) in the corresponding mapping object. 

We will generally consider \(\nCat{\s^+}\) together with its associated \emph{Dwyer--Kan model structure} (see~\cite[\S A.3.2]{HTT}). In this model structure the weak equivalences are the Dwyer--Kan equivalences, that is, the maps which are essentially surjective on homotopy categories and induce marked categorical equivalences on mapping objects. The fibrant objects are the enriched categories \(\C\) whose mapping objects \(\C(x,y)\) are all fibrant, that is, \(\infty\)-categories marked by their equivalences. The model category \(\nCat{\s^+}\) is then a presentation of the theory of \((\infty,2)\)-categories, and is Quillen equivalent to other known models, see \S\ref{sec:scaled} below.

We say that \(\E\in \nCat{\s^+} \) is a \emph{\(\Catoo\)-category} if it is fibrant in the Dwyer--Kan model structure, i.e.~if it is enriched over \(\infty\)-categories, with marking given by equivalences. A \emph{fibration} of \(\Catoo\)-categories is a fibration between fibrant objects in the Dwyer--Kan model structure on \(\nCat{\s^+}\).

\subsection{Scaled simplicial sets and \pdfoo-bicategories}\label{sec:scaled}
 \begin{define}[\cite{LurieGoodwillie}]
 	A \emph{scaled simplicial set} is a pair \((X,T_X)\) where \(X\) is simplicial set and \(T_X\) is a subset of the set of triangles of \(X\), called \emph{thin triangles}, containing the degenerate ones. A map of scaled simplicial sets \(f\colon (X,T_X)\rightarrow (Y,T_Y)\) is a map of simplicial sets \(f\colon X \rightarrow Y\) satisfying \(f(T_X)\subseteq T_Y\). 
 \end{define}

We will denote by \(\Ss\) the category of scaled simplicial sets. It is locally presentable and cartesian closed. When denoting an explicit scaled simplicial set we will often omit the reference to the degenerate edges. For example, we will write \((\Del^n,\Del^{\{0,1,n\}})\) for the scaled simplicial set whose underlying simplicial set is the \(n\)-simplex and whose thin triangles are all the degenerate triangles together with the triangle \(\Del^{\{0,1,n\}}\).

 \begin{define}
 	\label{d:anodyne}
 	The set of \emph{generating scaled anodyne maps} \(\bS\) is the set of maps of scaled simplicial sets consisting of:
 	\begin{enumerate}
 		\item[(i)]\label{item:anodyne-inner} the inner horns inclusions
 		\[
 		 \bigl(\Lambda^n_i,\{\Delta^{\{i-1,i,i+1\}}\}_{|\Lambda^n_i}\bigr)\rightarrow \bigl(\Delta^n,\{\Delta^{\{i-1,i,i+1\}}\}\bigr)
 		 \quad , \quad n \geq 2 \quad , \quad 0 < i < n ;
 		\]
 		\item[(ii)]\label{i:saturation} the map 
 		\[
 		 (\Delta^4,T)\rightarrow (\Delta^4,T\cup \{\Delta^{\{0,3,4\}}, \ \Delta^{\{0,1,4\}}\}),
 		\]
 		where
 		\[
 		 T\overset{\text{def}}{=}\{\Delta^{\{0,2,4\}}, \ \Delta^{\{ 1,2,3\}}, \ \Delta^{\{0,1,3\}}, \ \Delta^{\{1,3,4\}}, \ \Delta^{\{0,1,2\}}\};
 		\]
 		\item[(iii)]\label{item:anodyne_outer} the set of maps
 		\[
 		\Bigl(\Lambda^n_0\coprod_{\Delta^{\{0,1\}}}\Delta^0,\{\Delta^{\{0,1,n\}}\}\Bigr)\rightarrow \Bigl(\Delta^n\coprod_{\Delta^{\{0,1\}}}\Delta^0,\{\Delta^{\{0,1,n\}}\}\Bigr)
 		\quad , \quad n\geq 3.
 		\]
 	\end{enumerate}
 	A general map of scaled simplicial set is said to be \emph{scaled anodyne} if it belongs to the weakly saturated closure of \(\bS\).
 \end{define}

\begin{define}\label{d:bicategory}
An \emph{\(\infty\)-bicategory} is a scaled simplicial set \(\C\) which admits extensions along all maps in \(\bS\). 
\end{define}

To avoid confusion we point out that scaled simplicial sets as in Definition~\ref{d:bicategory} are referred to in~\cite{LurieGoodwillie} as \emph{weak \(\infty\)-bicategories}, while the term \(\infty\)-bicategory was used to designate the stronger property of being fibrant in the \emph{bicategorical model structure} on \(\Ss\) constructed in loc.~cit., whose cofibrations are the monomorphisms, and which serves as a model for the theory of \((\infty,2)\)-categories. It is related to the model of marked-simplicial categories mentioned above via a Quillen equivalence
\[
\xymatrixcolsep{1pc}
\vcenter{\hbox{\xymatrix{
			**[l]\Ss \xtwocell[r]{}_{\rN^{\sca}}^{\fCs}{'\perp}& **[r] \nCat{\s^+}}}} ,
\]
in which the right functor \(\rN^{\sca}\) is called as the \emph{scaled coherent nerve}. 
Nonetheless, as we have shown in \cite{GagnaHarpazLanariEquiv}, the weak and strong notions of \(\infty\)-bicategory in fact coincide. In particular, the fibrant objects in the bicategorical model structure can be characterized by the extension property of Definition~\ref{d:bicategory}, and the notion of weak \(\infty\)-bicategory will not be further mentioned in the present paper. 

\begin{notate}
We will refer to 
 the weak equivalences in the bicategorical model structure as \emph{bicategorical weak equivalences}. Since all the objects in the bicategorical model structure are cofibrant, the left Quillen equivalence \(\fCs\) preserves and detects weak equivalences.
\end{notate}

\begin{notate}
    In a drawing, every 2-simplex filled by a 2-cell with the equivalence symbol, or simply filled by an equivalence symbol, such as in the triangles
    \[
        \begin{tikzcd}[column sep=tiny]
            \bullet \ar[rr, ""{swap, name=s}] \ar{rd}   && \bullet  \\
                & \bullet \ar[ru]   &
                \ar[Rightarrow, from=s, to=2-2, shorten >= 2mm, "\simeq"{pos=0.4}]
        \end{tikzcd}
        \quad\text{and}\quad
        \begin{tikzcd}[column sep=tiny]
            \bullet \ar[rr, ""{swap, name=s}] \ar{rd}   && \bullet  \\
                & \bullet \ar[ru]   &
                \ar[phantom, from=s, to=2-2, shorten >= 2mm, "\simeq"{description, pos=0.4}]
        \end{tikzcd},
    \]
    is a thin 2-simplex. As for 3-simplices, we will often draw then as
planarised tetrahedra
    \begin{nscenter}
            \begin{tikzpicture}[scale=1.5]
            \square{%
                /square/label/.cd,
                0={$0$}, 1={$1$}, 2={$2$}, 3={$3$},
                01={$01$}, 12={$12$}, 23={$23$}, 03={$03$}, 13={$13$}, 02={$02$}, 
                012={$012$}, 023={$023$}, 013={$013$}, 123={$123$}
            }
            \end{tikzpicture}
    \end{nscenter}
where an additional equivalence symbol can appear in some of the triangles to indicate their thinness.
\end{notate}

\begin{define}\label{def:whiskering}
Given a 3-simplex \(\rho\colon \Del^3 \to X\) of the form
		\begin{nscenter}
			\begin{tikzpicture}[scale=1.5]
			\square{%
				/square/label/.cd,
                0={$x_0$}, 1={$x_1$}, 2={$x_2$}, 3={$x_3$},
				01={$f$}, 12={}, 23={}, 03={}, 
				012={$\simeq$}, 023={$\sig'$}, 013={$\simeq$}, 123={$\sig$},
				0123={}
			}
			\end{tikzpicture}
		\end{nscenter}
we will say that \(\rho\) exhibits \(\sig'\) as the \ndef{left whiskering} of \(\sig\) by \(f\). Similarly, a 3-simplex \(\rho\colon \Del^3 \to X\) of the form  
\begin{nscenter}
			\begin{tikzpicture}[scale=1.5]
			\square{%
				/square/label/.cd,
                0={$x_0$}, 1={$x_1$}, 2={$x_2$}, 3={$x_3$},
				01={}, 12={}, 23={$f$}, 
				012={$\sig$}, 023={$\simeq$}, 013={$\sig'$}, 123={$\simeq$},
				0123={}
			}
			\end{tikzpicture}
		\end{nscenter}
will be said to exhibit \(\sig'\) as the \ndef{right whiskering} of \(\sig\) by \(f\).
\end{define}

 \begin{notate}\label{not:scaled_flat-sharp}
    Let \(X\) be a simplicial set. We will denote by \(X_{\flat} = (X, \deg_2(X))\)
    the scaled simplicial set where the
    thin triangles of \(X\) are the degenerate \(2\)-simplices and by \(X_{\sharp} = (X, X_2)\) the scaled simplicial
    set where all the triangles of \(X\) are thin.
    The assignments
    \[X \mapsto X_\flat\qquad\text{and}\qquad X \mapsto X_\sharp\]
    are left and right adjoint, respectively, to the forgetful functor
    \(\Ss \to \s\).
\end{notate}

\begin{define}
    \label{core defi}
    Given a scaled simplicial set \(X\), we define its \emph{core} to be the simplicial set \(X^{\thi}\) spanned by those \(n\)-simplices of \(X\) whose 2-dimensional faces are all thin triangles. The assignment \(X \mapsto X^{\thi}\) is then right adjoint to the functor \((-)_{\sharp}\colon \s \to \Ss\).
\end{define}

\begin{warning}
    In \cite[\href{https://kerodon.net/tag/01XA}{Tag 01XA}]{LurieKerodon}, Lurie uses the term \emph{pith} in place of core, and denotes it by \(\operatorname{Pith}(\operatorname{\mathcal{C}})\).
\end{warning}

 \begin{rem}
 If \(\C\) is an \(\infty\)-bicategory then its core \(\C^{\thi}\) is an \(\infty\)-category.
 \end{rem}

\begin{define}\label{d:equivalence}
	Let \(\C\) be an \(\infty\)-bicategory. We will say that an edge in \(\C\) is \ndef{invertible} if it is invertible when considered in the \(\infty\)-category \(\C^{\thi}\), that is, if its image in the homotopy category of \(\C^{\thi}\) is an isomorphism. We will sometimes refer to invertible edges in \(\C\) as \ndef{equivalences}. We will denote by \(\C^{\minisimeq} \subseteq \C^{\thi}\) the subsimplicial set spanned by the invertible edges. Then \(\C^{\minisimeq}\) is an \(\infty\)-groupoid (that is, a Kan complex), which we call the \emph{core groupoid} of \(\C\). It can be considered as the \(\infty\)-groupoid obtained from \(\C\) by discarding all non-invertible \(1\)-cells and \(2\)-cells. If \(X\) is an arbitrary scaled simplicial set then we will say that an edge in \(X\) is \ndef{invertible} if its image in \(\C\) is invertible for any bicategorical equivalence \(X \to \C\) such that \(\C\) is an \(\infty\)-bicategory. This does not depend on the choice of the \(\infty\)-bicategory replacement \(\C\).
\end{define}

\begin{notate}
	\label{n:mapping}
	Let \(\C\) be an \(\infty\)-bicategory and let \(x,y \in \C\) be two vertices.
	In~\cite[\S 4.2]{LurieGoodwillie}, Lurie gives an explicit
	model for the mapping \(\infty\)\nbd-cat\-egory  from \(x\) to \(y\) in \(\C\) that we now recall.
	Let \(\Hom_{\C}(x,y)\) be the marked simplicial set whose \(n\)-simplices are given by maps \(f\colon\Del^n \times \Del^1 \lrar \C\) such that \(f_{|\Del^n \times \{0\}}\) is constant on \(x\), \(f_{|\Del^n \times \{1\}}\) is constant on \(y\), and the triangle \(f_{|\Del^{\{(i,0),(i,1),(j,1)\}}}\) is thin 
	for every \(0 \leq i\leq j \leq n\). An edge \(f\colon\Del^1 \times \Del^1 \lrar \C\) of \(\Hom_{\C}(x,y)\) is marked exactly when the triangle \(f_{|\Del^{\{(0,0),(1,0),(1,1)\}}}\) is thin. 
	The assumption that \(\C\) is an \(\infty\)-bicategory implies that 
	the marked simplicial set \(\Hom_{\C}(x,y)\) is \ndef{fibrant} in the marked categorical model structure, that is, it is an \(\infty\)-category whose marked edges are exactly the equivalences.	
\end{notate}

\begin{define}
We will denote by \(\Catoo\) the scaled coherent nerve of the (fibrant) marked-simplicial subcategory \((\s^+)^{\circ} \subseteq \s^+\) spanned by the fibrant marked simplicial sets. We will refer to \(\Catoo\) as the \emph{\(\infty\)-bicategory of \(\infty\)-categories.}
\end{define}

 \begin{define}
 We define \(\BiCat\) to be the scaled coherent nerve of the (large)
 marked-simplicial category \(\mathrm{BiCat}_{\Del}\) whose objects are the \(\infty\)-bicategories
 and whose mapping marked simplicial set, for \(\C,\D \in \mathrm{BiCat}_{\Del}\),
 is given by \(\mathrm{BiCat}_{\Del}(\C,\D) := \Fun^{\thi}(\C,\D)^{\natural}\).
 Here \(\Fun^{\thi}(\C,\D)\) is the core \(\infty\)-category of the internal hom scaled simplicial set \(\Fun(\C, \D)\),
 which happens to be an \(\infty\)-bicategory whenever \(\D\) is so (see \cite[Proposition 3.1.8 and Lemma 4.2.6]{LurieGoodwillie}),
 and by \((-)^{\natural}\) we mean that the associated marked simplicial set in which the marked arrows are the equivalences.
 We will refer to \(\BiCat\) as the \emph{\(\infty\)-bicategory of \(\infty\)-bicategories.}
 \end{define}

 Since the scaled coherent nerve functor \(\Nsc\) is a right Quillen equivalence, it determines an equivalence 
 \begin{equation}\label{e:enriched-cat-model}
 (\nCat{\s^+})_{\infty} \xrightarrow{\simeq} \BiCat^{\thi}
 \end{equation}
 between the \(\infty\)-category associated to the model category \(\nCat{\s^+}\) and the core \(\infty\)-category of \(\BiCat\).
In the model \(\nCat{\s^+}\) the \((\ZZ/2)^2\)-action on the theory of \((\infty,2)\)-categories can be realized by an action of \((\ZZ/2)^2\) on \(\nCat{\s^+}\) via model category isomorphisms: the operation \(\C \mapsto \C^{\op}\) which inverts only the direction of 1-morphisms is realized by setting \(\C^{\op}(x,y) = \C(y,x)\), while the operation \(\C \mapsto \C^{\co}\) of inverting only the direction of 2-morphisms is realized by setting \(\C^{\co}(x,y) = \C(x,y)^{\op}\), where the right hand side denotes the operation of taking opposites in marked simplicial sets. 
 Through the equivalence~\eqref{e:enriched-cat-model} 
 these two involutions induce a \((\ZZ/2)^{2}\)-action on the core \(\infty\)-category \(\BiCat^{\thi}\), which we then denote by the same notation. In particular, we have involutions 
 \[(-)^{\op}\colon \BiCat^{\thi} \to \BiCat^{\thi} \quad\text{and}\quad (-)^{\co}\colon \BiCat^{\thi} \to \BiCat^{\thi},\] 
 the first inverting the direction of 1-morphisms and the second the direction of 2-morphisms. 

 \begin{rem}\label{r:twisted-op}
 The \((\ZZ/2)^2\)-action on \(\BiCat^{\thi}\) does not extend to an action of \((\ZZ/2)^2\) on the \(\infty\)-bicategory \(\BiCat\). Instead, 
 it extends to a \emph{twisted} action, that is, \((-)^{\co}\) and \((-)^{\op}\) extend to equivalences of the form: 
 \[(-)^{\co}\colon \BiCat \xrightarrow{\simeq} \BiCat\] 
and 
 \[(-)^{\op}\colon \BiCat \xrightarrow{\simeq} \BiCat^{\co}.\]
 \end{rem}

\subsection{Join and slice}\label{sec:join-and-slice}

In~\cite[\S 2.2]{GagnaHarpazLanariEquiv} and~\cite[\S 2.1]{GagnaHarpazLanariLaxLimits} we used join and slice constructions in the setting of marked-scaled simplicial sets, that is, simplicial sets \(X\) endowed both with a distinguished collection \(E_X \subseteq X_1\) of marked edges and a distinguished collection \(T_X \subseteq X_2\) of thin triangles.
The category of marked-scaled simplicial sets will be denoted by \(\Sms\). Though we will need only a limited amount of the generality used in loc.~cit., let us recall the construction and terminology used there for the sake of consistency.  Given two marked-scaled simplicial sets \((X,E_X,T_X),(Y,E_Y,T_Y)\), their join is the \emph{scaled} simplicial set \((X * Y,T_{X * Y})\) whose underlying simplicial set \(X * Y\) is the usual join of simplicial sets
and the collection of thin triangles is
\[ T_{X * Y} := [T_X \times Y_0] \cup [E_X \times E_Y] \cup [X_0 \times T_Y] \subseteq [X_2 \times Y_0] \cup [X_1 \times Y_1] \cup [X_0 \times Y_2] = (X * Y)_2 .\]
For a fixed marked-scaled simplicial set \((Y,E_Y,T_Y)\) the functor \((X,E_X,T_X) \mapsto (X * Y,E_{X*Y})\) is a colimit preserving functor from marked-scaled simplicial sets to scaled simplicial sets under \((Y,T_Y)\), and admits a right adjoint, that is, an associated slice construction. Given a marked-scaled simplicial set \((K,E_K,T_K)\) and a map of scaled simplicial set \(p\colon (K,T_K) \to (Z,T_Z)\), we will denote by 
\((Z,T_Z)_{/p}\), or simply \(Z_{/p}\) for brevity, the valued of this right adjoint at \(p\). In particular, \(Z_{/p}\) is the marked-scaled simplicial set characterized by the property that maps of marked-scaled simplicial sets \((X,E_X,T_X) \to Z_{/p}\) correspond to maps of scaled simplicial sets \((X*K,T_{X*K}) \to (Z,T_Z)\) under \((K,E_K,T_K)\). We then write \(\ovl{Z}_{/p}\) for the \emph{underlying scaled simplicial set} of \(Z_{/p}\), obtained by forgetting the marking. To avoid confusion, let us emphasize that the thin triangles in \(\ovl{Z}_{/p}\) do depend on the marking \(E_K\) of \(K\), and not only on the map of scaled simplicial sets \(p\). Similarly, we denote by \(Z_{p/}\) the marked-scaled simplicial set representing the functor
\[
	(X,E_X,T_X) \mapsto \Map_{(\Ss)_{(K,T_K)/}}((K*X,T_{K*X}),T_{X*K},(Z,T_Z)),
\]
and by \(\ovl{Z}_{p/}\) its underlying scaled simplicial set. 

We will mostly be interested in the case where the target \(Z\) is an \(\infty\)-bicategory \(\C\), and \(K\) is either \(\Del^0\) or \(\prescript{\sharp}{}{\Delta^1}\), the latter being the 1-simplex endowed with the maximal marking and the (unique) trivial scaling. In the latter case we will sometimes make use of~\cite[Notation 2.3.4]{GagnaHarpazLanariLaxLimits} which we now recall for the convenience of the reader.

\begin{notate}\label{not:slice-over-arrow}
	Given a scaled simplicial set \((Z,T_Z)\) and an edge \(e\colon x \to y\) in \(Z\), we will denote by \(\trbis{Z}{e^{\sharp}}\in \Sms\) the result of the slice construction above applied to the marked-scaled simplicial set \(\prescript{\sharp}{}{\Delta^1}\) and the map of scaled simplicial sets \(\Del^1_{\flat} \to (Z,T_Z)\) determined by \(e\).
	Explicitly, the set of \(n\)-simplices in \(\trbis{Z}{e^{\sharp}}\) is given by:
	\[
	 (\trbis{Z}{e^{\sharp}})_n\stackrel{\text{def}}{=} \bigl\{\alpha\colon \prescript{\flat}{}{\Delta^n} \ast \prescript{\sharp}{}{\Delta^1} \to Z \ \vert \ \alpha_{\vert \Delta^{\{n+1,n+2\}}}=e \bigr\},
	\] 
	where here \(\prescript{\flat}{}{\Delta^n}\) denotes the \(n\)-simplex with minimal marking and minimal scaling.
	The marked edges of \(\trbis{Z}{e^{\sharp}}\) are those which factor through 
	\(\prescript{\sharp}{}{\Delta^1} \ast \prescript{\sharp}{}{\Delta^1}\)
	and the thin triangles are those which factor through 
	\((\Delta^2)^{\flat}_{\sharp} \ast \prescript{\sharp}{}{\Delta^1}\), where \((\Delta^2)^{\flat}_{\sharp}\) is the \(2\)-simplex
	with minimal marking and maximal scaling, that is, its unique non-degenerate triangle is thin. 
	As above, we will write \(\ovl{Z}_{/e^{\sharp}}\) for the underlying scaled simplicial set of \(Z_{/e^{\sharp}}\).
\end{notate}

When \(K=\Del^0\) the map \(p\) corresponds to a vertex \(y \in \C\), and we will denote the associated marked-scaled slice by \(\C_{/y}\). The fibre \((\C_{/y})_x\) of the projection \(\C_{/y} \to \C\) over a vertex \(x\) of \(\C\) is then a marked-scaled simplicial set all of whose triangles are thin. Its underlying marked simplicial set, denoted by \(\Maptr_{\C}(x,y)\) in~\cite[\S 2.3]{GagnaHarpazLanariGrayLaxFunctors}, is fibrant in the marked categorical model structure, so that one can treat it as an \(\infty\)-category (marked by its equivalences). As such, it is a model
for the mapping \(\infty\)-category of \(\C\), see~\cite[Proposition~2.23]{GagnaHarpazLanariEquiv}. We then write \(\uMaptr_{\C}(x,y)\) for the underlying simplicial set of \(\Maptr_{\C}(x,y)\).

\subsection{1-Inner/outer (co)cartesian fibrations}
\label{s:fibrations}

The theory of inner and outer (co)cartesian fibrations
of \(\infty\)-bicategories was developed in~\cite{GagnaHarpazLanariLaxLimits} as an analogue of the usual notion of (co)cartesian fibrations of \(\infty\)-categories. As in the latter case, such a fibration encodes the data of a family of \(\infty\)-categories functorially parametrized by the base \(\B\), only that in the \(\infty\)-bicategorical setting there are four different variance flavours for this functorial dependence one can consider. Specifically, inner (resp.~outer) cocartesian fibrations encode a covariant dependence on the 1-morphisms of \(\B\) and a covariant (resp.~contravariant) dependence on the level of 2-morphisms. Similarly, inner (resp.~outer) cartesian fibrations encode a contravariant dependence on the 1-morphisms of \(\B\) and a contravariant (resp.~covariant) dependence on the level of 2-morphisms.  

Below we recall the main definitions. We refer the reader to loc.~cit.\ for a comprehensive treatment.

 \begin{define}\label{d:weak}
 	We will say that a map of scaled simplicial sets \(X \rightarrow Y\) is a \ndef{weak fibration} if it has the right lifting property with respect to the following types of maps:
 	\begin{enumerate}
 		\item
 		All scaled inner horn inclusions of the form 
 		\[ (\Lam^n_i,\{\Del^{\{i-1,i,i+1\}}\}_{|\Lam^n_i}) \subseteq (\Del^n,\{\Del^{\{i-1,i,i+1\}}\}) \] 
 		for \(n \geq 2\) and \(0 < i < n\).
 		\item
 		The scaled horn inclusions of the form: \[\Bigl(\Lam^n_0 \plus{\Del^{\{0,1\}}}\Del^0,\{\Del^{\{0,1,n\}}\}_{|\Lam^n_0}\Bigr) \subseteq \Bigl(\Del^n\plus{\Del^{\{0,1\}}}\Del^0,\{\Del^{\{0,1,n\}}\}\Bigr)\] for \(n \geq 2\).
 		\item
 		The scaled horn inclusions of the form: \[\Bigl(\Lam^n_n \plus{\Del^{\{n-1,n\}}}\Del^0,\{\Del^{\{0,n-1,n\}}\}_{|\Lam^n_n}\Bigr) \subseteq \Bigl(\Del^n\plus{\Del^{\{n-1,n\}}}\Del^0,\{\Del^{\{0,n-1,n\}}\}\Bigr)\] for \(n \geq 2\).
 	\end{enumerate}
\end{define}

\begin{rem}\label{rem:bicategorical-is-weak}
The maps appearing in Definition~\ref{d:weak} are all trivial cofibrations with respect to the bicategorical model structure. This means that any bicategorical fibration \(\E \to \B\) is in particular a weak fibration. For example, if \(\E\) is an \(\infty\)-bicategory then the terminal map \(\E \to \Del^0\) is a weak fibration.
\end{rem}

\begin{define}\label{def:inner-outer}
Given a weak fibration \(f\colon X \rightarrow Y\), we will say that \(f\) is 
 \begin{itemize}
 \item a \ndef{1-inner fibration} if it detects thin triangles and the underlying map of simplicial sets is an inner fibration, that is, satisfies the right lifting property with respect to inner horn inclusions;
 \item a \ndef{1-outer fibration} if it detects thin triangles and  
 the underlying map of simplicial sets 
 satisfies the right lifting property with respect to the inclusions 
 \[\Lam^n_0 \coprod_{\Del^{\{0,1\}}}\Del^0 \subseteq \Del^n\coprod_{\Del^{\{0,1\}}}\Del^0 \quad \text{and} \quad
 \Lam^n_n \coprod_{\Del^{\{n-1,n\}}}\Del^0 \subseteq \Del^n\coprod_{\Del^{\{n-1,n\}}}\Del^0\]
 for \(n \geq 2\). 
 \end{itemize}

 Note that the collection of 1-inner fibrations is closed under the \((-)^{\op}\) duality, and the same holds for the collection of 1-outer fibrations.
 \end{define}

\begin{rem}
In~\cite{GagnaHarpazLanariEquiv} and~\cite{GagnaHarpazLanariLaxLimits} we used the terms inner and outer fibrations for what we called above 1-inner and 1-outer fibrations, respectively. The reason for the terminology update is the desired to more clearly distinguish between these notions and those of 2-inner and 2-outer fibrations introduced in the present paper. In principle, 1-inner and 1-outer fibrations between \(\infty\)-bicategories are the functors which induce right and left fibrations, respectively, on the level of mapping \(\infty\)-categories. The notions of 2-inner and 2-outer fibrations between \(\infty\)-bicategories correspond in turn to functors which induce cartesian and cocartesian fibrations on the level of mapping \(\infty\)-categories, together with the condition that composition of arrows preserves cocartesian edges (see \S\ref{sec:car-fib-enriched}). 
\end{rem}

 \begin{rem}
 	In \cite[\href{https://kerodon.net/tag/01WF}{Tag 01WF}]{LurieKerodon}, Lurie uses the term \emph{interior fibration} to encode what we just defined as 1-outer fibrations. Our choice in~\cite{GagnaHarpazLanariLaxLimits} (which already appeared in \cite{GagnaHarpazLanariEquiv}) is motivated by the intent of highlighting that \emph{special outer horns} admit fillers against such maps.
 \end{rem}

 \begin{define}\label{d:cartesian}
 	Let \(p\colon X \rightarrow Y\) be a weak fibration. We will say that an edge \(e\colon \Del^1 \rightarrow X\) is \ndef{\(p\)-cartesian} if the dotted lift exists in any diagram of the form
 	\[ \xymatrix{
 		(\Lam^n_n,\{\Del^{\{0,n-1,n\}}\}_{|\Lam^n_n}) \ar^-{\sig}[r]\ar[d] & (X,T_X) \ar^p[d] \\
 		(\Del^n,\{\Del^{\{0,n-1,n\}}\}) \ar@{-->}[ur]\ar[r] & (Y,T_Y) \\
 	}\]
 	with \(n \geq 2\) and \(\sig_{|\Del^{\{n-1,n\}}} = e\). We will say that \(e\) is \ndef{\(p\)-cocartesian} if \(e^{\op}\colon \Del^1 \to X^{\op}\) is \(p^{\op}\)-cartesian.

    As in~\cite[Definition 2.3.1]{GagnaHarpazLanariLaxLimits}, we will also say that the edge \(e \colon \Delta^1 \to X\) is \ndef{strongly \(f\)-(co)cartesian} if
    it is a (co)cartesian edge with respect to the underlying map of simplicial sets. 
 \end{define}

\begin{rem}\label{rem:equiv-is-car}
If \(p\colon \E \to \B\) is weak fibration between \(\infty\)-bicategories, then any equivalence in \(\E\) is both \(p\)-cartesian and \(p\)-cocartesian, see~\cite[Corollary 2.3.10]{GagnaHarpazLanariLaxLimits}. On the other hand, if \(e\colon x \to y\) is either a \(p\)-cartesian or \(p\)-cocartesian edge in \(\E\) such that \(pe\) is an equivalence in \(\B\) then \(e\) is an equivalence. To see this, let \(g\colon y \to x\) be an inverse to \(pe\) in \(\B\), equipped with thin triangles of the form
    \[
        \begin{tikzcd}
            px \ar[rr, equal, ""{swap, name=s}] \ar[rd,"pe"']   && px  \\
                & y \ar[ru, "g"']   &
                \ar[phantom, from=s, to=2-2, shorten >= 2mm, "\simeq"{pos=0.4}]
        \end{tikzcd}
        \quad\text{and}\quad
        \begin{tikzcd}
            py \ar[rr, equal, ""{swap, name=s}] \ar[rd,"g"']   && py  \\
              &  px \ar[ru,"pe"'] &   \ .
                \ar[phantom, from=s, to=2-2, shorten >= 2mm, "\simeq"{description, pos=0.4}]
        \end{tikzcd}
    \]
If \(e\) is \(p\)-cartesian then we can lift the right hand side triangle to a triangle in \(\E\) of the form
\[         
		\begin{tikzcd}
            y \ar[rr, equal, ""{swap, name=s}] \ar[rd,"g'"']   && y  \\
              &  x \ar[ru,"e"'] &   \ ,
                \ar[phantom, from=s, to=2-2, shorten >= 2mm, "\simeq"{description, pos=0.4}]
        \end{tikzcd}
\]
producing in particular a right inverse \(g'\) to \(e\). But then \(g'\) is also \(p\)-cartesian by~\cite[Lemma 2.3.8]{GagnaHarpazLanariLaxLimits}, and since \(pg'= g\) is invertible in \(\B\) we deduce from the same argument that \(g'\) also has a right inverse in \(\E\). It then follows from standard arguments (which can be applied on the level of the core \(\infty\)-category \(\C^{\thi}\)) that \(e\) and \(g'\) are homotopy inverses, and in particular \(e\) is an equivalence. If \(e\) is assumed instead to be \(p\)-cocartesian then one proceeds in the same manner by lifting the left hand side triangle to \(\E\). 
\end{rem}

\begin{rem}\label{rem:car-closed-equiv}
If \(p\colon \E \to \B\) is a weak fibration and \(e,e'\colon \Del^1 \to \E\) are two arrows which are equivalent in \(\Fun(\Del^1,\E)\) then \(e\) is \(p\)-cartesian if and only if \(e'\) is. To see this, note that in this case one has both an equivalence going from \(e\) to \(e'\) and an equivalence going from \(e'\) to \(e\), and so it will hence suffice to show that if we have an equivalence \(e \Rightarrow e'\) and \(e'\) is \(p\)-cartesian then \(e\) is \(p\)-cartesian. Indeed, such an equivalence is given by a map \(H\colon \Del^1_{\flat} \times \Del^1_{\flat} \to \E\) such that \(H|_{\Del^1_{\flat} \times \Del^{\{0\}}} = e\) and \(H|_{\Del^1_{\flat} \times \Del^{\{1\}}} = e'\). Since \(p\)-cartesian edges are closed under composition (\cite[Lemma 2.3.9]{GagnaHarpazLanariLaxLimits}) and every equivalence is \(p\)-cartesian (Remark~\ref{rem:equiv-is-car}) we get that \(H\) sends the diagonal edge \(\Del^{\{(0,0),(1,1)\}} \subseteq \Del^1 \times \Del^1\) to a \(p\)-cartesian edge. Then, from the partial 2-out-of-3 property for \(p\)-cartesian edges of~\cite[Lemma 2.3.8]{GagnaHarpazLanariLaxLimits} we get that \(e\) is \(p\)-cartesian. 

Passing to opposites, we also obtain from this argument that \(e\) is \(p\)-cocartesian if and only if \(e'\) is. In addition, if \(p\) is a 1-outer fibration then by~\cite[Proposition 2.3.7]{GagnaHarpazLanariLaxLimits} the collection of strongly \(p\)-(co)cartesian arrows coincides with that of \(p\)-(co)cartesian arrows, and hence \(e\) is strongly \(p\)-(co)cartesian if and only if \(e'\) is so (alternatively, when \(p\) is a 1-outer fibration the statements of~\cite[Lemma 2.3.8 and Lemma 2.3.9]{GagnaHarpazLanariLaxLimits} also apply to strongly \(p\)-(co)cartesian arrows, so that the above argument can simply be carried out verbatim).
\end{rem}

 \begin{define}\label{d:car-fibration}
 	Let \(f\colon X \rightarrow Y\) be a weak fibration of scaled simplicial sets.
 	We will say that \(f\) is a \ndef{cartesian fibration} if for every \(x \in X\) and an edge \(e\colon y \rightarrow f(x)\) in \(Y\) there exists a \(f\)-cartesian edge \(\wtl{e}\colon \wtl{y} \to x\) such that \(f(\wtl{e}) = e\).
 	Dually, we will say that \(f\colon X \rightarrow Y\) is a \ndef{cocartesian fibration} if \(f^{\op}\colon X^{\op} \rightarrow Y^{\op}\) is a cartesian fibration.
 \end{define}

 \begin{define}\label{d:car-fibration-2}
 	Let \(f\colon X \rightarrow Y\) be a weak fibration of scaled simplicial sets. We will say that \(f\) is a 1-inner (resp.~1-outer) cartesian fibration if it is both a 1-inner (resp.~1-outer) fibration and a cartesian fibration. Dually, we will say that \(f\) is a 1-inner (resp.~1-outer) cocartesian fibration if \(p^{\op}\) is a 1-inner
(resp.~1-outer) cartesian fibration. 
 \end{define}

 \begin{rem}\label{r:base-change}
 	The classes of weak fibrations, 1-inner/outer fibrations and 1-inner/outer (co)cartesian fibrations are all closed under base change.
 \end{rem}

\begin{example}\label{ex:slice-projection}
If \(\C\) is an \(\infty\)-bicategory then the projection \(\ovl{\C}_{/x} \to \C\) is an example of a 1-outer cartesian fibration, where the cartesian edges are exactly those whose corresponding triangle in \(\C\) is thin (equivalently, those corresponding to marked edges in \(\C_{/x}\)), 
see~\cite[Corollary 2.4.7]{GagnaHarpazLanariLaxLimits}. Similarly, \(\ovl{\C}_{x/} \to \C\) is a 1-outer cocartesian fibration, with the cocartesian edges again corresponding to thin triangles. More generally, for any map of scaled simplicial sets \(p \colon (K,T_K) \to \C\), the associated slice projections 
\(\ovl{\C}_{/p} \to \ovl{\C}\) and \(\ovl{\C}_{p/}\to \C\) are 1-outer cartesian and cocartesian fibrations, respectively.
\end{example}

\begin{prop}
	\label{cart are fib}
	Any (co)cartesian fibration between \(\infty\)-bicategories is a fibration in the bicategorical model structure on \(\Ss\).
\end{prop}
\begin{proof}
We prove the cartesian case, from which the cocartesian case can be deduced by passing to opposites.
Let \(p\colon \E \to \B\) be a cartesian fibration between \(\infty\)-bicategories. To show that \(p\) is a bicategorical fibration we need to produce the dotted lift in any square of the form
\begin{equation}\label{eq:is-fibrant} \xymatrix{
K \ar[r]^{f}\ar[d] & \E \ar[d] \\
L \ar[r]^{g}\ar@{..>}[ur] & \B 
}
\end{equation}
such that \(K \to L\) is a bicateogrical trivial cofibration of scaled simplicial sets.
Since \(\E\) is an \(\infty\)-bicategory it is in particular fibrant in the bicategorical model structure (see discussion in \S\ref{sec:scaled}), and hence we can extend \(f\) to a map \(h\colon L \to \E\). This is not yet a solution to the above lifting problem since the composite \(ph\) might be different from \(g\). The two maps \(ph\) and \(g\) agree however on \(K\) by construction. Since the bicategorical model structure is cartesian closed and \(\B\) is fibrant we may solve the lifting problem
\[ \xymatrix{
[K \times \Del^1_{\flat}] \coprod_{K \times \partial\Del^1} [L \times \partial\Del^1] \ar[r]\ar[d] & \B \\
L \times \Del^1_{\flat} \ar@{..>}[ur] & 
}\]
yielding a natural transformation \(H\colon L \times \Del^1_{\flat} \to \B\) from \(g\) and \(ph\) which is constant on \(K\). Since \(K \to L\) is a trivial cofibration and \(\B\) is fibrant the induced functor \(\Fun(L,\B) \to \Fun(K,\B)\) is an equivalence of \(\infty\)-bicategories, and since the arrow in \(\Fun(L,\B)\) associated to \(H\) maps to an identity arrow in \(\Fun(K,\B)\) we deduce that it must be invertible in \(\Fun(L,\B)\). In particular, the restriction of \(H\) to \(\{l\} \times \Del^1_{\flat}\) is an invertible arrow of \(\B\) for every vertex \(l \in L\). Now since \(p\) is a cartesian fibration we can choose for each \(l \in L_0 \setminus K_0\) a \(p\)-cartesian lift \(e_l\colon \{l\} \times \Del^1_{\flat} \to \E\) of \(H|_{\{l\} \times \Del^1_{\flat}}\). Then each \(e_l\) is a \(p\)-cartesian lift of an equivalence, and is hence itself an equivalence, see Remark~\ref{rem:equiv-is-car}. Let \(L'\subseteq L\) be the scaled simplicial subset whose underlying simplicial set is that of \(L\) and whose thin triangles are only those thin triangles which are contained in \(K\). Applying~\cite[Proposition 2.38]{GagnaHarpazLanariEquiv} we may solve the lifting problem
\[ \xymatrix{
[K \times \Del^1_{\flat}] \coprod_{K \times \Del^{\{1\}}} [L' \times \Del^{\{1\}}] \ar[r]\ar[d] & \E\ar[d] \\
L' \times \Del^1_{\flat} \ar@{..>}[ur]^{G}\ar[r]^{H} & \B
}\]
to yield a natural transformation \(G\colon L' \times \Del^1_{\flat} \to \E\) such that \(G|_{\{l\} \times \Del^1_{\flat}} = e_l\) for every \(l \in L_0 \setminus K_0\), where we point out that the assumption made in~\cite[Proposition 2.38]{GagnaHarpazLanariEquiv} that \(p\) detects thin triangles is not needed since \(L'\) does not contain any thin triangles that are not in \(K\), see~\cite[Remark 2.40]{GagnaHarpazLanariEquiv}. In particular, \(G\) is a levelwise invertible natural transformation. We wish to show that \(G|_{L'\times \Del^{\{0\}}}\) extends to \(L \times \Del^{\{0\}}\), thus providing a solution to the original lifting problem~\eqref{eq:is-fibrant}. Indeed, if \(\sig\colon \Del^2_{\flat} \to L'\) is a triangle that is thin in \(L\) then the composed map 
\[H_{\sig}\colon \Del^2_{\flat} \times \Del^1_{\flat} \xrightarrow{\sig \times \id} L'\times \Del^1_{\flat} \xrightarrow{G} \E\]
is a levelwise invertible natural transformation between two triangles, one of which is thin (since \(G|_{L' \times \Del^{\{1\}}}\) extends to \(L \times \Del^1_{\flat}\) by construction), and hence the other one is thin as well by~\cite[Corollary 3.5]{GagnaHarpazLanariEquiv}.
\end{proof}

Over a fixed base \(\B\), the collection of 1-inner cocartesian fibrations, and cocartesian edges preserving functors between them, can be organized into an \(\infty\)-bicategory \(\coCar^{\inn}(\B)\). This \(\infty\)-bicategory can be presented by a suitable model structure on the category of marked simplicial sets over the underlying simplicial set of \(\B\), developed in~\cite[\S 3.2]{LurieGoodwillie} using the machinery of categorical patterns. Lurie then constructs in loc.~cit.\ a straightening-unstraightening Quillen equivalence between this model structure and the projective model structure on \(\Fun(\fCs(\B),\s^+)\). In \cite{GagnaHarpazLanariLaxLimits} we used this to establish the following \(\infty\)-bicategorical form of the Grothendieck--Lurie correspondence, for both cartesian and cocartesian, inner and outer flavours of fibrations:

\begin{thm}[{\cite[Corollary~3.3.3]{GagnaHarpazLanariLaxLimits}}]
	\label{c:S-U-for-outer-fibs}
For an \(\infty\)-bicategory \(\B \in \BiCat\) there are natural equivalences of \(\infty\)-bicategories
\begin{align*}
    \coCar^{\inn}(\B) &\simeq \Fun(\B,\Catoo), \\
    \coCar^{\out}(\B) &\simeq \Fun(\B^{\co},\Catoo), \\
    \Car^{\inn}(\B) &\simeq \Fun(\B^{\coop},\Catoo), \\
    \Car^{\out}(\B) &\simeq \Fun(\B^{\op},\Catoo).
\end{align*}
\end{thm}

\section{2-Inner/outer cartesian fibrations}

In this section we will define the principal notion of this paper, namely that of 2-inner/outer (co)cartesian fibrations, and study their basic properties.

\subsection{Inner and outer triangles}

\begin{define}\label{def:inner-triangle}
Let \(p\colon \E \to \B\) be a weak fibration of \(\infty\)-bicategories and \(\sig\colon \Del^2 \to \E\) a triangle.
\begin{itemize}
\item
We will say that \(\sig\) is \emph{left \(p\)-inner} if the corresponding arrow in \(\ovl{\E}_{/\sig(2)}\) is strongly cartesian with respect to the projection \(\ovl{\E}_{/\sig(2)} \to \E \times_{\B} \ovl{\B}_{/p\sig(2)}\). 
\item
We will say that \(\sig\) is \emph{right \(p\)-inner} if the corresponding arrow in \(\ovl{\E}_{\sig(0)/}\) is strongly cocartesian with respect to the projection \(\ovl{\E}_{\sig(0)/} \to \E \times_{\B} \ovl{\B}_{p\sig(0)/}\). 
\item
We will say that \(\sig\) is \emph{left \(p\)-outer} if the corresponding arrow in \(\ovl{\E}_{/\sig(2)}\) is strongly cocartesian with respect to the projection \(\ovl{\E}_{/\sig(2)} \to \E \times_{\B} \ovl{\B}_{/p\sig(2)}\). 
\item
We will say that \(\sig\) is \emph{right \(p\)-outer} if the corresponding arrow in \(\ovl{\E}_{\sig(0)/}\) is strongly cartesian with respect to the projection \(\ovl{\E}_{\sig(0)/} \to \E \times_{\B} \ovl{\B}_{p\sig(0)/}\). 
\end{itemize}
\end{define}

\begin{rem}\label{rem:unwinding}
Unwinding the definitions, we see that \(\sig\) is left \(p\)-inner if and only if, for \(n \geq 3\),
every commutative square of the form:
\begin{equation}
\label{inner 2-simp}
\begin{tikzcd}
& \Delta^{\{n-2,n-1,n\}} \ar[dl] \ar[dr, "\sig"]\\
\Lambda^{n}_{n-1} \ar[rr] \ar[d]& &\E \ar[d,"p"]\\
\Delta^n \ar[urr,dotted]\ar[rr]&  &\B
\end{tikzcd}
\end{equation}
admits a diagonal filler as displayed by the dotted arrow, and right \(p\)-inner if the same holds for diagrams as above where \(\Lam^n_{n-1}\) is replaced by \(\Lam^n_{1}\) and \(\Del^{\{n-2,n-1,n\}}\) by \(\Del^{\{0,1,2\}}\). On the other hand, \(\sig\) is left \(p\)-outer if and only if 
every commutative square of the form:
\begin{equation}
\label{outer 2-simp}
\begin{tikzcd}
& \Delta^{\{0,1,n\}} \ar[dl] \ar[dr, "\sig"]\\
\Lambda^{n}_{0} \ar[rr] \ar[d]& &\E \ar[d,"p"]\\
\Delta^n \ar[urr,dotted]\ar[rr]&  &\B
\end{tikzcd}
\end{equation}
admits a diagonal filler as displayed by the dotted arrow, and and right \(p\)-outer 
if and only if the same holds for diagrams as above where \(\Lam^n_0\) is replaced by \(\Lam^n_{n}\) and \(\Del^{\{0,1,n\}}\) by \(\Del^{\{0,n-1,n\}}\).
\end{rem}

\begin{rem}\label{rem:thin-is-inner}
It follows from Remark~\ref{rem:unwinding} that any thin triangle in \(\E\) is both left and right \(p\)-inner. On the other hand, any thin triangle whose left leg is \(p\)-cocartesian is left \(p\)-outer and any thin triangle whose right leg \(p\)-cartesian is right \(p\)-outer. 
\end{rem}

Remark~\ref{rem:thin-is-inner} admits a type of a converse statement:

\begin{prop}
\label{cart over thin = thin}
Let \(p\colon \E \to \B\) be a weak fibration of \(\infty\)-bicategories. Suppose \(\sig \colon \Delta^2 \to \E\) is a triangle such that \(p(\sig)\) is thin in \(\B\). If \(\sig\) is either left or right \(p\)-inner then \(\sig\) is thin in \(\E\). The same holds if we assume that \(\sig\) is left \(p\)-outer and left-degenerate 
or that \(\sig\) is right \(p\)-outer and right-degenerate. 
\end{prop}
\begin{proof}
Write \(x = \sig(0),z = \sig(2)\).
Suppose first that \(\sig\) is left \(p\)-inner.
The condition that \(p(\sig)\) is thin means that the arrow determined by \(p(\sig)\) in \(\ovl{\B}_{/pz}\) is cartesian with respect to the projection \(\ovl{\B}_{/pz} \to \B\) (see Example~\ref{ex:slice-projection}). 
By base change it then follows that the arrow determined by \(\sig\) in \(\E \times_{\B} \ovl{\B}_{/pz}\) is cartesian with respect to the projection \(\E \times_{\B} \ovl{\B}_{/pz} \to \E\). At the same time, the arrow determined by \(\sig\) in \(\ovl{\E}_{/z}\) is cartesian with respect to the projection \(\ovl{\E}_{/z} \to \E \times_{\B} \ovl{\B}_{/pz}\) (by the definition of being left \(p\)-inner) 
and so we conclude that the arrow determined by \(\sig\) in \(\ovl{\E}_{/z}\) is cartesian with respect to the composed projection \(\ovl{\E}_{/z} \to \E\). By Example~\ref{ex:slice-projection} 
we then get that \(\sig\) is thin. The dual argument using \(\ovl{\E}_{x/}\) and \(\ovl{\B}_{px/}\) applies to the case where \(\sig\) is right \(p\)-inner.

Now suppose that \(\sig\) is left \(p\)-outer and left-degenerate. 
Then \(p\sig\) is left-degenerate 
and since \(p\sig\) is assumed thin it follows that the arrow in \(\ovl{\B}_{/pz}\) determined by \(p(\sig)\) is invertible (indeed, by Example~\ref{ex:slice-projection} 
it is a cocartesian arrow with respect to the 1-outer fibration \(\ovl{\B}_{/pz} \to \ovl{\B}\) lying over an equivalence). Since \(\sig\) is left-degenerate 
we then have that the arrow determined by \(\sig\) in \(\E \times_{\B} \ovl{\B}_{/pz}\) is invertible as well. Since \(\sig\) is left \(p\)-outer it now follows that the arrow in \(\ovl{\E}_{/z}\) determined by \(\sig\) is a cocartesian lift of an equivalence along the 1-outer fibration \(\ovl{\E}_{/z} \to \E \times_{\B} \ovl{\B}_{/pz}\), and is hence itself an invertible arrow \(\ovl{\E}_{/z}\). As such, this arrow is in particular cocartesian with respect to the projection \(\ovl{\E}_{z/} \to \E\) (Remark~\ref{rem:equiv-is-car}), and so we conclude that \(\sig\) is thin by Example~\ref{ex:slice-projection}. 
The dual argument using \(\ovl{\E}_{x/}\) and \(\ovl{\B}_{px/}\) applies to the case where \(\sig\) is right \(p\)-outer and right-degenerate.
\end{proof}

\begin{rem}\label{rem:closed-whiskering}
By~\cite[Lemma 2.3.9 and Lemma 2.3.8]{GagnaHarpazLanariLaxLimits} the collection of strongly (co)cartesian arrows in a given 1-outer fibration is closed under composition and has a partial 2-out-of-3 property. More precisely, if \(\C \to \D\) is a 1-outer fibration of \(\infty\)-bicategories and
\[
        \begin{tikzcd}
                & y \ar[rd, "g"]   & \\
            x \ar[rr, "h"{swap, name=s}] \ar[ru,"f"]   && z 
            \ar[from=s, to=1-2, phantom, phantom, "\simeq"{description, pos=0.4}]
        \end{tikzcd}
\]
is a thin triangle in \(\C\) such that \(g\) is strongly cartesian then \(f\) is strongly cartesian if and only if \(h\) is strongly cartesian. Dually, if \(f\) is strongly cocartesian then \(g\) is strongly cocartesian if and only if \(h\) is strongly cocartesian. Now for any weak fibration \(p\colon \E \to \B\) of \(\infty\)-bicategories, both \(\ovl{\E}_{x/} \to \E \times_{\B} \ovl{\B}_{x/}\) and \(\ovl{\E}_{/x} \to \E \times_{\B} \ovl{\B}_{/x}\) are 1-outer fibrations for every \(x \in \E\) by Remark~\ref{r:base-change} and Example~\ref{ex:slice-projection}. The above partial 2-out-of-3 property for strongly (co)cartesian edges then translates to a certain partial 2-out-of-3 property for inner/outer triangles. More precisely, suppose given a 3-simplex \(\rho\colon \Del^3 \to \E\) of the form

		\begin{nscenter}
			\begin{tikzpicture}[scale=1.5]
			\square{%
				/square/label/.cd,
                0={$x_0$}, 1={$x_1$}, 2={$x_2$}, 3={$x_3$},
				01={}, 12={$f$}, 23={}, 03={}, 
				012={$\theta$}, 023={$\sig$}, 013={$\tau$}, 123={$\eta$},
				0123={},
			}
			\draw (3.3,-0.75) node{$.$};			  
			\end{tikzpicture}
		\end{nscenter}
If \(\theta\) is thin then we may consider \(\rho\) as encoding a thin triangle in \(\ovl{\E}_{/x_3}\) exhibiting the edge associated to \(\sig\) as the composite of those associated to \(\tau\) and \(\eta\), whereas if \(\eta\) is thin then we may consider \(\rho\) as encoding thin a triangle in \(\ovl{\E}_{x_0/}\) exhibiting the edge associated to \(\tau\) as the composite of those associated to \(\theta\) and \(\sig\). We hence conclude the following:
\begin{enumerate}
\item
If \(\theta\) is thin and \(\eta\) is left \(p\)-inner then \(\tau\) is left \(p\)-inner if and only if \(\sig\) is. 
\item
If \(\eta\) is thin and \(\theta\) is right \(p\)-inner then \(\sig\) is right \(p\)-inner if and only if \(\tau\) is. \item
If \(\theta\) is thin and \(\tau\) is left \(p\)-outer then \(\eta\) is left \(p\)-outer if and only if \(\sig\) is. 
\item
If \(\eta\) is thin and \(\sig\) is right \(p\)-outer then \(\theta\) is right \(p\)-outer if and only if \(\tau\) is. \end{enumerate}
Combining this with Remark~\ref{rem:thin-is-inner}, we conclude that the collection of left \(p\)-inner triangles in \(\E\) is closed under left whiskering with 1-morphisms and the collection of right \(p\)-inner triangles is closed under right whiskering with 1-morphisms. On the other hand, the collection of left \(p\)-outer triangles is only closed under left whiskering with \(p\)-cocartesian 1-morphisms and the collection of right \(p\)-outer triangles is only closed under right whiskering with \(p\)-cartesian arrows.

\end{rem}

Combining Remark~\ref{rem:closed-whiskering}(4) with Remark~\ref{rem:thin-is-inner} we get that if
		\begin{nscenter}
			\begin{tikzpicture}[scale=1.5]
			\square{%
				/square/label/.cd,
                0={$x_0$}, 1={$x_1$}, 2={$x_2$}, 3={$x_3$},
				01={}, 12={}, 23={$g$}, 03={}, 
				012={$\theta$}, 023={$\simeq$}, 013={$\tau$}, 123={$\simeq$},
				0123={},
			}
			\end{tikzpicture}
		\end{nscenter}
is a 3-simplex such that \(g\) is \(p\)-cartesian then \(\theta\) is right \(p\)-outer if and only if \(\tau\) is. The ``only if'' direction of this implication also holds for right \(p\)-inner triangles by Remark~\ref{rem:closed-whiskering}(2) and Remark~\ref{rem:thin-is-inner}.
The following lemma shows that on the other hand, the ``if'' direction of this implication actually holds for left \(p\)-inner and left \(p\)-outer triangles:

\begin{lemma}\label{lem:closed-whiskering}
Let \(p\colon \E \to \B\) be a weak fibration. Given a 3-simplex \(\rho\colon \Del^3 \to \E\) as above with \(g\) being \(p\)-cartesian, if \(\tau\) is left \(p\)-inner or left \(p\)-outer then so is \(\theta\).
\end{lemma}
\begin{proof}
Consider the commutative diagram
\[
\begin{tikzcd}
\ovl{\E}_{/x_2} \ar[d, "q_2"] & \ovl{\E}_{/g^{\sharp}} \ar[l, "\simeq"']\ar[r]\ar[d, "q_{2,3}"] & \ovl{\E}_{/x_3}\ar[d, "q_3"] \\
\ovl{\E} \times_{\B} \ovl{\B}_{/px_2} & \E \times_{\B} \ovl{\B}_{/pg^{\sharp}}\ar[r]\ar[l, "\simeq"']\ar[d] & \E \times_{\B} \ovl{\B}_{/px_3}\ar[d]\\
& \ovl{\B}_{/pg^{\sharp}} \ar[r] & \ovl{\B}_{/px_3} \\
\end{tikzcd}
\]
in which the left pointing horizontal arrows are trivial fibrations by~\cite[Lemma 2.4.6]{GagnaHarpazLanariLaxLimits}. Now since the two faces of \(\rho\) leaning on \(g\) are thin we have that the 3-simplex \(\rho\) determines an arrow \(e\) in \(\ovl{\E}_{/g^{\sharp}}\), whose image in \(\ovl{\E}_{/x_2}\) is the arrow associated to \(\theta\). We conclude that \(\theta\) is left \(p\)-inner (resp.~left \(p\)-outer) if and only if \(e\) is \(q_{2,3}\)-cartesian (resp.~\(q_{2,3}\)-cocartesian). 
Now the image of \(e\) in \(\ovl{\E}_{/x_3}\) is the arrow associated to \(\tau\), and since \(\tau\) is assumed to be left \(p\)-inner (resp.~left \(p\)-outer) its associated arrow is \(q_3\)-cartesian (resp.~\(q_3\)-cocartesian) in \(\ovl{\E}_{/x_3}\). At the same time, the assumption that \(g\) is \(p\)-cartesian implies that the vertical external square on the right is homotopy cartesian by~\cite[Lemma 2.3.6]{GagnaHarpazLanariLaxLimits}, and since the bottom right square is also cartesian we get from the pasting lemma that the top right square is homotopy cartesian. We hence conclude that \(e\) is \(q_{2,3}\)-cartesian (resp.~\(q_{2,3}\)-cocartesian), as desired.
\end{proof}

\begin{rem}
	\label{uniqueness of cart 2-lifts}
Inner and outer lifts in a given weak fibration are unique up to equivalence (once they exist). For example, suppose we have two left \(p\)-inner triangles \(\alpha,\alpha'\) in \(\E\), whose restriction to \(\Lambda^2_1\) and whose image under \(p\) coincide. By Remark~\ref{rem:unwinding} (applied for \(\sig'\) and \(n=3\)) we can find a 3-simplex \(H\)
	of the form:
		\begin{nscenter}
			\begin{tikzpicture}[scale=1.5]
			\square{%
				/square/label/.cd,
				01={}, 12={$f$}, 23={$g$}, 03={$h$}, 02={$f$}, 13={$h'$},
				012={$=$}, 023={$\alpha$}, 013={$\beta$}, 123={$\alpha'$},
				0123={$H$},
				/square/arrowstyle/.cd,
				01={equal}, 012={phantom, description},
				/square/labelstyle/.cd,
				012={anchor=center}
			}
			\end{tikzpicture}
		\end{nscenter}
	where \(\beta\) lives over a degenerate triangle in \(\B\). 
	By Remark~\ref{rem:closed-whiskering} we then have that \(\beta\) is left \(p\)-inner as well, and hence thin by Proposition~\ref{cart over thin = thin}. We may consider \(H\) as exhibiting an equivalence between \(\alp\) and \(\alp'\): its left most leg is invertible, the two faces leaning on this leg are thin, and the remaining two faces are \(\alp\) and \(\alp'\). In a very similar manner, if \(\alp\) and \(\alp'\) are assumed instead to be right \(p\)-outer whose restriction to \(\Lambda^2_2\) and whose image under \(p\) coincide then we construct the same type of \(3\)-simplex \(H\), only that this time we will take its \(\Del^{\{0,1,3\}}\)-face to be degenerate  
	and extend \(H\) from its right outer horn using the right \(p\)-outerness of \(\alp\). Replacing \(p\) with \(p^{\op}\colon \E^{\op} \to \B^{\op}\) we get the analogous statements for the uniqueness of \emph{right} inner and \emph{left} outer lifts.
\end{rem}

\begin{rem}
	\label{rem:hom-cart-fib}
	Let \(p\colon \E \to \B\) be a weak fibration of \(\infty\)-bicategories. For given \(x,z \in \E\), if we base change the weak fibration \(\ovl{\E}_{/z} \to \E \times_{\B} \ovl{\B}_{/pz}\) along the map \(\{x\} \times_{\B} \ovl{\B}_{/pz} \to \E \times_{\B} \ovl{\B}_{/pz}\) then we get the map of maximally scaled simplicial sets, whose underlying map of simplicial sets is 
	\[p_{x,z}^{\triangleright}\colon \uMaptr_{\E}(x,z) \to \uMaptr_{\B}(px,pz),\] 
	which is a model for the induced map on mapping \(\infty\)-categories by~\cite[\S 2.3]{GagnaHarpazLanariEquiv}
    (see in particular \cite[Proposition 2.23]{GagnaHarpazLanariEquiv}). Since base change maps detect (co)cartesian edges it follows that for 
    a triangle \(\sig\colon \Del^2 \to \E\) such that \(\sig|_{\Del^{\{0,1\}}}\) is degenerate we have that: 
	\begin{itemize}
	\item
	If \(\sig\) is left \(p\)-inner 
	then the corresponding edge of \(\uMaptr_{\E}(x,z)\) is \(p^{\triangleright}_{x,z}\)-cartesian. 
	\item
	If \(\sig\) is left \(p\)-outer 
	then the corresponding edge of \(\uMaptr_{\E}(x,z)\) is \(p^{\triangleright}_{x,z}\)-cocartesian.
	\end{itemize}
\end{rem}

\subsection{2-Inner and 2-outer fibrations}

\begin{define}\label{def:inner-outer-fibration}
Let \(p\colon \E \to \B\) be a weak fibration and \(\sig \colon \Del^2 \to \B\) a triangle.
\begin{itemize}
\item
We say that \(\sig\) has a \ndef{sufficiently supply of left (resp.~right) \(p\)-inner lifts} if for every \(\rho\colon \Lam^2_1 \to \E\) lifting \(\sig|_{\Lam^2_1}\) 
there exists a left (resp.~right) \(p\)-inner triangle \(\tau\colon \Del^2 \to \E\) such that \(p\tau = \sig\) and \(\tau|_{\Lam^2_1} = \rho\).
\item
We say that \(\sig\) has a \ndef{sufficiently supply of left \(p\)-outer lifts} if for every \(\rho\colon \Lam^2_0 \to \E\) lifting \(\sig|_{\Lam^2_0}\) and such that \(\rho|_{\Del^{\{0,1\}}}\) is \(p\)-cocartesian, 
there exists a left \(p\)-outer triangle \(\tau\colon \Del^2 \to \E\) such that \(p\tau = \sig\) and \(\tau|_{\Lam^2_0} = \rho\).
\item
We say that \(\sig\) has a \ndef{sufficiently supply of right \(p\)-outer lifts} if for every \(\rho\colon \Lam^2_2 \to \E\) lifting \(\sig|_{\Lam^2_0}\) and such that \(\rho|_{\Del^{\{1,2\}}}\) is \(p\)-cartesian, 
there exists a right \(p\)-outer triangle \(\tau\colon \Del^2 \to \E\) such that \(p\tau = \sig\) and \(\tau|_{\Lam^2_2} = \rho\).
\end{itemize}
\end{define}

\begin{define}
\label{cart inner fib defi}
A weak fibration of scaled simplicial sets \(p\colon\E \rightarrow \B\) is said to be a \emph{2-inner fibration} if 
every triangle in \(\B\) admits both a sufficient supply of left \(p\)-inner lifts and a sufficient supply of right \(p\)-inner lifts. 

We say that \(p\) is a 2-inner cartesian fibration if it is both a 2-inner fibration and a cartesian fibration (in the sense of Definition~\ref{d:car-fibration}).
We say that \(p\) is a \emph{2-inner cocartesian fibration} if \(p^{\op}\colon \E^{\op} \to \B^{\op}\) is a 2-inner cartesian fibration.
\end{define}

\begin{define}
\label{cart outer fib defi}
A weak fibration of scaled simplicial sets \(p\colon\E \rightarrow \B\) is said to be a \emph{2-outer fibration} if the following conditions are satisfied:
\begin{enumerate}
\item 
Every triangle in \(\B\) admits both a sufficient supply of left \(p\)-outer lifts and a sufficient supply of right \(p\)-outer lifts. 
\item
The collection of left \(p\)-outer triangles in \(\E\) is closed under right whiskering and the collection of right \(p\)-outer triangles is closed under left whiskering. 
\end{enumerate}
We say that \(p\) is a \ndef{2-outer cartesian fibration} if it is both a 2-outer fibration and a cartesian fibration (in the sense of Definition~\ref{d:car-fibration}).
Dually, we say that \(p\) is a \ndef{2-outer cocartesian fibration} if \(p^{\op}\colon \E^{\op} \to \B^{\op}\) is a 2-outer cartesian fibration.
\end{define}

\begin{rem}\label{rem:2-base-change}
As for the previously introduced classes of fibrations, the class of 2-inner/outer (co)cartesian fibrations is readily seen to be closed under base change.
\end{rem}

\begin{define}\label{def:morphism}
Given 2-inner/outer (co)cartesian fibrations \(q\colon \D \rightarrow \A\) and \(p\colon \E \rightarrow\B\), a \ndef{morphism of 2-inner/outer (co)cartesian fibrations} from \(q\) to \(p\) is a commutative square 
\[\begin{tikzcd}
\D \ar[d,"q"{swap}] \ar[r,"g"] & \E \ar[d,"p"]\\
\A \ar[r,"f"]& \B
\end{tikzcd}\]
such that \(g\) sends \(q\)\nbd-(co)cartesian arrows to \(p\)-(co)cartesian arrows and left/right \(p\)-inner/outer triangles to left/right \(q\)-inner/outer triangles.
\end{define}

\begin{prop}\label{prop:2-inner is locally cartesian}
Let \(p\colon \E \to \B\) be a weak fibration of \(\infty\)-bicategories. Then
\begin{enumerate}
\item
If \(p\) is a 2-inner fibration then the induced map \(p_{x,y}\colon \uMaptr_{\E}(x,y) \to \uMaptr_{\B}(px,py)\) is a cartesian fibration of \(\infty\)-categories for every \(x,y \in \E\).
\item
If \(p\) is a 2-outer fibration then the induced map \(p_{x,y}\colon \uMaptr_{\E}(x,y) \to \uMaptr_{\B}(px,py)\) is a cocartesian fibration of \(\infty\)-categories for every \(x,y \in \E\).
\end{enumerate}
\end{prop}
\begin{proof}
The condition that \(p\) is a weak fibration implies that the induced map \(\ovl{\E}_{/y} \to \E \times_{\B} \ovl{\B}_{/py}\) is a 1-outer fibration (see Remark~\ref{r:base-change} and Example~\ref{ex:slice-projection}), and hence its base change \((\ovl{\E}_{/y})_x \to  (\ovl{\B}_{/py})_{px}\) is a 1-outer fibration as well, and in particular a weak fibration. This last map is between \(\infty\)-bicateories in which every triangle is thin, and hence its underlying map of simplicial sets \(p_{x,y}\colon \uMaptr_{\E}(x,y) \to \uMaptr_{\B}(px,py)\) is an inner fibration between \(\infty\)-categories. We also note that arrows in \(\uMaptr_{\B}(px,py)\) correspond to triangles \(\sig\colon \Del^2 \to \B\) such that \(\sig|_{\Del^{\{0,1\}}}\) is degenerate on \(px\) and \(\sig|_{\Del^{\{2\}}} = py\). More precisely, these are arrows from \(\sig|_{\Del^{\{0,2\}}}\) to \(\sig|_{\Del^{\{1,2\}}}\), considered as vertices in \(\uMaptr_{\B}(px,py)\). 

Now if \(p\) is a 2-inner fibration then for any choice of an edge \(g\colon x \to y\) in \(\E\) lifting \(\sig|_{\Del^{\{1,2\}}}\) there exists a left \(p\)-inner triangle \(\tau\) such that \(\tau|_{\Del^{\{0,1\}}}\) is degenerate, \(\tau|_{\Del^{\{1,2\}}} = g\) and \(p\tau = \sig\). By Remark~\ref{rem:hom-cart-fib} the triangle \(\tau\) determines a \(p_{x,y}\)-cartesian lift with target \(g\) of the arrow in \(\uMaptr_{\B}(px,py)\) determined by \(\sig\). Since \(g\) was arbitrary we conclude that \(p_{x,y}\) is a cartesian fibration. Similarly, if \(p\) is a 2-outer fibration we have that for any choice of an edge \(g\colon x \to y\) in \(\E\) lifting \(\sig|_{\Del^{\{0,2\}}}\) there exists a left \(p\)-outer triangle \(\tau\) such that \(\tau|_{\Del^{\{0,1\}}}\) is degenerate, \(\tau|_{\Del^{\{0,2\}}} = g\) and \(p\tau = \sig\). By Remark~\ref{rem:hom-cart-fib} the triangle \(\tau\) determined a \(p_{x,y}\)-cocartesian lift with domain \(g\) of the arrow in \(\uMaptr_{\B}(px,py)\) determined by \(\sig\), and so \(p_{x,y}\) is a now a cocartesian fibration.
\end{proof}

\subsection{Congruent triangles}

Our goal in the present subsection is to establish some preliminary results showing that for most questions about 2-outer (co)cartesian fibrations between \(\infty\)-bicategories, one may restrict attention to left/right outer triangles which are left/right-degenerate in the following sense:

\begin{define}\label{def:left-right-degenerate}
    We say that a triangle \(\sigma \colon \Delta^2 \to X\) is \ndef{left-degenerate}
    if the edge \(\sigma|_{\Delta^{\{0, 1\}}}\) of \(X\) is degenerate.
    The triangle \(\sigma\) is said to be \ndef{right-degenerate} if the edge \(\sigma|_{\Delta^{\{1,2\}}}\) is degenerate.
\end{define}

\begin{warning}
To avoid confusion, let us emphasize that left (or right) degenerate triangles in the sense of Definition~\ref{def:left-right-degenerate} are not necessarily themselves degenerate. This terminology is also used in~\cite{GarciaStern2Fib}.
\end{warning}

We will make use of the following construction.

\begin{define}\label{def:congruent}
Let \(X\) be a scaled simplicial set. 
Given a 3-simplex \(\rho\colon \Del^3 \to X\) of the form
		\begin{nscenter}
			\begin{tikzpicture}[scale=1.5]
			\square{%
				/square/label/.cd,
                0={$x_0$}, 1={$x_1$}, 2={$x_2$}, 3={$x_3$},
				01={}, 12={$f$}, 23={}, 03={}, 
				012={$\simeq$}, 023={$\sig$}, 013={$\sig'$}, 123={$\simeq$},
				0123={},
			}
			\end{tikzpicture}
		\end{nscenter}
we say that \(\rho\) exhibits \(\sig\) as \emph{left congruent} to \(\sig'\), and \(\sig'\) as \emph{right congruent} to \(\sig\). 
\end{define}

\begin{lemma}\label{lem:2-out-of-3}
Let \(p\colon \E \to \B\) be a weak fibration and let \(\sig,\sig'\) be two triangles in \(\E\) 
such that \(\sig\) is left congruent to \(\sig'\) via a 3-simplex \(\rho\) as in Definition~\ref{def:congruent}. Then the following holds:
\begin{enumerate}
\item
\(\sig\) is left/right \(p\)-inner if and only if \(\sig'\) is.
\item
If \(f\) is \(p\)-cocartesian and  \(\sig'\) is left \(p\)-outer then \(\sig\) is left \(p\)-outer.
\item
If \(f\) is \(p\)-cartesian and \(\sig\) is right \(p\)-outer then \(\sig'\) is right \(p\)-outer.
\item
If \(f\) is an equivalence then \(\sig\) is left/right \(p\)-outer if and only if \(\sig'\) is.
\end{enumerate}
\end{lemma}

\begin{proof}
The first three statements follows directly from the partial 2-out-of-3 properties elaborated in Remark~\ref{rem:closed-whiskering} together with the fact that thin triangles are always left and right \(p\)-inner and that thin triangles with \(p\)-cocartesian left leg are left \(p\)-outer, while thin triangles with \(p\)-cartesian right leg are right \(p\)-outer, see Remark~\ref{rem:thin-is-inner}. To prove the last claim, we note that if \(f\) is an equivalence then the 3-simplex \(\rho\) determines in particular an equivalence between the arrows associated to \(\sig\) and \(\sig'\) in \(\ovl{\E}_{{x_0/}}\), where \(x_0 = \sig(0) = \sig'(0)\), and hence each of these arrows is strongly cartesian with respect to the projection \(\ovl{\E}_{{x_0}/} \to \E \times_{\B} \B_{{px_0}/}\) if and only if the other is so (see Remark~\ref{rem:car-closed-equiv}). Similarly, \(\rho\) determines an equivalence between the arrows associated to \(\sig\) and \(\sig'\) in \(\ovl{\E}_{/{x_3}}\),
where \(x_3 = \sig(2) = \sig'(2)\), and hence each of these arrows is strongly cocartesian with respect to the projection \(\ovl{\E}_{/{x_3}} \to \E \times_{\B} \B_{/{px_3}}\) if and only if the other is so.
\end{proof}

\begin{lemma}\label{lem:second-step}
Let \(p\colon \E \to \B\) be a weak fibration between \(\infty\)-bicategories and let \(\sig,\sig'\) be two triangles in \(\B\) such that \(\sig\) is left congruent to \(\sig'\). Then the following holds:
\begin{enumerate}
\item
If the left leg of \(\sig'\) admits a sufficient supply of \(p\)-cocartesian lifts and \(\sig'\) admits a sufficient supply of left \(p\)-outer  
lifts then \(\sig\) admits a sufficient supply of left \(p\)-outer lifts.
\item
If the right leg of \(\sig\) admits a sufficient supply of \(p\)-cartesian lifts and \(\sig\) admits a sufficient supply of right \(p\)-outer 
lifts then \(\sig'\) admits a sufficient supply of right \(p\)-outer 
lifts.
\end{enumerate}
\end{lemma}
\begin{proof}
We prove the first claim, the second claim then follows by applying the first claim to \(\E^{\op}\) and switching the roles of \(\sig\) and \(\sig'\). Let \(\rho\colon \Del^3 \to \E\) be a 3-simplex as above exhibiting \(\sig\) as left congruent to \(\sig'\). In particular, \(\sig = \rho|_{\Del^{\{0,2,3\}}}$. We need to show that for every pair of arrows 
\[ \xymatrix{
y_0\ar[r]^{e_{0,3}}\ar[dr]_{e_{0,2}} & y_3 \\
& y_2  
}\]
of \(\E\) lifting \(\sig|_{\Del^{\{0,2\}}}\) and \(\sig|_{\Del^{\{0,3\}}}\), respectively, with \(e_{0,2}\) being \(p\)-cocartesian, there exists a left \(p\)-outer triangle 
	\[
		\begin{tikzcd}
		y_0\ar[dr,"e_{0,2}"{description,name=k}]\ar[r, "e_{0,3}"{description,name=k'}]& y_3\\
		\ar[Rightarrow, shorten <= 1.5ex, shorten >= 2.5ex, from=k',to=2-2]& x_2\ar[u,"e_{2,3}"']
		\end{tikzcd}
    \]
lifting \(\sig\). 
Now by assumption the left leg of \(\sig'\) admits a \(p\)-cocartesian lift \(e_{0,1}\colon y_0 \to y_1\).
Since \(\sig'\) is assumed to have a sufficient supply of left \(p\)-outer lifts we may now find a left \(p\)-outer lift \(\tau'\) of \(\sig'\), depicted as
	\[
		\begin{tikzcd}
		y_0\ar[r, "e_{0,3}"{description,name=k'}]\ar[d,"e_{0,1}"']& y_3\\
		y_1 \ar[Rightarrow, shorten <= 1.5ex, shorten >= 2.5ex, from=k',to=2-1]\ar[ur,"e_{1,3}"']& \ .
		\end{tikzcd} 
    \]
At the same time, since \(\rho|_{\Del^{\{0,1,2\}}}\) is thin and \(e_{0,1}\) is \(p\)-cocartesian the pair \(e_{0,1},e_{0,2}\) extends to a thin triangle \(\theta\colon \Del^{\{0,1,2\}} \to \E\) such that \(p\theta = \rho|_{\Del^{\{0,1,2\}}}\). Set \(e_{1,2} = \theta|_{\Del^{\{1,2\}}}\). By the 2-out-of-3 property for \(p\)-cocartesian edges (e.g., the dual of~\cite[Lemma 2.3.8]{GagnaHarpazLanariLaxLimits}) we get that \(e_{1,2}\) is \(p\)-cocartesian. Since \(\rho|_{\Del^{\{1,2,3\}}}\) is thin we may now lift it to a triangle \(\eta\colon \Del^{\{1,2,3\}} \to \E\) extending \(e_{1,2}\) and \(e_{1,3}\).  
The triangles \(\tau',\theta\) and \(\eta\) now glue to give a map \(\alp\colon \Lam^3_1 \to \E\) lifting \(\rho|_{\Lam^3_1}\), where \(\alp\) sends \(\Del^{\{0,1,2\}}\) and \(\Del^{\{1,2,3\}}\) to thin triangles by construction. Since \(p\) is a weak fibration we may extend \(\alp\) to a map \(\bet \colon \Del^3 \to \E\) lifting \(\rho\), so that \(\tau := \bet|_{\Del^{\{0,2,3\}}}\) gives in particular a triangle lifting \(\sig\) and extending \(e_{0,2},e_{0,3}\). The 3-simplex \(\bet\) now exhibits \(\tau\) as left congruent to \(\tau'\) and hence \(\tau\) is left \(p\)-outer by Lemma~\ref{lem:2-out-of-3}(2).
\end{proof}

\begin{rem}\label{rem:left-degenerate}
In the proof of Lemma~\ref{lem:second-step}(1) we have a complete freedom in choosing the \(p\)-cocartesian lift \(e_{0,1}\) of \(\sig'|_{\Del^{\{0,1\}}}\). We may consequently slightly weaken the assumption on \(\sig'\): it suffices to assume that for every choice of lift of \(e_{1,3}\) of \(\sig'|_{\Del^{\{1,2\}}}\) there exists some left \(p\)-outer lift \(\tau'\) of \(\sig'\) such that \(\tau|_{\Del^{\{1,2\}}} = e_{1,3}\) ad \(\tau|_{\Del^{\{0,1\}}}\) is \(p\)-cartesian (as opposed to assuming this for any choice of \(p\)-cocartesian lift of \(\tau|_{\Del^{\{0,1\}}}\)).
For example, if \(\sig'\) is left-degenerate then it suffices to assume that it has a sufficient supply of \emph{left-degenerate} left \(p\)-outer lifts (that is, for each choice of a lift of \(\sig'|_{\Del^{\{1,2\}}}\)). A similar statement holds for the right \(p\)-outer case of Lemma~\ref{lem:second-step}(2). 
\end{rem}

\begin{lemma}\label{lem:reduce}
Let \(\B\) be an \(\infty\)-bicategory. Then the following holds:
\begin{enumerate}
\item
Any triangle \(\sig\) in \(\B\) is left congruent to a left-degenerate triangle \(\sig'\). 
\item
Any triangle \(\sig'\) in \(\B\) is right congruent to a right-degenerate triangle \(\sig\). 
\end{enumerate}
\end{lemma}
\begin{proof}
We prove the first claim. The second claim then follows by applying the first statement to \(\B^{\op}\) and switching the roles of \(\sig\) and \(\sig'\). Let us depict \(\sig\) as
	\[
		\begin{tikzcd}
		x_0\ar[dr,"f"{description,name=k}]\ar[r, "h"{description,name=k'}]& x_3\\
		\ar[Rightarrow, shorten <= 1.5ex, shorten >= 2.5ex, from=k',to=2-2]& x_2\ar[u,"g"']
		\end{tikzcd}
    \]
Let \(K \subseteq \Del^3\) be the simplicial subset spanned by the faces \(\Del^{\{0,2,3\}}\) and \(\Del^{\{0,1,2\}}\) and let \(\rho\colon K \to \B\) be the map which sends \(\Del^{\{0,2,3\}}\) to \(\sig\) and \(\Del^{\{0,1,2\}}\) to the degenerate triangle whose left leg is degenerate and whose other two legs are \(f\). We may then visualise \(\rho\) as
	\[
		\begin{tikzcd}
		x_0\ar[dr,"f"{description,name=k}]\ar[r, "h"{description,name=k'}]\ar[d,equal] & x_3\\
		x_1\ar[r,"f"'] \ar[phantom, "=",from=k,to=2-1]\ar[Rightarrow, shorten <= 1.5ex, shorten >= 2.5ex, from=k',to=2-2]& x_2 \ . \ar[u,"g"'] 
		\end{tikzcd} 
    \]
Now since \(\B\) is an \(\infty\)-bicategory we may extend \(\rho\) to a map \(\rho'\colon \Lam^3_2 \to \B\) which sends the triangle \(\Del^{\{1,2,3\}}\) to a thin triangle, and then proceed to extend \(\rho'\) to a full 3-simplex \(\Del^3 \to \B\), which we can depict as
		\begin{nscenter}
			\begin{tikzpicture}[scale=1.5]
			\square{%
				/square/label/.cd,
                0={$x_0$}, 1={$x_1$}, 2={$x_2$}, 3={$x_3$},
				01={}, 02={$f$},12={$f$}, 23={$g$}, 03={$h$}, 
				012={$=$}, 023={$\sig$}, 013={$\sig'$}, 123={$\simeq$},
				0123={},
				/square/arrowstyle/.cd,
				01={equal}, 012={phantom, description},
				/square/labelstyle/.cd,
				012={anchor=center}
			}
			\draw (3.3,-0.75) node{$.$};
			\end{tikzpicture}
		\end{nscenter}
This 3-simplex then exhibits \(\sig\) as left congruent to \(\sig'\), and \(\sig'\) is left-degenerate, as desired.
\end{proof}

\begin{cor}\label{cor:only-degenerate}
Let \(p\colon \E \to \B\) be a weak fibration. If every left-degenerate triangle in \(\B\) has a sufficient supply of left \(p\)-outer lifts then every triangle in \(\B\) has a sufficient supply of left \(p\)-outer lifts. Similarly, if every right-degenerate triangle in \(\B\) has a sufficient supply of right \(p\)-outer lifts then every triangle in \(\B\) has a sufficient supply of right \(p\)-outer lifts.
\end{cor}
\begin{proof}
Combine Lemma~\ref{lem:reduce} with Lemma~\ref{lem:second-step} using the fact that any degenerate arrow in \(\B\) has a sufficient supply of (co)cartesian lifts.
\end{proof}

Using Remark~\ref{rem:left-degenerate} we may also obtain the following strengthening of Corollary~\ref{cor:only-degenerate}:
\begin{cor}\label{cor:only-degenerate-enhanced}
Let \(p\colon \E \to \B\) be a weak fibration. If every left-degenerate triangle in \(\B\) has a sufficient supply of left-degenerate left \(p\)-outer lifts then every triangle in \(\B\) has a sufficient supply of left \(p\)-outer lifts. Similarly, if every right-degenerate triangle in \(\B\) has a sufficient supply of right-degenerate right \(p\)-outer lifts then every triangle in \(\B\) has a sufficient supply of right \(p\)-outer lifts.
\end{cor}

\subsection{Homotopy invariance of fibrations}

Our goal in the present section is to prove the following homotopy invariance property for (co)cartesian fibrations.
\begin{prop}\label{prop:invariance}
Let
\[
\begin{tikzcd}
\D \ar[d,"q"{swap}] \ar[r,"\simeq"] & \E \ar[d,"p"]\\
\A \ar[r,"\simeq"]& \B
\end{tikzcd}
\]
be a commutative diagram of \(\infty\)-bicategories whose vertical maps are both bicategorical fibrations and whose horizontal maps are bicategorical equivalences. Then \(p\) is an 2-inner/outer (co)cartesian fibration if and only if \(q\) is. In addition, an edge in \(\D\) is \(q\)-(co)cartesian if and only if its image in \(\E\) is \(p\)-(co)cartesian, and similarly a triangle in \(\D\) is left/right \(q\)-inner/outer if and only if its image in \(\E\) is so with respect to \(p\).
\end{prop}

The proof of Proposition~\ref{prop:invariance} will require a couple of lemmas, and will be given at the end of the section.

\begin{lemma}\label{lem:invariance-1}
Let
\[
\begin{tikzcd}
\D \ar[d,"q"{swap}] \ar[r,"\simeq"] & \E \ar[d,"p"]\\
\A \ar[r,"\simeq"]& \B
\end{tikzcd}
\]
be a commutative diagram of \(\infty\)-bicategories whose vertical maps are both bicategorical fibrations and whose horizontal maps are bicategorical equivalences. Then an arrow in \(\D\) is \(q\)-(co)cartesian if and only if its image in \(\E\) is \(p\)-(co)cartesian.
\end{lemma}
\begin{proof}
By~\cite[Proposition 2.3.7]{GagnaHarpazLanariLaxLimits} we may replace the property of being cartesian with that of being weakly (co)cartesian. The desired claim now follows from the characterization~\cite[Proposition 2.3.3]{GagnaHarpazLanariLaxLimits} of weakly (co)cartesian arrows in terms of mapping spaces.
\end{proof}

In any weak fibration \(p\colon \E \to \B\) between \(\infty\)-bicategories, the collection of \(p\)-(co)cartesian arrows is closed under equivalences in the arrow category, see~\cite[Remark 2.3.12]{GagnaHarpazLanariLaxLimits}. We now show that a similar property holds for inner/outer triangles:

\begin{lemma}\label{lem:invariance-triangle}
Let \(p\colon \E \to \B\) be a weak fibration between \(\infty\)-bicategories, and let \(H\colon \Del^1_{\flat} \times \Del^2_{\flat} \to \E\) be levelwise invertible natural transformation between triangles. Then \(\sig_0 := H|_{\Del^{\{0\}} \times \Del^2}\) is left/right \(p\)-inner/outer if and only if \(\sig_1 := H|_{\Del^{\{1\}} \times \Del^2}\) is so.
\end{lemma}
\begin{proof}
To fix ideas we prove the left inner and left outer cases, the right inner and right outer cases then follow by replacing \(p\) with opposite. In addition, it will suffice to prove that \(\sig_0\) is left \(p\)-inner/outer as soon as \(\sig_1\) is such, since we can get the other direction by replacing \(H\) by the inverse equivalence. We hence assume that \(\sig_1\) is left \(p\)-inner/outer.

For \(i=0,1,2\) let \(\rho_i\colon \Del^3 \to \Del^1 \times \Del^2\) be the \(3\)-simplex given by 
\[
\rho_i(j) = \begin{cases}
(0,j)& j\leq i\\
(1,j-1) & j > i
\end{cases}
\]
We may then write \(H\rho_0\) as
		\begin{nscenter}
			\begin{tikzpicture}[scale=1.5]
			\square{%
				/square/label/.cd,
				01={$\simeq$}, 12={}, 23={}, 03={},
				012={$\simeq$}, 023={$\tau$}, 013={$\simeq$}, 
				123={$\sig_1$},
				/square/arrowstyle/.cd,
				012={phantom, description},
				013={phantom, description},
				/square/labelstyle/.cd,
				012={anchor=center}
			}
			\end{tikzpicture}
		\end{nscenter}
for some triangle \(\tau\). Since \(\sig_1\) is left \(p\)-inner/outer we get from the 2-out-of-3 properties of Remark~\ref{rem:closed-whiskering} that \(\tau\) is left \(p\)-inner/outer (in the outer case, we also point out that the triangle \(H\rho_0|_{\Del^{\{0,1,3\}}}\) is left \(p\)-outer by Remark~\ref{rem:thin-is-inner} and Remark~\ref{rem:equiv-is-car} since it is thin and its left leg is an equivalence). 
We now write \(H\rho_1\) as
		\begin{nscenter}
			\begin{tikzpicture}[scale=1.5]
			\square{%
				/square/label/.cd,
				01={}, 12={$\simeq$}, 23={}, 03={},
				012={$\simeq$}, 023={$\tau$}, 013={$\tau'$}, 
				123={$\simeq$},
				/square/arrowstyle/.cd,
				012={phantom, description},
				123={phantom, description},
				/square/labelstyle/.cd,
				012={anchor=center}
			}
			\end{tikzpicture}
		\end{nscenter}
		for some triangle \(\tau'\). Since \(\tau\) is left \(p\)-inner/outer we now get from Points (1) and (4) of Lemma~\ref{lem:2-out-of-3} that \(\tau'\) is left \(p\)-inner/outer.
Finally, we may write \(H\rho_2\) as
		\begin{nscenter}
			\begin{tikzpicture}[scale=1.5]
			\square{%
				/square/label/.cd,
				01={}, 12={}, 23={$\simeq$}, 03={},
				023={$\simeq$}, 012={$\sig_0$}, 013={$\tau'$}, 
				123={$\simeq$},
				/square/arrowstyle/.cd,
				023={phantom, description},
				123={phantom, description},
				/square/labelstyle/.cd,
			}
			\end{tikzpicture}
		\end{nscenter}
		and so by Lemma~\ref{lem:closed-whiskering} we conclude that \(\sig_0\) is left \(p\)-inner/outer as well, as desired.
\end{proof}

\begin{proof}[Proof of Proposition~\ref{prop:invariance}]
We first note that since both vertical arrows are fibrant and cofibrant in the arrow category (with respect to the projective model structure), the existence of a levelwise equivalence from \(q\) to \(p\) implies the existence of a levelwise equivalence from \(p\) to \(q\). It will hence suffice to show that if \(p\) is a 2-inner/outer (co)cartesian fibration then so is \(q\). Now since weak equivalences between fibrant objects are preserved under base change along fibrations it follows from the 2-out-of-3 property that the map \(\D \to \E \times_{\B} \A\) is an equivalence of \(\infty\)-bicategories. The map \(\E' := \E \times_{\B} \A \to \A\) is then again a 2-inner/outer (co)cartesian fibration (Remark~\ref{rem:2-base-change}).
We now factor the weak equivalence \(\D \xrightarrow{\simeq} \E'\) as a composite
\[ \D \xrightarrow{\simeq} \D' \xrightarrow{\simeq} \E' \]
where the first map is a trivial cofibration and the second a trivial fibration. Then \(\D' \to \E' \to \A\) is a composite of a 2-inner/outer (co)cartesian fibration and a trivial fibration of scaled simplicial sets, and is hence itself again a 2-inner/outer (co)cartesian fibration. On the other hand, since \(\D \to \D'\) is a trivial cofibration
and \(\D \to \A\) is in particular a bicategorical fibration we may solve the lifting problem
\[ \xymatrix{
\D \ar@{=}[r] \ar[d] & \D \ar[d] \\
\D' \ar[r]\ar@{-->}[ur] & \A \\
}\]
so that we obtain a retract diagram of arrows
\[ \xymatrix{
\D \ar[d]^{q}\ar[r]^{i} & \D' \ar[d]^{q'}\ar[r]^{r} & \D \ar[d]^{q} \\
\A \ar@{=}[r] & \A \ar@{=}[r] & \A \ .
}\]
Now by Lemma~\ref{lem:invariance-1} the map \(\D' \to \D\) preserves (co)cartesian arrows over \(\A\), and hence the fact that \(\D'\) has a sufficient supply of \(q'\)-(co)cartesian edges implies that \(\D\) has a sufficient supply of \(q\)-(co)cartesian edges. It is thus left to show that \(\D\) also has a sufficient supply of left/right \(q\)-inner/outer triangles.
Solving a lifting problem of the form
\[ \xymatrix{
[\Del^1 \times \D] \coprod_{\partial \Del^1 \times \D'} [\partial \Del^1 \times \D] \ar[r] \ar[d] & \D' \ar[d]^{q'} \\
\D' \times \Del^1 \ar[r]\ar@{-->}[ur]^{H} & \A \\
}\]
we obtain a natural transformation \(H\) over \(\A\) from \(i \circ r\colon \D' \to \D'\) to the identity \(\id_{\D'}\),
whose restriction to \(\D\) is constant on \(i\).
Since \(i\) is a trivial cofibration it is in particular essentially surjective, and so the natural transformation \(H\) is levelwise invertible.
We now prove the inner case. Suppose given a triangle \(\sig\colon \Del^2 \to \A\) and a lift \(\rho\colon \Lam^2_1 \to \D\) of \(\sig|_{\Lam^2_1}\).
Since \(q'\) is a left/right inner fibration we may extend \(i\rho\) to a left/right \(q'\)-inner triangle \(\tau\colon \Del^2 \to \D'\) lifting \(\sig\).
Evaluating the levelwise invertible natural transformation \(H\) at \(\tau\) we obtain an equivalence \(ir\tau \Rightarrow \tau \)
covering the identity transformation on \(\tau\) and restricting to the identity transformation from \(i\rho\) to itself on \(\Lam^2_1\).
By Lemma~\ref{lem:invariance-triangle} we deduce that \(ir\tau\) is also \(q'\)-inner. Since \(i\) admits a retraction
and by the explicit description of left/right inner triangles in terms of lifting properties as in Remark~\ref{rem:unwinding}
we then deduce that \(r\tau\) itself is a left/right \(q\)-inner extension of \(\rho\) lifting \(\sig\).
In the outer case, the argument for the existence of a sufficient supply of left/right \(q\)-outer triangles is completely analogous,
but one also needs to show the closure of \(q\)-outer triangles under whiskering as required in Definition~\ref{cart outer fib defi}(2).
But this again follows from the corresponding property for \(q'\) and the fact that \(i\colon \D \to \D'\) detects left/right outer triangles,
as it admits a retraction over \(\A\).
\end{proof}

\subsection{Equivalences of cartesian fibrations}

This section is devoted to the proof of a fibre-wise criterion to test equivalences of fibrations.

\begin{thm}
	\label{fibrewise eq}
	Consider a morphism of 2-inner/outer (co)cartesian fibrations (in the sense of Definition~\ref{def:morphism}) given by a commutative diagram of the form
	\begin{equation}
	\label{fib eq square}
	\begin{tikzcd}
	\E \ar[d,"p"{swap}] \ar[r,"r"]&\E' \ar[d,"q"]\\
	\B \ar[r,"f"] & \B' \ .
	\end{tikzcd}
	\end{equation}
	where \(p,q\) are 2-inner/outer cartesian fibrations between \(\infty\)-bicategories. 
    Suppose that \(f\) is an equivalence of \(\infty\)-bicategories. Then \(r\) is an equivalence if and only if the induced map \(r_b\colon \E_b \to \E'_{f(b)}\) on the level of fibres is an equivalence of \(\infty\)-bicategories for all \(b \in \B\). 
\end{thm}

Before proving this, we need a preliminary result, generalizing 
\cite[Proposition 2.4.4.2]{HTT}. Here we make use of the results from~\cite{GagnaHarpazLanariLaxLimits} involving the slice \(\infty\)-bicategories associated with marked-scaled simplicial sets, see \S\ref{sec:join-and-slice} and Notation~\ref{not:slice-over-arrow}.

\begin{lemma}
	\label{characterization of fibres}
	Let \(p\colon \C \to \D\) be a 2-inner/outer cartesian fibration of \(\infty\)-bicategories.
    Let \(x,y \in \C\) be two objects, \(\bar{e}\colon px \to py\) an arrow between their images in \(\D\), and \(e\colon x' \to y\) a \(p\)-cartesian lift of \(\bar{e}\). 
    Then the (homotopy) fibre of the map 
    \[\phi\colon \uMaptr_{\C}(x,y)\to \uMaptr_{\D}(px,py)\]
    at \(\bar{e}\) is naturally equivalent to the mapping
    \(\infty\)-category \(\uMaptr_{\C_{px}}(x,x')\), where \(\C_{px}\) denotes the (homotopy) fibre of \(p\) over \(px\).
\end{lemma}

\begin{proof}
To begin with we note that  
the map \(\phi\) is a cartesian (or cocartesian in the outer case) fibration of \(\infty\)-categories by Proposition~\ref{prop:2-inner is locally cartesian} 
and hence the homotopy fibre in question is equivalent to the corresponding strict fibre.
    We now consider the following diagram 
\[\xymatrix{
    &\trbis{\ovl{\C}}{e^{\sharp}}\ar[dl]_{\simeq}\ar[dr]^{\simeq} & \\
\trbis{\ovl{\C}}{x'}\times_{\trbis{\ovl{\D}}{px'}}\trbis{\ovl{\D}}{\bar{e}^{\sharp}} &&     \trbis{\ovl{\C}}{y}\times_{\trbis{\ovl{\D}}{py}}\trbis{\ovl{\D}}{\bar{e}^{\sharp}}
}\] 
where the left diagonal map is a trivial fibration by \cite[Lemma 2.4.6]{GagnaHarpazLanariLaxLimits} and the right diagonal map is a trivial fibration by~\cite[Lemma 2.3.6]{GagnaHarpazLanariLaxLimits} and the assumption that \(e\) is \(p\)-cartesian (and is hence in particular weakly \(p\)-cartesian, in the sense of loc.~cit.).
Taking fibres over \((x,s_1(\bar{e}))\in \C\times_{\D}\trbis{\D}{\bar{e}^{\sharp}}\) we thus get a zig-zag of equivalences 
\[\xymatrix{
    &\ovl{\C}_{/e^{\sharp}} \times_{\C\times_{\D}\trbis{\D}{\bar{e}^{\sharp}}}\{(x,s_1(\bar{e}))\}  \ar[dl]_{\simeq}\ar[dr]^{\simeq} & \\
\uMaptr_{\C_{p(x)}}(x,x')_{\sharp}  &&      \phi^{-1}(\bar{e})_{\sharp}
}\]
relating the fibre of \(\phi\) over \(\bar{e}\) to \(\uMaptr_{\C_{p(x)}}(x,x')\), as desired.
\end{proof}

\begin{proof}[Proof of Theorem \ref{fibrewise eq}]
We prove the 2-inner cartesian case. The proof for the remaining three variance flavours proceeds in exactly the same manner.
	Let us begin by proving the ``only if'' direction of the statement. Every object in the square \eqref{fib eq square} is fibrant, and the vertical maps are fibrations by Proposition~\ref{cart are fib}, therefore every pullback along these maps is automatically a homotopy pullback. Consider the following commutative cube:
	\[
	\begin{tikzcd}[row sep=scriptsize, column sep=scriptsize]
	& \E_b \arrow[dl] \arrow[rr,"r_b"] \arrow[dd] & & \E'_{f(b)}\arrow[dl] \arrow[dd] \\
	\E \arrow[rr," \ \ \ r", crossing over] \arrow[dd,"p"{swap}]  & &  \E' \ar[dd,"\ q"{pos=0.7}, crossing over] \\
	& \Delta^0 \arrow[dl,"\{b\}"{swap}] \arrow[rr] & & \Delta^0 \arrow[dl,"\{f(b)\}"] \\
	\B \arrow[rr,"f"] & & \B' \arrow[from=uu, crossing over]\\
	\end{tikzcd}\]
	The front face and the side faces are homotopy pullbacks by assumption, so that the back one must also be such. Since the bottom horizontal map is an equivalence, the map \(r_b\) must be an equivalence as well.
	
	Assume now that \(r\) is a fibrewise equivalence, and let us prove it is essentially surjective on objects and fully faithful. To see that \(r\) is essentially surjective, factor it as a composite \(\E \to \E' \times_{\B'}\B \to \E'\), where the first map is essentially surjective since \(r_b\colon \E_b \to \E'_{fb}\) is essentially surjective for every \(b \in \B\) and the second is essentially surjective because \(f\colon \B \to \B'\) is so. 
	Concerning fully-faithfulness, we consider the following commutative square of \(\infty\)-categories:
	\begin{equation}
	\label{l}
	\begin{tikzcd}
	\uMaptr_{\E}(x,y)\ar[d,"p"{swap}] \ar[r,"r"]& \uMaptr_{\E'}(rx,ry) \ar[d,"q"]\\
	\uMaptr_{\B}(px,py) \ar[r,"f"] & \uMaptr_{\B'}(qrx,qry)
	\end{tikzcd}
	\end{equation}
	Here, we have used the fact that \(qr=fp\), and we denoted the induced action between the \(\infty\)-categories 
    of morphisms with the same letter as that of the map between the relevant \(\infty\)-bicategories.
    Now, the vertical maps are cartesian fibrations of \(\infty\)-categories by Proposition~\ref{prop:2-inner is locally cartesian}, 
    and the bottom horizontal map is an equivalence since \(f\) is fully-faithful. 
    Therefore, the top horizontal map is an equivalence if and only if the square \eqref{l} is a homotopy pullback, 
    which happens precisely if the maps induced between the (strict) fibres of the vertical maps are equivalences. 
    Thanks to Lemma \ref{characterization of fibres}, the fibre of the left-hand side map over \(\bar{e}\colon p(x)\to p(y)\) 
    coincides with \(\uMaptr_{\E_{px}}(x,x')\) for some \(p\)-cartesian 1-simplex \(e\colon x' \to y\) that lifts \(\bar{e}\). 
    Since \(r\) is assumed to preserve cartesian edges we have that \(re\colon rx'\to ry\) is a \(q\)-cartesian lift of \(f\bar{e}\), 
    and so the fibre of \(q\) over \(f(\bar{e})\) is given by \(\uMaptr_{\E'_{qrx}}(rx,rx')\). 
    We can now finish the proof by observing that the induced map 
    \[\uMaptr_{\E_{px}}(x,x')\to \uMaptr_{\E'_{qrx}}(rx,rx')\] 
    is an equivalence of \(\infty\)-categories, since \(\E_{px}\to \E'_{fpx}\simeq \E'_{qrx}\) is assumed to be an equivalence of \(\infty\)-bicategories.
\end{proof}

\section{The domain projection} 

In this section 
we analyse a key example of a 2-outer cartesian fibration: the domain projection \(\mathrm{d}\colon\RMap(\Delta^1,\C)\rightarrow \C\)
induced by the inclusion \(\{0\}\hookrightarrow\Delta^1\), where \(\RMap(\Del^1,\C)\) is the \(\infty\)-bicategory whose objects are
the arrows in \(\C\) and whose morphisms are the lax-commutative squares. It is characterized by the property that for every scaled simplicial set \((K,T_K)\), maps \((K,T_K) \to \RMap(\Del^1,\C)\) correspond to maps of scaled simplicial sets 
\[ \Del^1_{\flat} \otimes (K,T_K) \to \C ,\]
where \(\otimes\) denotes the \emph{Gray product} of scaled simplicial sets, see~\cite[Definition 2.1]{GagnaHarpazLanariGrayLaxFunctors}. Explicitly, \(\Del^1_{\flat} \otimes (K,T_K)\) is the scaled simplicial set whose underlying simplicial set is the cartesian product \(\Del^1 \times K\), and where a triangle \(\sig\colon \Del^2 \to \Del^1 \times K\) is thin if and only if its image in \(K\) belongs to \(T_K\) and either \(\sig|_{\Del^{\{0,1\}}}\) maps to a degenerate edge in \(K\) or \(\sig|_{\Del^{\{1,2\}}}\) maps to a degenerate edge in \(\Del^1\). Here, we have used the fact that all the triangles in \(\Del^1_{\flat}\) are thin, otherwise the additional condition of projecting to a thin triangle in the first factor would have been necessary.

\begin{rem}
 \label{rem:dom_projection_weak_fibration}
 By the main result of~\cite{GagnaHarpazLanariGrayLaxFunctors} the functor \(-\otimes -\) is a left Quillen bifunctor. It then follows that for any \(\infty\)-bicategory \(\C\) the domain projection
	\[
		\mathrm{d}\colon\RMap(\Delta^1,\C)\rightarrow \C.
	\]
is a bicategorical fibration of \(\infty\)-bicategories, 
and in particular a weak fibration (see Remark~\ref{rem:bicategorical-is-weak}).
\end{rem}

Our argument to show that the domain fibration is a 2-outer cartesian fibrations is based on the following extension lemma:
\begin{lemma}\label{lem:technical-part}
Let \(\C\) be an \(\infty\)-bicategory, and for \(n \geq 2\) suppose given an extension problem of the form
\[ \xymatrix{
\big([\Del^1 \times \Lam^n_n] \displaystyle\mathop{\coprod}_{\partial \Del^1\times \Lam^n_n} [\partial \Del^1\times \Del^n],T'\big) \ar[d]\ar[r]^-{\rho} & \C  \\
\big(\Del^1 \times \Del^n,T\big)\ar@{..>}[ur]  & \ .
}\]
where 
\[T = \{\Del^{\{(0,n-1),(0,n),(1,n)\}},\Del^{\{(0,0),(1,n-1),(1,n)\}}\}\] 
and \(T'\) its restriction to the top left corner. 
If \(\rho\) sends \(\Del^{\{1\}} \times \Del^{\{n-1,n\}}\) to an invertible edge in \(\C\) then the dotted extension exists.  
\end{lemma}
\begin{proof}
We define a sequence of scaled maps
	\[
		\tau_1,\ldots,\tau_{n-1}\colon (\Delta^n,\Delta^{\{0,n-1,n\}}) \rightarrow \bigl(\Delta^1\times \Delta^n, T\bigr)
	\]
	in the following manner:
	\[
	\tau_i(j) = \begin{cases}
	(0,j)& j<i\\
	(1,j) & j \geq i
	\end{cases}
	\]
	Set
	\(K_0 := \big([\Del^1 \times \Lam^n_n] \coprod_{\partial \Del^1\times \Lam^n_n} [\partial \Del^1\times \Del^n],T'\big)\)
and \(K_{i} := K_{i-1}\cup  \tau_{i}\), for \(1\leq i \leq n-1\), so that we have pushout diagrams of the form:
	\[
	\begin{tikzcd}
	(\Lambda^n_n, \{\Delta^{\{0,n-1,n\}}\} )\ar[d]\ar[r] &K_{i-1}\ar[d]\\
	(\Delta^n, \{\Delta^{\{0,n-1,n\}}\}) \ar[r,"\tau_i"]&K_i
	\end{tikzcd}
	\]
	for every \(1\leq i \leq n-1\). These all have the property that the edge \(\Delta^{\{n-1,n\}}\) is mapped to an equivalence in \(\C\). 
    By~\cite[Corollary 2.3.10]{GagnaHarpazLanariLaxLimits}, all these maps admit lifts against \(\C\), and so we can extend the map \(\rho\) to a map \(\rho'\colon K_{n-1} \rightarrow \C\).  
	Consider now the \((n+1)\)-simplices
	\(\sigma_0,\ldots,\sigma_{n}\colon \Delta^{n+1}\rightarrow \Delta^1\times \Delta^n\) defined as follows:
	\[
		\sigma_i(j)=
			\begin{cases}
				(0,j) & j\leq i\\
				(1,j-1) & j>i.
			\end{cases}
	\]
	Observe that \(\sigma_{0},\ldots,\sigma_{n-1}\) can be promoted to maps \[(\Delta^{n+1},\{\Delta^{\{0,n,n+1\}}\})\to \bigl(\Delta^1\times \Delta^n, T\bigr)\] whereas \(\sigma_n\) can be promoted to a map \[(\Delta^{n+1}, \{\Delta^{\{n-1,n,n+1\}}\})\to \bigl(\Delta^1\times \Delta^n, T\bigr)\]
	We now set \(K_{n+k} = K_{n-1+k}\cup \sigma_{n-k}\) for \(0\leq k \leq n\), and observe that \(K_{2n}=\bigl(\Delta^1\times \Delta^n, T\bigr)\).
We then have a pushout diagram of the form:
	\begin{equation*}
	\begin{tikzcd}
	(\Lambda^{n+1}_n, \{\Delta^{\{n-1,n,n+1\}}\}|_{\Lambda^{n+1}_n}) \ar[d]\ar[r] &K_{n-1}\ar[d]\\
	(\Delta^{n+1}, \{\Delta^{\{n-1,n,n+1\}}\})\ar[r,"\sigma_n"]&K_{n}
	\end{tikzcd}
	\end{equation*}
	and, for every \(k>0\), there are pushout diagrams of the form:
		\begin{equation*}
	\begin{tikzcd}
	(\Lambda^{n+1}_{n+1}, \{\Delta^{\{0,n,n+1\}}\})\ar[d]\ar[r] &K_{n-1+k}\ar[d]\\
	(\Delta^{n+1}, \{\Delta^{\{0,n,n+1\}}\})\ar[r,"\sigma_{n-k}"]&K_{n+k}
	\end{tikzcd}
	\end{equation*}
where the corresponding image of the edge \(\Delta^{\{n,n+1\}}\) in \(\C\) is an equivalence. As before, this is enough to prove the existence of an extension to \(\C\).
\end{proof}

\begin{rem}\label{rem:dual-form}
Applying Lemma~\ref{lem:technical-part} to \(\C^{\op}\) one obtains the following dual form of its statement, which we spell out for the convenience of the reader:
suppose given an extension problem of the form
\[ \xymatrix{
\big([\Del^1 \times \Lam^n_0] \displaystyle\mathop{\coprod}_{\partial \Del^1\times \Lam^n_0} [\partial \Del^1\times \Del^n],T'\big) \ar[d]\ar[r]^-{\rho} & \C  \\
\big(\Del^1 \times \Del^n,T\big)\ar@{-->}[ur]  & \ .
}\]
where 
\[T = \{\Del^{\{(0,0),(1,0),(1,1)\}},\Del^{\{(0,0),(0,1),(1,n)\}}\}\] 
and \(T'\) its restriction to the top left corner. 
If \(\rho\) sends \(\Del^{\{1\}} \times \Del^{\{0,1\}}\) to an invertible edge in \(\C\) then the dotted lift exists.  
\end{rem}

We now consider the question of cartesian 1-simplices for the domain projection.

\begin{lemma}
	\label{cart lift of dom n=1}
	Let \(\C\) be an \(\infty\)-bicategory and consider the domain projection \(\dom\colon \RMap(\Del^1,\C) \to \C\) as above.
	Then a 1-simplex \(\alpha\colon \Delta^1 \rightarrow \RMap(\Delta^1,\C)\) is \(d\)-cartesian if 
	its transpose \(\hat{\alpha}\colon \Delta^1\otimes \Delta^1\to\C\) corresponds to a commutative square (i.e.~both its non-degenerate triangles are thin) with the side \(\hat{\alpha}(\{1\}\times \Delta^1)\) being an equivalence in \(\C\). Pictorially, \(\hat{\alpha}\) looks as follows:
	\[
		\begin{tikzcd}
		a\ar[dr,"k"{description,name=k'}]\ar[d,"g"{swap}]\ar[r]& x\ar[d,"\simeq"]\\
		b\ar[r,"l"{swap}] \ar[Rightarrow,"\simeq"{phantom}, from=k',to=2-1]\ar[Rightarrow, "\simeq", from=k',to=1-2]& y
		\end{tikzcd}
    \]
\end{lemma}

\begin{proof}
	Consider a 1-simplex \(\alpha\colon \Delta^1\rightarrow \RMap(\Delta^1,\C)\),  as described above. Given \(n\geq 2\) and a solid square of the form:
	\[\begin{tikzcd}
	& \Delta^{\{n-1,n\}} \ar[dl] \ar[dr, "\alpha"]\\
	(\Lambda^n_n,\Lambda^n_n\cap\Delta^{\{0,n-1,n\}}) \ar[rr] \ar[d]& &\RMap(\Delta^1,\C) \ar[d,"\mathrm{d}"]\\
	(\Delta^n,\Delta^{\{0,n-1,n\}})\ar[rr]\ar[urr,dotted]&  &\C
	\end{tikzcd}\] we have to exhibit a filler as indicated by the dotted arrow. Note that this lifting problem corresponds to one of the form:
	\begin{equation}
	\label{1}
	\begin{tikzcd}
	\big(\Delta^1 \times \Lambda^n_n\displaystyle\mathop{\coprod}_{\Del^{\{0\}} \times \Lam^n_n} \Del^{\{0\}}\times 	\Delta^n,T'\big)\ar[r,"f"]\ar[d]& \C\\
	(\Delta^1\times \Delta^n,T) \ar[ur,dotted]
	\end{tikzcd}
	\end{equation}
	in which \(\Delta^1\times \Delta^{\{n-1,n\}}\) is mapped to \(\hat{\alpha}\) under the adjunction \(\Delta^1\otimes- \dashv \RMap(\Delta^1,-)\). 
	Here \(T\) is the union of the triangles which are thin in \(\Delta^1_{\flat}\otimes (\Delta^n,\{\Delta^{\{0,n-1,n\}}\})\), together with \(\Delta^{\{(0,n-1),(0,n),(1,n)\}}\),  
	and \(T'\subseteq T\) is the subset of those triangles which are contained in the domain of the vertical arrow in~\eqref{1}.
	Since the edge \(\Del^{\{1\}} \times \Del^{\{n-1,n\}}\) maps to an invertible edge of \(\C\) by our assumption on \(\alp\) and \(T\) contains \(\Delta^{\{1\}} \times \Delta^{\{0, n-1, n\}}\)
	we can extend the map \(f\) in~\eqref{1}
	to a map 
\[	f'\colon \big([\Delta^1 \times \Lambda^n_n]\displaystyle\mathop{\coprod}_{\partial \Del^1 \times \Lam^n_n} [\partial \Del^1\times 	\Delta^n],T''\big) \to \C\]
	where \(T''\) is the intersection of \(T\) with the triangles in \([\Delta^1 \times \Lambda^n_n]\displaystyle\mathop{\coprod}_{\partial \Del^1 \times \Lam^n_n} [\partial \Del^1\times 	\Delta^n]\)
Let \(S = T'' \cup \{\Del^{\{(0,n-1),(0,n),(1,n)\}},\Del^{\{(0,0),(1,n-1),(1,n)\}}\}\). Then \(S\) is contained in \(T\) and by Lemma~\ref{lem:technical-part} we may extend \(f'\) to \((\Del^1 \times \Del^n,S)\). When \(n \geq 3\) we have that \(T=T''=S\) and so the proof is complete. In the case \(n=2\) one needs to additionally verify that the resulting extension sends \(\Del^{\{(0,0),(1,0),(1,1)\}}\) to a thin triangle.
Indeed, this follows from~\cite[Proposition~3.4]{GagnaHarpazLanariEquiv} since \(S\) contains \(\Del^{\{(0,0),(1,1),(1,2)\}}, \Del^{\{(0,0),(1,0),(1,2)\}}\)
and \(\Del^{\{(1,0),(1,1),(1,2)\}}\), and \(\Del^{\{(1,1),(1,2)\}}\) maps to an invertible edge in \(\C\) by assumption.
\end{proof}

\begin{lemma}
	\label{lemma:dom_projection_enough_1-cartesian}
	The domain projection
	\[
		\mathrm{d}\colon\RMap(\Delta^1,\C)\rightarrow \C.
	\]
	of an \(\infty\)-bicategory \(\C\) has enough cartesian edges.
\end{lemma}

\begin{proof}
	A lifting problem of the form
	\[
		\begin{tikzcd}
		\Delta^{\{1\}} \ar[r,"h"] \ar[d]& \RMap(\Delta^1,\C) \ar[d,"d"]\\
		\Delta^1 \ar[r,"f"]&\C
		\end{tikzcd}
	\]
	corresponds to the following data in \(\C\):
	\[
		\begin{tikzcd}
		x\ar[d,"f"{swap}]\\
		y\ar[r,"h"]&z 
		\end{tikzcd}.
	\]
	This can be extended to a 1-simplex in \(\RMap(\Delta^1,\C)\) as depicted below
	\[
		\begin{tikzcd}[column sep=large, row sep=large]
		x\ar[d,"f"{swap}] \ar[r,"hf"] \ar[dr,"hf"{description, name=hf}] &
		z\ar[d, equal] \\
		y \ar[r,"h"{below}] & z
		\ar[phantom, from=hf, to=2-1, "\simeq"]
		\ar[phantom, from=hf, to=1-2, "="]
		\end{tikzcd}\ ,
	\]
	where bottom triangle is given by an extension along
	$\Lambda^2_1 \to \Delta^2_{\sharp}$ and the upper triangle
	is degenerate. Such a 1-simplex is then \(\dom\)-cartesian by Lemma~\ref{cart lift of dom n=1}.
\end{proof}

\begin{lemma}
	\label{cart lift of dom n=2}
	Let \(\alpha\colon \Delta^2 \rightarrow \RMap(\Delta^1,\C)\) be a triangle with transpose \(\hat{\alpha}\colon \Delta^1 \otimes \Delta^2 \to \C\).
	\begin{enumerate} 
		\item If \(\hat{\alp}|_{\Del^{\{1\}} \times \Del^2}\) is thin and \(\alp|_{\Del^{\{1,2\}}}\) satisfies the assumption of Lemma~\ref{cart lift of dom n=1} (so that it is \(\dom\)-cartesian by that lemma) then \(\alp\) is right \(\dom\)-outer.
		\item If \(\hat{\alp}|_{\Del^{\{1\}} \times \Del^2}\) is thin and \(\alp|_{\Del^{\{0,1\}}}\) is invertible then \(\alp\) is left \(\dom\)-outer. 
	\end{enumerate}
\end{lemma}

\begin{proof}
We first prove (1). By Remark~\ref{rem:unwinding} we need to consider a lifting problem of the form:
	\[\begin{tikzcd}
	& \Delta^{\{0,n-1,n\}}_{\flat} \ar[dl] \ar[dr, "\alpha"]\\
	(\Lambda^{n}_{n})_{\flat} \ar[rr] \ar[d]& &\RMap(\Delta^1,\C) \ar[d,"\mathrm{d}"]\\
	\Delta^n_{\flat} \ar[rr]\ar[urr,dotted]&  &\C
	\end{tikzcd}\] for \(n\geq 3\), where we have to exhibit a diagonal filler as indicated by the dotted arrow.
	We obtain an equivalent lifting problem of the form:
	\[
		\begin{tikzcd}
	\big(\Delta^1 \times \Lambda^n_n\displaystyle\mathop{\coprod}_{\Del^{\{0\}} \times \Lam^n_n} \Del^{\{0\}}\times 	\Delta^n,T\big)\ar[r,"f"]\ar[d]& \C\\
	(\Delta^1\times \Delta^n,T) \ar[ur,dotted]
	\end{tikzcd}
\]
where \(T\) is the set of thin triangles in \(\Del^1_{\flat} \otimes \Del^n_\flat\) together with \(\Delta^{\{(0,n-1),(0,n),(1,n)\}}\) and \(\Del^{\{1\}} \times \Del^{\{0,n-1,n\}}\) 
(all of whom are contained in the top left corner, since \(n \geq 3\)). 
    Our assumption on \(\alp\) implies that the edge \(\Delta^{\{1\}} \times \Delta^{\{n-1, n\}}\) is invertible in \(\C\)  
    and so by~\cite[Corollary 2.3.10]{GagnaHarpazLanariLaxLimits} we can extend the map \(f\) to a map
\[	f'\colon \big(\Delta^1 \times \Lambda^n_n\displaystyle\mathop{\coprod}_{\partial \Del^1 \times \Lam^n_n} \partial \Del^1\times 	\Delta^n,T\big) \to \C.\]
Now since \(T\) contains \(\Del^{\{(0,0),(1,0),(1,n-1)\}}\), \(\Del^{\{(0,0),(1,0),(1,n)\}}\) and \(\Del^{\{1\}} \times \Del^{\{0,n-1,n\}}\)  
and \(\C\) is an \(\infty\)-bicategory, 
the map \(f'\) must also send \(\Del^{\{(0,0),(1,n-1),(1,n)\}}\) to a thin triangle (see~\cite[Remark~3.1.4]{LurieGoodwillie}).
We may consequently apply Lemma~\ref{lem:technical-part} in order to extend \(f'\) to all of \((\Del^1 \times \Del^n,T)\), as desired.

We now prove (2). By Remark~\ref{rem:unwinding} we now need to consider a lifting problem of the form:
	\[\begin{tikzcd}
	& \Delta^{\{0,1,n\}}_{\flat} \ar[dl] \ar[dr, "\alpha"]\\
	(\Lambda^{n}_{0})_{\flat} \ar[rr] \ar[d]& &\RMap(\Delta^1,\C) \ar[d,"\mathrm{d}"]\\
	\Delta^n_{\flat} \ar[rr]\ar[urr,dotted]&  &\C
	\end{tikzcd}\] 
	with \(n \geq 3\). 
	We obtain an equivalent lifting problem of the form:
	\[
		\begin{tikzcd}
	\big(\Delta^1 \times \Lambda^n_0\displaystyle\mathop{\coprod}_{\Del^{\{0\}} \times \Lam^n_0} \Del^{\{0\}}\times 	\Delta^n,T\big)\ar[r,"f"]\ar[d]& \C\\
	(\Delta^1\times \Delta^n,T) \ar[ur,dotted]
	\end{tikzcd}
\]
where \(T\) is the set of thin triangles in \(\Del^1_{\flat} \otimes \Del^n_\flat\) together with \(\Delta^{\{(0,0),(0,1),(1,1)\}}\) and \(\Delta^{\{1\}} \times \Delta^{\{0, 1, n\}}\).
Since the image of the edge \(\Delta^{\{1\}} \times \Delta^{\{0, 1\}}\) is invertible in \(\C\) 
we can extend the map \(f\) to a map
\[	f'\colon \big([\Delta^1 \times \Lambda^n_0]\displaystyle\mathop{\coprod}_{\partial \Del^1 \times \Lam^n_0} [\partial \Del^1\times 	\Delta^n],T\big) \to \C.\]
Now since \(T\) contains \(\Del^{\{(0,0),(1,0),(1,1)\}}\), \(\Del^{\{(0,0),(1,0),(1,n)\}}\) and \(\Del^{\{1\}} \times \Del^{\{0,1,n\}}\)  
and \(\C\) is an \(\infty\)-bicategory,
the map \(f'\) must also send \(\Del^{\{(0,0),(1,1),(1,n)\}}\) to a thin triangle. Since \(f'\) also sends
 \(\Del^{\{(0,0),(0,1),(1,1)\}}\) and \(\Del^{\{(0,1),(1,1),(1,n)\}}\) to thin triangles, the same holds for \(\Del^{\{(0,0),(0,1),(1,n)\}}\). We may consequently apply the dual form of Lemma~\ref{lem:technical-part} (see Remark~\ref{rem:dual-form}), in order to extend \(f'\) to all of \((\Del^1 \times \Del^n,T)\), as desired.
\end{proof}	

Finally, we are ready to prove the main result of this section.

\begin{thm}
\label{thm:domain proj}
Given an \(\infty\)-bicategory \(\C\), the domain projection 
\[\mathrm{d}\colon\RMap(\Delta^1,\C)\rightarrow \C\] is a 2-outer cartesian fibration, whose \(p\)-cartesian 1-simplices 
are those described in Lemma~\ref{cart lift of dom n=1}. In addition, the right \(p\)-outer triangles whose right legs are cartesian are those described in Lemma~\ref{cart lift of dom n=2}(1), and the left \(p\)-outer triangles 
whose left leg is invertible  
are those described in Lemma~\ref{cart lift of dom n=2}(2).
\end{thm}

\begin{proof}
By Remark~\ref{rem:dom_projection_weak_fibration}, the map
\(\mathrm{d}\) is a weak fibration and by Lemma \ref{cart lift of dom n=1} it has a sufficient supply of \(\dom\)-cartesian lifts for 1-morphisms. By the essential unicity of \(\dom\)-cartesian lifts we deduce that all \(\dom\)-cartesian arrows are of the form described in Lemma \ref{cart lift of dom n=1}.

We now show that the triangles in \(\C\) 
have a sufficient supply of right \(p\)-outer lifts.
This translates into a lifting problem of the form:
\[
\begin{tikzcd}
\Lambda^2_2 \ar[r,"\rho"]\ar[d]& \RMap(\Delta^1,\C) \ar[d,"\text{d}"] \\
\Delta^2 \ar[r,"\gamma"] \ar[ur,dotted] &\C
\end{tikzcd},
\]
with \(\rho|_{\Del^{\{1,2\}}}\) a \(\mathrm{d}\)-cartesian edge, and so (as argued just above) of the form described in Lemma \ref{cart lift of dom n=1}. This means in particular that \(\rho|_{\Del^{\{1,2\}}}\) corresponds to a map $\Del^1_{\flat} \otimes \Del^{\{1,2\}}_{\flat} \to \C$ which sends both triangles to thin triangles and the edge \(\Del^{\{1\}} \times \Del^{\{1,2\}}\) to an invertible edge in \(\C\). Solving this lifting problem in a way that produces a triangle of \(\RMap(\Delta^1,\C)\) of the form described in Lemma~\ref{cart lift of dom n=2}(1) then corresponds to solving a lifting problem of the form
\begin{equation}
\begin{tikzcd}[column sep=large]
\label{lift pr 1}
\big(\Delta^1 \times \Lambda^2_2\displaystyle\mathop{\coprod}_{\Del^{\{0\}} \times \Lambda^2_2} \Del^{\{0\}} \times \Delta^2,T'\big) 
	\ar[r,"\hat{\rho} \cup \gamma"] \ar[d] & \C\\
(\Delta^1\times \Delta^2,T) \ar[ur,dotted]
\end{tikzcd}.
\end{equation}
where \(T'\) is obtained by intersection from \(T\), which in turn contains all triangles which are thin in \(\Del^1_\flat \otimes \Del^2_\flat\)
as well as \(\Del^{\{1\}} \times \Del^2\) 
and \(\Del^{\{(0,1),(0,2),(1,2)\}}\), and \(\hat{\rho}\) sends \(\Del^{\{1\}} \times \Del^{\{1,2\}}\) to an invertible edge in \(\C\).  
Since \(\C\) is an \(\infty\)-bicategory we may then extend \(\hat{\rho} \cup \gamma\) to a map
 \[
 	g\colon \big(\Delta^1\times \Lambda^2_2
 	\displaystyle\mathop{\coprod}_{\partial \Del^1 \times \Lambda^2_2} \partial\Delta^1\times \Delta^2,T''\big) \to \C.
 \]
where \(T''\) is obtained from \(T\) by intersection. Applying Lemma~\ref{lem:technical-part} we may extend \(g\) to a map \(g'\colon (\Del^1 \times \Del^2, T'' \cup \{\Del^{\{(0,0),(1,1),(1,2)\}}\}) \to \C\). We then observe that there is exactly one triangle in \(T\) which is not in \(T'' \cup \{\Del^{\{(0,0),(1,1),(1,2)\}}\}\), and that is the triangle \(\Del^{\{(0,0),(1,0),(1,1)\}}\). But this triangle is sent by \(g\) to a thin triangle in \(\C\) by~\cite[Proposition 3.4]{GagnaHarpazLanariEquiv}, since \(T\) contains \(\Del^{\{(0,0),(1,1),(1,2)\}}\), \(\Del^{\{(0,0),(1,0),(1,2)\}}\) and \(\Del^{\{(1,0),(1,1),(1,2)\}}\), and \(g\) sends \(\Del^{\{1\}} \times \Del^{\{1,2\}}\) 
to an invertible edge.

We now show that the triangles in \(\C\) have a sufficient supply of left \(p\)-outer lifts. Invoking Corollary~\ref{cor:only-degenerate}, it will suffice to test this only for triangles in \(\B\) whose left leg is degenerate.
This then translates into a lifting problem of the form:
\[
\begin{tikzcd}
\Lambda^2_0 \ar[r,"\rho"]\ar[d]& \RMap(\Delta^1,\C) \ar[d,"\text{d}"] \\
\Delta^2 \ar[r,"\gamma"] \ar[ur,dotted] &\C
\end{tikzcd},
\]
with \(\rho|_{\Del^{\{0,1\}}}\) a \(\mathrm{d}\)-cocartesian lift of an identity, and hence an invertible edge of \(\RMap(\Delta^1,\C)\). This means in particular that \(\rho|_{\Del^{\{0,1\}}}\) corresponds to a map $\Del^1_{\flat} \otimes \Del^{\{0,1\}}_{\flat} \to \C$ which sends both triangles to thin triangles and the edges \(\Del^{\{0\}} \times \Del^{\{0,1\}}\) and \(\Del^{\{1\}} \times \Del^{\{0,1\}}\) to invertible edges in \(\C\). Solving this lifting problem in a way that produces a triangle of \(\RMap(\Delta^1,\C)\) of the form described in Lemma~\ref{cart lift of dom n=2}(2) then corresponds to solving a lifting problem of the form
\begin{equation}
\begin{tikzcd}[column sep=large]
\label{lift pr 2}
\big(\Delta^1 \times \Lambda^2_0\displaystyle\mathop{\coprod}_{\Del^{\{0\}} \times \Lambda^2_0} \Del^{\{0\}} \times \Delta^2,T'\big) 
	\ar[r,"\hat{\rho} \cup \gamma"] \ar[d] & \C\\
(\Delta^1\times \Delta^2,T) \ar[ur,dotted]
\end{tikzcd}.
\end{equation}
where \(T'\) is obtained by intersection from \(T\), which in turn contains all triangles which are thin in \(\Del^1_{\flat} \otimes \Del^2_{\flat}\) as well as the triangles \(\Del^{\{1\}} \times \Del^2\) 
and \(\Del^{\{(0,0),(0,1),(1,1)\}}\), and \(\hat{\rho}\) sends \(\Del^{\{\eps\}} \times \Del^{\{0,1\}}\) to an invertible edge in \(\C\) for \(\eps=0,1\).  
Since \(\C\) is an \(\infty\)-bicategory we may then extend \(\hat{\rho} \cup \gamma\) to a map
 \[
 	g\colon \big(\Delta^1\times \Lambda^2_0
 	\displaystyle\mathop{\coprod}_{\partial \Del^1 \times \Lambda^2_0} \partial\Delta^1\times \Delta^2,T''\big) \to \C.
 \]
where \(T''\) is obtained from \(T\) by intersection. Applying the dual of Lemma~\ref{lem:technical-part} (Remark~\ref{rem:dual-form}) we may extend \(g\) to a map \(g'\colon (\Del^1 \times \Del^2, T'' \cup \{\Del^{\{(0,0),(0,1),(1,2)\}}\}) \to \C\). We then observe that there are exactly two triangles in \(T\) which are not in \(T''\), namely, \(\Del^{\{(0,0),(1,1),(1,2)\}}\) and \(\Del^{\{(0,1),(1,1),(1,2)\}}\). To finish the proof we will show that these two triangles are sent by \(g\) to thin triangles in \(\C\). For the first one, we note that since \(g\) sends \(\Del^{\{(0,0),(1,0),(1,1)\}}\), \(\Del^{\{(0,0),(1,0),(1,2)\}}\) and \(\Del^{\{1\}} \times \Del^2\) 
to thin triangles, then it must also send \(\Del^{\{(0,0),(1,1),(1,2)\}}\) to a thin triangle. Then, since \(g\) also sends \(\Del^{\{(0,0),(0,1),(1,1)\}}\) and \(\Del^{\{(0,0),(0,1),(1,2)\}}\) to thin triangles, and the edge \(\Del^{\{(0,0),(0,1)\}}\) to an equivalence, then it must send \(\Del^{\{(0,1),(1,1),(1,2)\}}\) to a thin triangle as well (see \cite[Proposition~3.4]{GagnaHarpazLanariEquiv}).

Having provided sufficiently many left and right \(\mathrm{d}\)-outer lifts of the form appearing in Lemma~\ref{cart lift of dom n=2},
the uniqueness of \(\mathrm{d}\)-outer lifts, as expressed for example in Remark~\ref{uniqueness of cart 2-lifts},
shows that all right \(\mathrm{d}\)-outer triangles whose right leg is cartesian are of the form described in~\ref{cart lift of dom n=2}(1),
and all left \(\mathrm{d}\)-outer triangles whose left leg is invertible 
are of the form described in~\ref{cart lift of dom n=2}(2).
To show that the collection of right (resp.~left) \(\mathrm{d}\)-outer triangles is closed under right (resp.~left) whiskering,
it will hence suffice to show that the collection of triangles \(\alp\colon \Del^2 \to \RMap(\Del^1,\C)\)
whose adjoint \(\hat{\alp}\colon \Del^1 \otimes \Del^2 \to \C\) sends \(\Del^{\{1\}} \times \Del^2\) to a thin triangle,
is closed under both left and right whiskering. Indeed, this is exactly the pre-image in \(\RMap(\Del^1,\C)\)
under the \emph{codomain projection} of the collection of thin triangles in \(\C\), and the collection of thin triangles is closed under whiskering from both sides.
\end{proof}

\begin{rem}
Passing to opposites, Theorem~\ref{thm:domain proj} implies that the codomain projection
\[\mathrm{cod}\colon\LMap(\Delta^1,\C)\rightarrow \C\] 
is 2-outer cocartesian fibration, whose cocartesian arrows and outer triangles admit similar descriptions. Note that we have switched not only from domain to codomain but also from \(\RMap\) to \(\LMap\). If we only switch from domain to codomain then the resulting projection
\[\mathrm{cod}\colon\RMap(\Delta^1,\C)\rightarrow \C\] 
is a 2-inner cocartesian fibration. This claim does not formally follow from the outer statement, but its proof can be obtained using a completely analogous argument, replacing the key extension result of Lemma~\ref{lem:technical-part} with one involving the pushout product of \(\partial \Del^1 \to \Del^1\) and a suitable \emph{inner} horn inclusion. Similarly, the projection
\[\mathrm{d}\colon\LMap(\Delta^1,\C)\rightarrow \C\] 
is a 2-inner cartesian fibration.
\end{rem}

\section{Enriched cartesian fibrations}\label{sec:enriched-fibrations}
In this section we define the notion of cartesian fibration in the context of marked-simplicial categories. The definition is motivated by the one given for 2-categories by Buckley in \cite{BuckleyFibred}, which we recall in \S\ref{2-fib of 2-cats}.

\subsection{Recollection: fibrations of 2-categories}
 \label{2-fib of 2-cats}
Fibrations of 2-categories were initially introduced by Hermida~\cite{HermidaFib}. 
A suitably modified definition was later given by Buckley~\cite{BuckleyFibred}, who also proved an (un)straightening-type result. In what follows we give a concise summary of the main results of loc.~cit., to be considered as a motivation for the discussion of simplicial categories in \S\ref{sec:car-fib-enriched}.

\begin{define}
	Let \(p\colon \E \to \B\) be a 2-functor between 2-categories. 
	\begin{itemize}
		\item A 1-cell \(f\colon x \to y\) in \(\E\) is \(p\)-\emph{cartesian} if the following square is a pullback of categories for every \(a\in \E\):
		\[\begin{tikzcd}
		\E(a,x) \ar[r,"f\circ-"] \ar[d,"p_{a,x}"{swap}]& \E(a,y)\ar[d,"p_{a,y}"]\\
		\B(pa,px)\ar[r,"p(f)\circ -"]& \B(pa,py)
		\end{tikzcd}\]
		\item A 2-cell \(\alpha\colon f \Rightarrow g\colon x \to y\) in \(\E\) is \(p\)-\emph{cartesian} if it is cartesian with respect to the induced functor \(p_{x,y}\colon \E(x,y)\to \B(px,py)\).
	\end{itemize}
\end{define}
The notion of cartesian fibration for 2-categories amounts to the existence of enough cartesian lifts, as in the 1-dimensional case, but it also requires an additional property: cartesian 2-cells must be closed under horizontal composition. 

\begin{define}
\label{d:fibrations of 2cats}
	A 2-functor between 2-categories \(p\colon \E \to \B\) is called a 2-\emph{fibration} if it satisfies the following properties:
	\begin{enumerate}
		\item for every object \(e\in \E\) and every 1-cell \(f\colon b \to p(e)\) there exists a \(p\)-cartesian 1-cell \(h\colon a \to  e\) in \(\E\) such that \(p(h)=f\);
		\item for every pair of objects \(x,y\) in \(\E\), the map \(p_{x,y}\colon \E(x,y)\to \B(px,py)\) is a cartesian fibration of categories;
		\item cartesian 2-cells are closed under horizontal composition, i.e.~for every triple of objects \((x,y,z)\) in \(\E\), the functor \(\circ_{x,y,z}\colon \E(y,z)\times \E(x,y)\to \E(x,z)\) sends \(p_{y,z}\times p_{x,y}\)-cartesian 1-cells to \(p_{x,z}\)-cartesian ones.
	\end{enumerate} 
\end{define}

Replacing cartesian lifts for 1-cells by cocartesian lifts and similarly for 2-cells one may obtain four 
different variants of fibration, corresponding to the four possible types of variance for pseudo-functors \(\B \to \nCat{2}\).

\begin{rem}
Condition (3) of Definition~\ref{d:fibrations of 2cats} is equivalent to requiring that given 1-cells in \(\E\) of the form \(f\colon w\to x\) and \(g\colon y \to z\), the whiskering functors \(-\circ f\colon E(x,y)\to \E(w,y)\) and \(g\circ - \colon \E(x,y)\to \E(x,z)\) preserve cartesian 2-cells. This follows from the fact that horizontal composition can be obtained from vertical composition and whiskering composition.
\end{rem}

The following result appears as Theorem 2.2.11 in \cite{BuckleyFibred}.
\begin{thm}
	There exists an equivalence of 3-categories between \(2\mathbf{Fib}_s(\B)\) and \([\B^{op}_{co},\nCat{2}]\), the former being the 3-category of fibrations equipped with a choice of cartesian lifts compatible with composition, while the latter is the 3-category of (strict) 2-functors into \(\nCat{2}\), natural transformations and modifications. 
\end{thm}
In the same paper, the author proves several weakenings of this statement, by looking at fibrations without a choice of lifts (which correspond to \emph{pseudofunctors}) and fibrations of bicategories.

\subsection{Cartesian fibrations of enriched categories}\label{sec:car-fib-enriched}

We now consider cartesian fibrations in the setting of \(\Catoo\)-categories. All definitions and statements can be dualized to the case of cocartesian fibrations by replacing all \(\Catoo\)-categories by their opposites.

\begin{define}
Let \(p\colon\E\rightarrow \B\) be a map of marked-simplicial categories. A morphism \(h\in\E(e',e)_0\) is said to be \(p\)-\emph{cartesian} if the induced square
\[\begin{tikzcd}[column sep=large]
\E(a,e')\ar[r,"h\circ-"] \ar[d,"p_{a,e'}"{swap}]&\E(a,e)\ar[d,"p_{a,e}"]\\
\B(pa,pe') \ar[r,"p(h)\circ-"]& \B(pa,pe)
\end{tikzcd}\]
is a homotopy pullback for the model structure on marked simplicial sets.
\end{define}

\begin{define}
	\label{cart fib marked cat}
Let \(p\colon \E \rightarrow\B\) be a fibration of \(\Catoo\)-categories. We say that \(p\) is an \emph{enriched cartesian fibration} if for any \(e \in \E\) and any morphism \(f\in\B(a,p(e))_0\) there exists a \(p\)-cartesian morphism \(h\in \E(e',e)_0\) for some \(e'\in\E\), such that \(p(h)=f\). 
\end{define}

\begin{define}
Let \(p\colon \E \rightarrow\B\) be a fibration of  \(\Catoo\)-categories. We say that \(p\) is an \emph{enriched 2-inner} (resp.~\emph{2-outer}) \emph{fibration} if it satisfies the following properties:
\begin{enumerate}
\item For every pair of objects \((x,y)\) in \( \E\), the map \(p_{x,y}\colon\E(x,y)\rightarrow \B(px,py)\) is a cartesian (resp.~cocartesian) fibration on the level of underlying simplicial sets.
\item For every 0-simplices \(u\colon y\rightarrow z\) and \(v\colon w \to x \) in \(\E\), the commutative squares
\[\begin{tikzcd}[column sep=large]
 \E(x,y)\ar[r,"u\circ-"]\ar[d,"p_{x,y}"{swap}]& \E(x,z)\ar[d,"p_{x,z}"] && \E(x,y)\ar[r,"-\circ v"]\ar[d,"p_{x,y}"]&\E(w,y)\ar[d,"p_{w,y}"]\\
\B(px,py) \ar[r,"p(u)\circ-"]& \B(px,pz) && \B(px,py)\ar[r,"-\circ p(v)"]& \B(pw,py)
\end{tikzcd}\]
are morphisms of cartesian (resp.~cocartesian) 
fibrations, i.e., the top horizontal maps in both squares preserve cartesian (resp.~cocartesian) edges.
\end{enumerate}
We will say that \(p\) is \ndef{an enriched 2-inner} (resp.~\ndef{2-outer}) \ndef{cartesian fibration} if it is an enriched 2-inner (resp.~2-outer) fibration and an enriched cartesian fibration.
\end{define}

Our goal in the present section is to prove the following:
\begin{thm}
\label{N detects fib cart} 
Let \(p\colon \E \to \B\) be a fibration of \(\Catoo\)-categories. Then \(p\colon \E \rightarrow \B\) is an enriched 2-inner (resp.~2-outer) cartesian 
fibration in the sense of Definition \ref{cart fib marked cat} if and only if
\[
 \rN^{\sca}(p)\colon \rN^{\sca} \E\rightarrow\rN^{\sca} \B
\]
is a 2-inner (resp.~2-outer) cartesian 
fibration of \(\infty\)-bicategories.
\end{thm}

The remainder of this section is devoted to the proof of Theorem~\ref{N detects fib cart}.

\begin{define}
We will denote by \(\Box^n = (\Del^1)^n\) the \emph{simplicial \(n\)-cube} and by \(\partial \Box^n\) its \emph{boundary}, so that the inclusion \(\partial \Box^n \subseteq \Box^n\) can be identified with the pushout-product of \(\partial \Del^1 \hrar \Del^1\) with itself \(n\) times. For \(i=1,...,n\) we 
denote by \(\sqcap^{n,i}_{\eps} \hrar \Box^{n}\) the iterated pushout-product 
\[[\partial \Del^1 \hrar \Del^1] \Box \hdots \Box [\Del^{\{\eps\}} \hrar \Del^1] \Box \hdots \Box [\partial \Del^1 \hrar \Del^1],\]
where \(\eps \in \{0,1\}\) and \([\Del^{\{\eps\}} \hrar \Del^1]\) appears in the \(i\)'th factor. 
\end{define}

Lurie introduces in \cite[Definition 3.1.1.1]{HTT} the class of \emph{cartesian anodyne} morphisms (called \emph{marked anodyne} in loc.~cit.), which is the smallest weakly saturated
class generated by a certain set of monomorphisms of marked simplicial sets listed in \cite[Definition 3.1.1.1]{HTT}, which contain, in particular, the collection of all (minimally marked) inner horn inclusions, as well as the marked outer horn inclusion \((\Lam^n_n,\{\Del^{\{n-1,n\}}\}) \hrar (\Del^n,\{\Del^{\{n-1,n\}}\})\). 
Dually, we shall call \emph{cocartesian anodyne} maps 
the smallest weakly
saturated class generated by the opposites of those maps (or simply those maps whose opposites are cartesian anodyne).
By~\cite[Proposition 3.1.2.3]{HTT}, (co)cartesian anodyne maps are closed under pushout-product
with monomorphisms.

\begin{lemma}
\label{marked cube}
For \(\eps \in \{0,1\}\), \(n \geq 1\) and \(1 \leq i \leq n\) let \(E^i_{\eps} \subseteq (\Box^n)_1\) be the set of all degenerate edges together with the edge \((\eps,...,\eps) \times \Del^1 \times (\eps,...,\eps)\) (where \(\Del^1\) sits in the \(i\)-th place). Then the inclusion of marked simplicial sets \((\sqcap^{n,i}_{1},E^i_1) \hrar (\Box^{n},E^i_1)\) is cartesian anodyne and the inclusion of marked simplicial sets \((\sqcap^{n,i}_{0},E^i_0) \hrar (\Box^{n},E^i_0)
\) is cocartesian anodyne.
\end{lemma}
\begin{proof}
We note that the \(\eps=0\) and \(\eps=1\) statements imply each other by passing to opposites. 
We hence just prove the cartesian case. 
Ignoring the order of factors, we may identify the map \(\sqcap^{n,i}_{1} \hrar \Box^n\) with the box product of \(\partial \Box^{n-1} \hrar \Box^{n-1}\) and \(\Del^{\{1\}} \hrar \Del^1\). The non-degenerate \(m\)-simplices of \(\Box^{n-1}\) which are not in \(\partial \Box^{n-1}\) correspond to maps of posets \([m] \to [1]^{n-1}\) whose projection to each factor \([1]\) is surjective. It then follows that the initial vertex of such an \(m\)-simplex must be \((0,...,0)\) and the final vertex must be \((1,...,1)\). Adding these non-degenerate simplices one by one in an order that respects dimensions (that is, first all the 1-dimensional ones, then all the 2-dimensional ones, etc., until dimension \(n-1\)) results in a factorization of the map \((\sqcap^{n,i}_{1},E^i_1) \hrar (\Box^{n},E^i_1)\) into a finite composite of pushouts of maps of the form 
\[ \left(\Del^{\{1\}} \times \Del^m \displaystyle\mathop{\coprod}_{\Del^{\{1\}} \times \partial \Del^m} \Del^1 \times \partial \Del^m, \{\Del^1 \times \Del^{\{m\}}\}\right)  \to (\Del^1 \times \Del^m, \{\Del^1 \times \Del^{\{m\}}\}) \]
for \(m \geq 1\). 
It will hence suffice to show that each of these maps is cartesian anodyne. 
For \(\ell=0,...,m\) let
\[\tau_\ell\colon \Del^{m+1} \rightarrow \Del^1 \times \Del^m\] be the map given on vertices by the formula
\[ \tau_\ell(j) = \left\{\begin{matrix} (0,j) & j \leq \ell \\ (1,j-1) & j > \ell \end{matrix}\right. \]
and for \(k=0,\dots,m+1\) let \(Z_k \subseteq (\Del^{1} \times \Del^m,\{\Del^1 \times \Del^{\{m\}}\})\) be the marked simplicial subset obtained as the union of 
\([\Del^1 \times \partial \Del^{m}] \coprod_{\Del^{\{1\}}\times \partial \Del^{m}}[\Del^{\{1\}}\times \Del^{m}]\)	
and the simplices \(\tau_\ell\) for \(0 \leq \ell < k\). Set
\[
	Z_0 \stackrel{\text{def}}{=} \Big([\Del^1 \times \partial \Del^{m}] \coprod_{\Del^{\{1\}}\times \partial \Del^{m}}[\Del^{\{1\}}\times \Del^{m}],\{\Del^1 \times \Del^{\{m\}}\}\Big).
\]
We then have an ascending filtration of marked simplicial sets
\[
Z_0 \subseteq Z_{1} \subseteq \dots \subseteq Z_{m+1} = (\Del^1 \times \Del^{m},\{\Del^1 \times \Del^{\{m\}}\}) \ .
\]
For each \(k=0,\dots,m-1\) we then find a pushout square of marked simplicial sets
\[ 
\xymatrix{
(\Lam^{m+1}_{k+1})^{\flat} \ar[r]\ar[d] & Z_{k} \ar[d] \\
(\Del^{m+1})^{\flat} \ar[r] & Z_{k+1} \\
}\]
so that \(Z_{k} \to Z_{k+1}\) is inner anodyne, and in particular cartesian anodyne. Finally, in the last step \(k=m\) we find a pushout square of the form
\[ 
\xymatrix{
(\Lam^{m+1}_{m+1},\{\Del^{\{m,m+1\}}\}) \ar[r]\ar[d] & Z_{m} \ar[d] \\
(\Del^{m+1},\{\Del^{\{m,m+1\}}\}) \ar[r] & Z_{m+1} \\
}\]
so that \(Z_m \to Z_{m+1}\) is cartesian anodyne, as desired.
\end{proof}

We recall from~\cite{GagnaHarpazLanariLaxLimits} the comparison of the notion of (co)cartesian 1-cells between the enriched and scaled models:

\begin{prop}[{\cite[Proposition~3.1.3]{GagnaHarpazLanariLaxLimits}}]
\label{comparison enriched 1-simpl}
Given a fibration \(p\colon \E \to \B\) of \(\Catoo\)-categories, an arrow in \(\E\) is \(p\)-cartesian if and only if the corresponding 1-simplex of \(\Nsc \E\) is \(\Nsc(p)\)-cartesian.
\end{prop}

Since the morphisms in a given \(\Catoo\)-category are in bijection with the edges of its scaled coherent nerve we readily obtain:
\begin{cor}
\label{comparison enriched car fib}
A fibration \(p\colon \E \to \B\) of \(\Catoo\)-categories is cartesian if and only if \(\Nsc(p)\) is a cartesian fibration of \(\infty\)-bicategories (in the sense of Definition~\ref{d:car-fibration}).
\end{cor}

We now consider the analogous question on the level of triangles.

\begin{rem}
	Let \(\E\) be a \(\Catoo\)-category and consider two \(2\)-simplices \(\alpha\) and \(\alpha'\) of \(\Nsc \E\).
	If a \(3\)-simplex \(\rho\) of \(\Nsc\E\) exhibits \(\alpha\) as left-congruent to \(\alpha'\) (cf.~Definition~\ref{def:congruent})
	in the following form
	\begin{nscenter}
		\begin{tikzpicture}[scale=1.5]
			\square{%
				/square/label/.cd,
                0={$x$}, 1={$x$}, 2={$y$}, 3={$z$},
				01={}, 12={$g$}, 23={$h$}, 03={$f$}, 02={$g$}, 13={$i$}, 
				012={$=$}, 023={$\alpha$}, 013={$\alpha'$}, 123={$\simeq$},
				0123={},
				/square/labelstyle/.cd,
				012={description},
				/square/arrowstyle/.cd,
				01={equal}, 012={phantom}
			}
		\end{tikzpicture}
	\end{nscenter}
	then the \(2\)-simplex \(\rho^*\) of \(\E(x, z)\) corresponding to \(\rho\) goes from
	the 1-simplex \(\alpha^*\) of \(\E(x, z)\) corresponding to \(\alpha\) to the compposition
	of the 1-simplex \((\alpha')^*\) corresponding to \(\alpha'\) followed by an equivalence.
	Hence, in the arrow \(\infty\)-category of \(\E(x, z)\) the 3-simplex \(\rho\) induces an equivalence
	between \(\alpha^*\) and the composition of \((\alpha')^*\) with an arrow of \(\E(x, z)\) which is an equivalence.
\end{rem}

\begin{lemma}
	\label{cart 2 simp part A}
Let \(p\colon \E \to \B\) be a fibration of \(\Catoo\)-categories. 
Then a triangle \(\alpha \colon \Delta^2 \to \Nsc \E\) of the form 
\begin{nscenter}
	\begin{tikzpicture}[scale=1.2]
	\triangle{%
		/triangle/label/.cd,
		0={$x$}, 1={$y$}, 2={$z$},
		01={$f$}, 12={$g$}, 02={$h$},
		012={$\alpha$}
	}
	\end{tikzpicture}
\end{nscenter}
is left \(\Nsc(p)\)-inner if and only if the corresponding 1-simplex \(\alpha^{\ast}\colon h\to gf\) in \(\E(x,z)\) is \(p_{x,z}\)-cartesian and maps to a \(p_{x',z}\)-cartesian arrow in \(\E(x',z)\) after pre-composing with any arrow \(x'\to x\). Similarly, \(\alpha\) is right \(\Nsc(p)\)-inner if and only if the corresponding 1-simplex \(\alpha^{\ast}\colon h\to gf\) in \(\E(x,z)\) is \(p_{x,z}\)-cartesian and maps to a \(p_{x,z'}\)-cartesian arrow in \(\E(x,z')\) after post-composing with any arrow \(z\to z'\).
\end{lemma}
\begin{proof}
We prove the left inner case. The proof for right inner triangles is completely analogous.
Suppose first that \(\alpha^{\ast}\) is \(p_{x,z}\)-cartesian in \(\E(x,z)\) and maps to a \(p_{x',z}\)-cartesian arrow in \(\E(x',z)\) after pre-composing with any arrow \(x'\to x\). 
By Remark~\ref{rem:unwinding} we have to provide a solution for every lifting problem of the form
\[\begin{tikzcd}
& \Delta^{\{n-2,n-1,n\}} \ar[dl] \ar[dr, "\alpha"]\\
\Lambda^{n}_{n-1} \ar[rr] \ar[d]& &\Nsc \E \ar[d,"\Nsc(p)"]\\
\Delta^n \ar[rr]&  &\Nsc \B 
\end{tikzcd}\]
for \(n \geq 3\).
Transposing along the adjunction \(\fCs\dashv \Nsc\), this corresponds to a lifting problem of the form 
\begin{equation}
\label{t}
\begin{tikzcd}
& \fCs\Delta^{\{n-2,n-1,n\}} \ar[dl] \ar[dr, "\alpha^{\ast}"]\\
\fCs \Lambda^{n}_{n-1} \ar[rr,"f"] \ar[d]& & \E \ar[d,"p"]\\
\fCs\Delta^n \ar[rr,"g"]&  & \B
\end{tikzcd}
\end{equation} where we have committed a small abuse of language denoting by \(\alpha^{\ast}\) also the transpose map
\(\fCs\Delta^2 \to \E\) induced by \(\alpha\). As a straightforward calculation shows, the lifting problem in \eqref{t} corresponds, at the level of marked simplicial sets, to the lifting problem
\[\begin{tikzcd}
(\sqcap^{n-1,n-1}_1)^{\flat} \ar[d] \ar[r] & \E(x',z)\ar[d,"p_{x',z}"]\\
(\Box^{n-1})^{\flat}\ar[r] & \B(px',pz) \ ,
\end{tikzcd}\]
where \(x' := f(0)\).
Moreover, 
the edge \((1,\ldots,1)\times \Delta^1\) is mapped by the top horizontal map to the 
whiskering of the 1-simplex \(\alp^*\)  
by a sequence of 0-simplices connecting \(x'\) and \(x\), and it is therefore \(p_{x',z}\)-cartesian in \(\E(x',z)\). The desired lift hence exists by 
Lemma \ref{marked cube}.

We now prove the ``only if'' direction, and so we assume that \(\alp\) is left \(p\)-inner. By Lemma~\ref{lem:reduce} we may find a left-degenerate triangle \(\alp'\) such that \(\alp\) is left congruent to \(\alp'\), so that \(\alp'\) is also left \(p\)-inner by Lemma~\ref{lem:2-out-of-3}. Then \(\alp'\) has the same first and last vertex as \(\alp\) and hence determines an edge \((\alp')^*\) in \(\E(x,z)\). Furthermore, as pointed out in the previous remark the 3-simplex exhibiting \(\alp'\) as left congruent to \(\alp\)
also determines an equivalence in the arrow \(\infty\)-category of \(\E(x,z)\) between the the edge
\(\alp^*\) and the edge given by \((\alp')^*\) followed by an equivalence. Since every equivalence is \(p_{x,z}\)-cartesian and moreover the property of being \(p_{x,z}\)-cartesian is invariant under equivalences we may replace \(\alp\) by \(\alp'\), so that we may simply assume that \(\alp\) is left-degenerate. We now consider the commutative diagram of marked simplicial sets
\begin{equation}\label{eq:compare-mapping} 
\xymatrix{
\Hom^{\triangleright}_{\Nsc\E}(x,z) \ar[r]^-{\simeq}\ar[d]^{p^{\triangleright}_{x,y}} & \Hom_{\Nsc\E}(x,z)\ar[d] & \E(x,z) \ar[l]_-{\simeq}\ar[d]^{p_{x,y}} \\
\Hom^{\triangleright}_{\Nsc\B}(x,z) \ar[r]^-{\simeq} & \Hom_{\Nsc\B}(x,z) & \B(x,z) \ar[l]_-{\simeq} \ ,
}
\end{equation}
in which the horizontal maps are marked categorical equivalences and the vertical maps are fibrations between fibrant objects in the marked categorical model structure. Here, the horizontal equivalences in the right column are given by the canonical isomorphism \(\Hom_{\Nsc\E}(x,z) \cong \Un^{\sca}\E(x,z)\), see~\cite[Remark 4.2.1]{LurieGoodwillie}, and the horizontal equivalences in the left column are established in~\cite[Proposition 2.33]{GagnaHarpazLanariEquiv}. The triangle \(\alp\) determines arrows in both the top left and top right marked simplicial sets, and the images of these two arrows coincides in \(\Hom_{\Nsc\E}(x,z)\) by direct inspection. We hence obtain that the arrow \(\alp^*\) is \(p_{x,z}\)-cartesian in \(\E(x,z)\) if and only if the arrow determined by \(\alp\) in \(\Hom^{\triangleright}_{\Nsc\E}(x,z)\) is \(p^{\triangleright}_{x,z}\)-cartesian. The latter, and hence the former, indeed holds when \(\alp\) is left \(p\)-inner by Remark~\ref{rem:hom-cart-fib}. The desired implication is now a consequence of the closure of left \(p\)-inner triangles under left whiskering (see Remark~\ref{rem:closed-whiskering}).
\end{proof}

We now come to the outer counterpart of Lemma~\ref{cart 2 simp part A}.

\begin{lemma}
\label{cart 2 simp part B}
Let \(p\colon \E \to \B\) be a fibration of \(\Catoo\)-categories. 
Let \(\alpha^{\ast}\) be a triangle in \(\Nsc \E\) of the form
\begin{nscenter}
	\begin{tikzpicture}[scale=1.2]
	\triangle{%
		/triangle/label/.cd,
		0={$x$}, 1={$y$}, 2={$z$},
		01={$f$}, 12={$g$}, 02={$h$},
		012={$\alpha^{\ast}$}
	}
	\end{tikzpicture}
\end{nscenter}
Then the following holds:
\begin{enumerate}
\item
If the associated 1-simplex \(\alpha^{\ast}\colon h\to gf\) is \(p_{x,z}\)-cocartesian in \(\E(x,z)\) and \(f\colon x \to y\) is \(\Nsc(p)\)-cocartesian then \(\alp\) is left \(\Nsc(p)\)-outer.
\item
If the associated 1-simplex \(\alpha^{\ast}\colon h\to gf\) is \(p_{x,z}\)-cocartesian in \(\E(x,z)\) and \(g\colon y \to z\) is \(\Nsc(p)\)-cartesian then \(\alp\) is right \(\Nsc(p)\)-outer.
\item
If \(\alpha\) is left \(\Nsc(p)\)-outer and \(f\) is a degenerate edge 
then \(\alpha^{\ast}\) is \(p_{x,z}\)-cocartesian.
\item
If \(\alpha\) is right \(\Nsc(p)\)-outer and \(g\) is a degenerate edge 
then \(\alpha^{\ast}\) is \(p_{x,z}\)-cocartesian.
\end{enumerate}
\end{lemma}
\begin{proof}
Statements (3) and (4) follows from Remark~\ref{rem:hom-cart-fib} by using the commutative diagram~\eqref{eq:compare-mapping} as in the proof of Lemma~\ref{cart 2 simp part A}, and statements (1) and (2) imply each other by passing to opposites. We now prove (1). Suppose that \(f\) 
is \(\Nsc(p)\)-cocartesian and that \(\alpha^{\ast}\) is \(p_{x,z}\)-cocartesian in \(\E(x,z)\), and let us prove that \(\alp\) is left \(\Nsc(p)\)-outer. 
By Remark~\ref{rem:unwinding} we have to provide a solution for every lifting problem of the form
\begin{equation}
\begin{tikzcd}
& \Delta^{\{0,1,n\}} \ar[dl] \ar[dr, "\alpha^{\ast}"]\\
\Lambda^{n}_{0} \ar[rr] \ar[d]& &\Nsc \E \ar[d,"\Nsc(p)"]\\
\Delta^n \ar[rr]&  &\Nsc \B
\end{tikzcd}
\end{equation}
Transposing along the adjunction \(\fCs\dashv \Nsc\) we obtain an equivalent lifting problem of the form
\[\begin{tikzcd}
& \fCs\Delta^{\{0,1,n\}} \ar[dl] \ar[dr, "\alpha^{\ast}"]\\
\fCs\Lambda^{n}_{0} \ar[rr,"f"] \ar[d]& & \E \ar[d,"p"]\\
\fCs \Delta^n \ar[rr,"g"]&  & \B \ .
\end{tikzcd}\]
A simple combinatorial analysis shows that this amounts to compatibly solving the lifting problems determined by the back and front faces of the cube
\begin{equation}\label{eq:the-big-cube}
\begin{tikzcd}
	& \sqcap^{n-1,1}_0 \ar[dd] \ar[rr,"f_{0{,}n}"] & & \E(x,z) \ar[dd,"p_{x{,}z}"]\\
	\partial \Box^{n-2}
 \ar[ur,"-\circ \{0{,}1\}"]\ar[dd] \ar[rr,crossing over] \ar[rr, "f_{1{,}n}"{pos=0.6}] & & \E(y,z) \ar[ur,"-\circ f"] \\
	& \Box^{n-1}  \ar[rr,"g_{0{,}n}"{pos=0.35}] & & \B(px,pz) \\
	\Box^{n-2}
 \ar[ur,"-\circ \{0{,}1\}"] \ar[rr,"g_{1{,}n}"] & & \B(py,pz) \ar[ur,"-\circ p(f)"{swap}]
    \ar[from=2-3, to=4-3, crossing over, "\ p_{y{,}z}"{pos=0.65}]
\end{tikzcd}
\end{equation}
where we have identified 
\[\fCs \Del^n(0,n) = \Box^{n-1} \quad \fCs \Lambda^n_0(0,n) = \sqcap^{n-1,1}_0 \]
\[ \fCs \Delta^n(1,n) = \Box^{n-2} \quad\text{and}\quad \fCs \Lambda^n_0(1,n) = \partial\Box^{n-2}.\]
Equivalently, we may consider this as a lifting problem in the arrow category of marked simplicial sets, involving the morphism between arrows encoded by the left square against the morphism between arrows encoded by the right square. We may endow this arrow category with the projective model structure, so that cofibrations are Reedy cofibrations and fibrations are levelwise. We then find that in the lifting problem encoded by the above cube, the left arrow constitutes a cofibration between cofibrant objects, while the right arrow is a fibration between fibrant objects (by our assumption that \(p\) is a fibration of \(\Catoo\)-categories). 

Now, it is a general fact concerning model categories that the lifting property in a given square involving a cofibration between cofibrant objects against a fibration between fibrant objects is a \emph{homotopy invariant property}, that is, it does not change if one replaces the given square by a levelwise weakly equivalent one of the same nature.
In particular, in proving the existence of a lift we may as well replace the cube~\eqref{eq:the-big-cube} with a levelwise weakly equivalent one, as long as we make sure it also has the property that its left square is Reedy cofibrant and its right square is levelwise fibrant with vertical legs fibrations. We now choose to make such a modification by simply replacing the corner \(\E(y,z)\) with the fibre product 
\[\X := \E(x,z) \times_{\B(px,pz)} \B(py,pz),\] 
which is also a homotopy fibre product since the vertical legs are fibrations between fibrant objects. The map \(\E(y,z) \to \X\) is an equivalence: indeed, by Proposition~\ref{comparison enriched 1-simpl} \(f\in \E(x,y)_0\) is a cocartesian arrow, and hence the right square in~\eqref{eq:the-big-cube} is homotopy cartesian. We conclude that the new cube is levelwise equivalent to the old one, while clearly still keeping the same property of having its left square Reedy cofibrant and its right square levelwise fibrant with vertical maps fibrations. At the same time, by its construction, the data of a lift in the modified cube is the same as a lift in its back square
\[
\begin{tikzcd}
\sqcap^{n-1,1}_0 \ar[r] \ar[d] & \E(x,z) \ar[d,"p_{x,z}"]\\
\Box^{n-1} \ar[r]& \B(px,pz) \ ,
\end{tikzcd}
\]
where the edge corresponding to \(\Delta^1 \times (0,\ldots,0)\) in \(\sqcap^{n-1,1}_0\) is sent to a \(p_{x,z}\)-cocartesian edge in \(\E(x,z)\) by assumption. A solution therefore exists by Lemma \ref{marked cube}, thus concluding the proof of the proposition.
\end{proof}

\begin{prop}
\label{N detects fib inner outer} 
Let \(p\colon \E \to \B\) be a fibration of \(\Catoo\)-categories. Then \(p\) is an enriched 2-inner (resp.~2-outer) fibration if and only if \[
 \rN^{\sca}(p)\colon \rN^{\sca} \E\rightarrow\rN^{\sca} \B
\]
is a 2-inner (resp.~ 2-outer) fibration of \(\infty\)-bicategories.
\end{prop}
\begin{proof}
We first note that since \(p\) is a fibration between fibrant objects the same holds for \(\rN^{\sca}(p)\), and so the latter is always a weak fibration.
We now recall that the vertices in \(\rN^{\sca}(\E)\) correspond exactly to the objects of \(\E\), and the edges \(f\colon x \to y\) in \(\rN^{\sca}(\E)\)
correspond exactly to a pair of objects \(x,y \in \E\) and a vertex \(f\in \E(x,y)\). In addition,  triangles 
\[ 
\xymatrix{
& y \ar[dr]^{g} & \\
x\ar[ur]^{f} \ar[rr]^{h} && z
}\]
correspond bijectively to triples of objects \(x,y,z \in \E\), triples of vertices \(f \in \E(x,y),g \in \E(y,z), h \in \E(x,z)\) and an edge \(\tau\colon h \Rightarrow g \circ f\) in \(\E(x,z)\).

Considering inner lifts of triangles, it follows directly from Lemma~\ref{cart 2 simp part A} that the triangles in \(\rN^{\sca}(\B)\) admits a 
sufficient supply of left (resp.~right) \(\rN^{\sca}(p)\)-inner lifts if and only if each \(p_{x,y}\colon \E(x,y) \to \B(x,y)\) has a sufficient supply of cartesian lifts, 
and these cartesian lifts are closed under pre-composition (resp.~post-composition) in \(\E\). 
We may then conclude that \(\rN^{\sca}(p)\) is an 2-inner fibration if and only if \(\E\) is a 2-inner fibration of \(\Catoo\)-categories.

We now consider the outer case. By Corollary~\ref{cor:only-degenerate-enhanced} we may restrict attention to triangles in \(\B\) with one leg degenerate and lifts whose same leg is degenerate in \(\E\). We then deduce from Lemma~\ref{cart 2 simp part B} that the triangles in \(\rN^{\sca}(\B)\) admits a sufficient supply of left (resp.~right) \(p\)-outer lifts if and only if each \(p_{x,y}\colon \E(x,y) \to \B(x,y)\) has a sufficient supply of cocartesian lifts. Since the definition of 2-outer fibration contains explicitly the closure under left/right whiskering (which corresponds, up to equivalence, to pre/post composition), we can conclude that \(\rN^{\sca}(p)\) is a 2-outer fibration if and only if \(\E\) is an enriched 2-outer fibration of \(\Catoo\)-categories, as desired.
\end{proof}

We can now deduce our main result of interest:
\begin{proof}[Proof of Theorem~\ref{N detects fib cart}]
Our goal is to compare enriched \(2\)-inner cartesian fibrations of \(\Catoo\)-categories
with \(2\)-inner cartesian fibrations of \(\infty\)-bicategories.
The cartesian fibrational part is given by Corollary~\ref{comparison enriched car fib} while the \(2\)-inner fibrational part is dealt with by Proposition~\ref{N detects fib inner outer}.
Combining the two results the theorem is thereby proven.
\end{proof}

We finish this section by collecting a few corollaries which can be easily deduced from the comparison of Theorem~\ref{N detects fib cart}.

\begin{cor}\label{cor:left-is-right}
Let \(p\colon \E \to \B\) be a 2-inner (co)cartesian fibration. Then a triangle \(\sig\colon \Del^2 \to \E\) is left \(p\)-inner if and only if it is right \(p\)-inner.
\end{cor}
\begin{proof}
Since \(\rN^{\sca}\) is a right Quillen equivalence there exists a fibration \(q\colon \C \to \D\) of \(\Catoo\)-categories which fits in a commutative diagram of the form
\begin{equation}
\begin{tikzcd}
    \E \ar[d,"p"{swap}] \ar[r,"\simeq"] & \rN^{\sca} \C \ar[d,"\rN^{\sca}q"]\\
    \B \ar[r,"\simeq"] & \rN^{\sca} \D \ ,
\end{tikzcd}
\end{equation}
where the horizontal maps are equivalences of \(\infty\)-bicategories. Now the right vertical arrow is a bicategorical fibration, being the image under a right Quillen functor of a Dwyer--Kan fibration, while the left vertical arrow is a bicategorical fibration by Proposition~\ref{cart are fib}. Applying Proposition~\ref{prop:invariance} we now obtain that the right vertical map is a 2-inner/outer (co)cartesian fibration and that the top horizontal map preserves and detects left/right inner triangles. It will hence suffice to prove the left and right \(\rN^{\sca}(p)\)-inner triangles coincide in \(\rN^{\sca}(\C)\). By Proposition~\ref{cart 2 simp part A} this amounts to showing that for every \(x,y \in \C\) and morphism \(e\in \C(x,y)_1\) in the mapping \(\infty\)-category, the condition that \(e\) is cartesian with respect to \(q_{x,y}\colon \C(x,y) \to \D(qx,qy)\) and remains cartesian after post-composing with any morphism \(y \to z\) is equivalent to the condition that \(e\) is \(q_{x,y}\)-cartesian and remains cartesian after pre-composing with any morphism \(w \to x\). Indeed, by Theorem~\ref{N detects fib cart} we have that \(q\colon \C \to \D\) is an enriched 2-inner cartesian fibration of \(\Catoo\)-categories, and so both conditions are equivalent to \(e\) simply being \(q_{x,y}\)-cartesian.
\end{proof}

\begin{cor}
	\label{cor:2fib are 1fib}
    Let \(\E\) and \(\B\) two \(\infty\)-bicategories.
	\begin{enumerate}
	\item
A given 2-inner (co)cartesian fibration \(p\colon \E \to \B\) is equivalent to a 1-inner (co)car\-te\-sian fibration if and only if every triangle is left and right inner.
\item
A given 2-outer (co)cartesian fibration \(p\colon \E \to \B\) is equivalent to a 1-outer (co)car\-te\-sian fibration if and only if every triangle whose left leg is \(p\)-cocartesian is left outer and every triangle whose right leg is \(p\)-cartesian is right outer.
\end{enumerate}
\end{cor}
\begin{proof}
From the explicit description of Remark~\ref{rem:unwinding} one immediately finds that if \(p\) is a 1-inner (co)cartesian fibration then every triangle is both left and right \(p\)-inner. Similarly, if \(p\) is a 1-outer cartesian fibration then any (co)cartesian arrow is automatically strongly (co)cartesian by~\cite[Proposition 2.3.7]{GagnaHarpazLanariLaxLimits} (and its dual), and so 
every triangle whose left leg is \(p\)-cocartesian is left \(p\)-outer and any triangle whose right leg is \(p\)-cartesian is right \(p\)-outer.

To prove the ``if'' direction, we now invoke the fact that \(\rN^{\sca}\) is a right Quillen equivalence to deduce the existence of fibration \(q\colon \C \to \D\) of \(\Catoo\)-categories which fits in a commutative diagram of the form
\begin{equation}\label{eq:rigid}
\begin{tikzcd}
    \E \ar[d,"p"{swap}] \ar[r,"\simeq"] & \rN^{\sca} \C \ar[d,"\rN^{\sca}q"]\\
    \B \ar[r,"\simeq"] & \rN^{\sca} \D \ ,
\end{tikzcd}
\end{equation}
whose horizontal maps are equivalences of \(\infty\)-bicategories. Now the right vertical arrow is a bicategorical fibration, being the image under a right Quillen functor of a Dwyer-Kan fibration, while the left vertical arrow is a bicategorical fibration by Proposition~\ref{cart are fib}. Applying Proposition~\ref{prop:invariance} we now obtain that the right vertical map is a 2-inner/outer (co)cartesian fibration,
and hence by Theorem~\ref{N detects fib cart} the functor \(q\colon \C \to \D\) is an enriched 2-inner/outer (co)cartesian fibration of \(\Catoo\)-categories. For every \(x,y \in \E\) with images \(x',y'\in \Nsc\C\) (which we can identify with objects of \(\C\)) we may then consider the commutative diagram
\begin{equation}\label{eq:compare-mapping-2} 
\xymatrix{
\Hom^{\triangleright}_{\E}(x,y) \ar[d]\ar[r]^-{\simeq} & \Hom^{\triangleright}_{\Nsc\C}(x',y') \ar[r]^-{\simeq}\ar[d] & \Hom_{\Nsc\C}(x',y')\ar[d] & \C(x',y') \ar[l]_-{\simeq}\ar[d] \\
\Hom^{\triangleright}_{\B}(px,py) \ar[r]^-{\simeq} & \Hom^{\triangleright}_{\Nsc\D}(qx',qy') \ar[r]^-{\simeq} & \Hom_{\Nsc\D}(qx',qy') & \D(qx',qy') \ar[l]_-{\simeq} \ ,
}
\end{equation}
in which all vertical arrows are cartesian fibrations in the inner case and cocartesian fibrations in the outer case. Our assumption implies in particular that every left-degenerate triangle in \(\E\) is left \(p\)-inner/outer, which by Remark~\ref{rem:hom-cart-fib} implies that the left-most vertical map in~\eqref{eq:compare-mapping-2} is a left fibration in the outer case and a right fibration in the inner case. The same consequently holds for all vertical maps in~\eqref{eq:compare-mapping-2}, which implies that \(q\colon \C \to \D\) is an enriched 1-inner/outer (co)cartesian fibration of \(\Catoo\)-categories. By~\cite[Proposition 3.1.3]{GagnaHarpazLanariLaxLimits} this means that \(\Nsc q\colon \Nsc\C \to \Nsc\D\) is a 1-inner/outer (co)cartesian fibration. Since the square~\eqref{eq:rigid} is an equivalence between two bicategorical fibrations between fibrant objects, these two fibrations satisfy the same right lifting properties. Applying this to the right lifting properties of Definition~\ref{def:inner-outer}, we conclude that if the right one is a 1-inner fibration then so is the left, and if the right one is a 1-outer fibration then so is the left. We may consequently deduce that \(p\colon \E \to \B\) is a 1-inner/outer (co)cartesian fibration, as desired.
\end{proof}

\bibliographystyle{amsplain}
\bibliography{biblio}
\end{document}